\documentclass[11pt,reqno]{amsart}
\usepackage[T1]{fontenc}
\usepackage[utf8]{inputenc}
\usepackage{lmodern}
\usepackage{microtype}
\usepackage{amssymb,amsmath, amsthm, amsfonts}

\usepackage[dvipsnames]{xcolor}
\usepackage{graphicx}
\usepackage{listings}

\usepackage[normalem]{ulem}

\usepackage[margin=1in]{geometry}
\usepackage{lstautogobble}
\usepackage{enumerate}
\usepackage[shortlabels]{enumitem}
\usepackage{thmtools}
\usepackage{thm-restate}
\usepackage{amsthm}
\usepackage{verbatim}
\usepackage{accents}
\usepackage{mathtools}
\usepackage{physics}

\usepackage[colorlinks=true]{hyperref}
\hypersetup{
  colorlinks=true,
  linkcolor=blue!45!black,
  citecolor=magenta,
  urlcolor=blue!45!black,
  pdftitle={Sharp mean-field estimates for diffusive log/Riesz gases in the Hilbert--Schmidt regime},
  pdfauthor={Matias G. Delgadino, Rishabh Gvalani, Matthew Rosenzweig}
}
\usepackage[nameinlink,capitalise,noabbrev]{cleveref}
\usepackage[bottom]{footmisc}
\crefformat{equation}{(#2#1#3)}

\usepackage{float}
\restylefloat{table}

\usepackage{mathrsfs}
\setlist{
  listparindent=\parindent,
  parsep=0pt,
}
\setlist[itemize]{topsep=4pt,itemsep=2pt,parsep=1pt}
\setlist[enumerate]{topsep=4pt,itemsep=3pt,parsep=1pt}
\allowdisplaybreaks

\theoremstyle{plain}
\newtheorem{thm}{Theorem}[section]
\newtheorem{prop}[thm]{Proposition}
\newtheorem{lemma}[thm]{Lemma}
\newtheorem{cor}[thm]{Corollary}

\theoremstyle{definition}

\newtheorem{remark}[thm]{Remark}

\newtheorem{conj}[thm]{Conjecture}

\Crefname{thm}{Theorem}{Theorems}
\Crefname{prop}{Proposition}{Propositions}
\Crefname{conj}{Conjecture}{Conjectures}
\Crefname{assumption}{Assumption}{Assumptions}

\numberwithin{equation}{section} 

\DeclarePairedDelimiter{\brak}{\lbrack}{\rbrack}
\DeclarePairedDelimiter{\pa}{\lparen}{\rparen}
\DeclarePairedDelimiter{\brac}{\lbrace}{\rbrace}

\DeclareMathOperator{\supp}{supp}

\DeclareMathOperator{\osc}{osc}

\renewcommand{\det}{\mathrm{det}}

\renewcommand{\d}{\mathsf{d}}
\renewcommand{\dd}{\mathop{}\!\mathrm d}

\newcommand{\R}{{\mathbb{R}}}

\newcommand{\N}{{\mathbb{N}}}
\newcommand{\Q}{{\mathbb{Q}}}
\newcommand{\Z}{{\mathbb{Z}}}
\newcommand{\K}{{\mathsf{K}}}

\newcommand{\T}{{\mathbb{T}}}
\newcommand{\g}{{\mathsf{g}}}

\renewcommand{\k}{\mathsf{k}}

\newcommand{\nab}{\nabla}

\newcommand{\tl}{\tilde}

\newcommand{\nn}{\nonumber}

\newcommand{\ol}{\overline}

\newcommand{\XN}{X_N}

\newcommand{\vep}{\varepsilon}

\newcommand{\be}{\beta}

\newcommand{\Tc}{\mathcal{T}}

\newcommand{\indic}{\mathbf{1}}
\newcommand{\f}{\mathsf{f}}
\newcommand{\Fr}{\mathsf{F}}
\newcommand{\Hr}{\mathsf{H}}

\newcommand{\Er}{\mathsf{E}}

 \newcommand{\iu}{\mathrm{i}\mkern1mu}

\newcommand{\Uprop}{\mathsf{U}}

\newcommand{\E}{{\mathbb{E}}}

\newcommand{\wh}{\widehat}

\newcommand{\s}{\mathsf{s}}

\renewcommand{\P}{\mathcal{P}}

\let\div\relax
\DeclareMathOperator{\div}{\mathrm{div}}

\def\XXint#1#2#3{{\setbox0=\hbox{$#1{#2#3}{\int}$ }
\vcenter{\hbox{$#2#3$ }}\kern-.6\wd0}}

\let\oldtocsection=\tocsection

\let\oldtocsubsection=\tocsubsection

\let\oldtocsubsubsection=\tocsubsubsection

\renewcommand{\tocsection}[2]{\hspace{0em}\oldtocsection{#1}{#2}}
\renewcommand{\tocsubsection}[2]{\hspace{1em}\oldtocsubsection{#1}{#2}}
\renewcommand{\tocsubsubsection}[2]{\hspace{2em}\oldtocsubsubsection{#1}{#2}}

\newcommand{\step}[1]{\par\medskip\noindent\textbf{#1}}

\title[Sharp mean-field estimates for diffusive log/Riesz gases]{Sharp mean-field estimates for diffusive log/Riesz gases in the Hilbert--Schmidt regime}

\author[M.G. Delgadino]{Matias G. Delgadino}
\address{Matias Delgadino, Department of Mathematics, The University of Texas at Austin, PMA 8.100, 2515 Speedway, Austin, TX 78712, USA}
\email{matias.delgadino@math.utexas.edu}
\urladdr{https://math.utexas.edu/directory/matias-delgadino}
\author[R. Gvalani]{Rishabh Gvalani}
\address{Rishabh Gvalani, School of Mathematics, University of Edinburgh, James Clerk Maxwell Building, Peter Guthrie Tait Road, Edinburgh EH9 3FD, United Kingdom}
\email{rishabh.gvalani@ed.ac.uk}
\urladdr{https://webhomes.maths.ed.ac.uk/~rgvalani/}
\author[M. Rosenzweig]{Matthew Rosenzweig}
\address{Matthew Rosenzweig, Department of Mathematical Sciences, Carnegie Mellon University, 7127 Wean Hall, 5000 Forbes Avenue, Pittsburgh, PA 15213, USA}
\email{mrosenz2@andrew.cmu.edu}
\urladdr{https://matthewrosenzweigwork-max.github.io/}
\thanks{M.G.D. was supported by NSF grant DMS-2205937. RG is partially funded by the Deutsche Forschungsgemeinschaft (DFG,
German Research Foundation) - SPP 2410 Hyperbolic Balance Laws in Fluid
Mechanics: Complexity, Scales, Randomness (CoScaRa). M.R. was supported by NSF grants DMS-2441170, DMS-2345533, DMS-2342349.}

\begin{document}
\raggedbottom

\begin{abstract}
\begingroup
We study fixed-temperature logarithmic and Riesz gases after subtracting the leading mean-field contribution relative to a prescribed background law.  For repulsive interactions in the Hilbert--Schmidt regime, we prove $N$-uniform bounds and quantitative convergence of the resulting modulated partition function to a normalization expressed by the Carleman--Fredholm determinant of the centered interaction operator.  We show that the Hilbert--Schmidt threshold is sharp and obtain explicit lower bounds on the rate of divergence at and above it; these rates are expected to be nonoptimal.  The proof combines positive-definite truncations and a low/high-frequency decomposition with exponential inequalities and Gaussian-chaos asymptotics for canonical degree-two $U$-statistics.  As consequences, we establish entropic commutator estimates with the sharp $O(N^{-1})$ additive scale in the modulated-free-energy method, a static joint linear-statistics central limit theorem, and a dynamical central limit theorem for joint linear statistics at finitely many times.  This extends the logarithmic partition-function estimates of the first two authors to the full Riesz Hilbert--Schmidt range and identifies the limiting determinant normalization.  For the attractive logarithmic interaction at sufficiently small inverse temperature, we also prove analogous results.
\endgroup
\end{abstract}

\maketitle
\enlargethispage{2pt}

\section{Introduction}
\label{sec:intro}

\begingroup

For a symmetric two-body interaction potential $W$ and one-body external potential $V$, consider the mean-field Hamiltonian
\begin{align}\label{eq:intro-general-Hamiltonian}
  \mathcal H_N(\XN)
  \coloneqq
  \frac{1}{2N}\sum_{1\le i\ne j\le N} W(x_i,x_j)
  + \sum_{i=1}^N V(x_i),
  \qquad
  \XN:=(x_1,\ldots,x_N).
\end{align}
The associated canonical Gibbs ensemble is
\begin{align}\label{eq:intro-general-Gibbs}
  \dd\mathbb P_{N,\be}(\XN)
  = \frac{1}{Z_{N,\be}}
  e^{-\be \mathcal H_N(\XN)}\,\dd\XN .
\end{align}
Here the normalizing constant $Z_{N,\be}$ is the partition function.  We work in the fixed-temperature mean-field scaling, in which $\be$ is independent of $N$; we refer to this as the diffusive temperature regime.  Throughout the paper, unless explicitly indicated otherwise, the position variables belong to the whole space $\R^\d$; the flat-torus and kinetic variants are recorded in \cref{rem:torus-kinetic-variants}.

The basic question of this paper is to identify the normalization which remains after the leading-order macroscopic mean-field contribution has been separated relative to a prescribed product background.  Physically, the separation fixes the macroscopic density profile and asks what free-energy correction is left from the correlated microscopic fluctuations.  A bounded residual normalization means that, after the mean-field cost has been removed, the singular interaction contributes only an $O(1)$ fluctuation factor.  Divergence of this residual factor is the ultraviolet obstruction at the Hilbert--Schmidt threshold.  Mathematically, this is the normalization needed to pass between product and modulated Gibbs references in the modulated-free-energy method \cite{RS2023lsi}.  Controlling this normalization also allows us to prove averaged versions of the commutator estimates, deterministic pointwise versions of which have been widely used in the mean-field and fluctuation theory of log/Coulomb/Riesz gases; see Serfaty's lecture notes \cite{SerfatyLN} and the third author's recent survey \cite{Rosenzweig2026Commutators}.

For repulsive logarithmic and Riesz interactions in the Hilbert--Schmidt regime, we prove sharp fixed-temperature partition-function estimates, including a quantitative convergence rate to the limiting determinant normalization and the sharpness of the square-integrability threshold.  We then use these bounds as exponential-moment inputs for entropic commutator estimates governing variations of the modulated energy.  Finally, we derive sharp mean-field estimates for interacting diffusions, a static central limit theorem for joint linear statistics under the modulated Gibbs law, and a dynamical central limit theorem for joint linear statistics at finitely many times.  For the attractive logarithmic interaction at sufficiently small inverse temperature, we also prove analogous results.  These results are used in the authors' companion works {\cite{DGRSizeOfChaos2026,DelgadinoGvalaniStein2026}} to obtain sharp two-sided bounds for the size of chaos for diffusive log/Riesz gas measures in several statistical divergences up to the optimal Hilbert--Schmidt threshold and to prove sharp quantitative CLTs for fluctuations around the mean-field limit.

The precise statements, their relation to previous work, and the proof strategy are given in \cref{sec:main-results-proof-strategy}.

\endgroup

\section{Main results and proof strategy}
\label{sec:main-results-proof-strategy}

\begingroup

\subsection{Modulated Gibbs normalization for log/Riesz gases}
\label{subsec:intro-background-splitting}
\label{subsec:riesz-modulated-energies-partition-functions}
\label{subsec:log-riesz-specialization}
We first isolate the modulated normalization relative to a prescribed background law, then identify its Hilbert--Schmidt determinant limit and specialize the construction to logarithmic and Riesz interactions.

\medskip\noindent\textbf{Splitting relative to a background law.}

Fix a probability measure $\mu$ on $\R^\d$.  When $\mu$ is absolutely continuous with strictly positive density, write $\dd\mu=Z_\mu^{-1}e^{-U}\,\dd x$, with $U=-\log\mu$ understood up to an additive constant.  For integers $N\ge1$, write $[N]\coloneqq\{1,\ldots,N\}$.  Set
\begin{align}\label{eq:hmuImuDef}
  h_\mu(x) \coloneqq \int_{\R^\d} W(x,y)\,\dd\mu(y),
  \qquad
  I_\mu \coloneqq \frac12\iint_{(\R^\d)^2} W(x,y)\,\dd\mu(x)\,\dd\mu(y).
\end{align}
We use the notation
\begin{align}\label{eq:tildedef}
\tl{f}(x,y) \coloneqq f(x,y) - \int_{\R^\d} f(x,z)\,\dd\mu(z) - \int_{\R^\d} f(z,y)\,\dd\mu(z)+ \int_{(\R^\d)^2} f(z,w)\,\dd\mu(z)\,\dd\mu(w)
\end{align}
for the Hoeffding centering of a two-body kernel.  The modulated energy relative to $\mu$ is characterized by the identity\footnote{This identity is often called the splitting formula (relative to thermal equilibrium) in the log/Coulomb/Riesz literature \cite[Chapter~5, Section~5.1]{SerfatyLN}.  It is a special case of the Hoeffding decomposition of a $U$-statistic.}
\begin{align}\label{eq:intro-modulated-splitting}
  \frac{1}{2N}\sum_{1\le i\ne j\le N} W(x_i,x_j)
  = N\Fr_N(\XN,\mu) + \sum_{i=1}^N h_\mu(x_i) - N I_\mu .
\end{align}
Equivalently,
\begin{align}\label{eq:FNdef}
\Fr_N(\XN,\mu) = \frac{1}{2N^2}\sum_{1\le i\ne j\le N}\tl{W}(x_i,x_j) - \frac{1}{N^2}\sum_{i=1}^Nh_\mu(x_i)+\frac{1}{N}I_\mu.
\end{align}
Consequently,
\begin{align}\label{eq:intro-Z-to-K-reduction}
  Z_{N,\be}
  = Z_\mu^N e^{\be N I_\mu}
  \E_{\mu^{\otimes N}}\brak*{\exp\left\{
    -\be N\Fr_N(\XN,\mu)
    - \sum_{i=1}^N \bigl(\be V+\be h_\mu-U\bigr)(x_i)
  \right\}}.
\end{align}
If the one-body term is absorbed into the reference law, equivalently if $U=\be(V+h_\mu)$ up to an additive constant, the expectation in \eqref{eq:intro-Z-to-K-reduction} reduces to the modulated partition function
\begin{align}\label{eq:intro-K-main}
  \K_{N,\be}(\mu)
  \coloneqq
  \E_{\mu^{\otimes N}}\brak*{\exp\{-\be N\Fr_N(\XN,\mu)\}}.
\end{align}
The condition $U=\be(V+h_\mu)$ is precisely the Euler--Lagrange relation satisfied by the thermal equilibrium measure introduced below.  Conversely, for $\be>0$, \eqref{eq:intro-K-main} is an ordinary Gibbs partition function with one-body potential $\be^{-1}U-h_\mu$, up to the deterministic factor in \eqref{eq:intro-Z-to-K-reduction}.  Thus, the adjective ``modulated'' refers to the centered normal form obtained after subtracting the mean-field contribution.

Associated to a configuration $\XN=(x_1,\ldots,x_N)$, denote the empirical measure by
\begin{align}\label{eq:muNdef}
\mu_N\coloneqq \frac1N\sum_{i=1}^N\delta_{x_i}.
\end{align}
The modulated Gibbs measure is
\begin{align}\label{eq:intro-Q-def}
 \dd\Q_{N,\be}(\mu)(\XN)\coloneqq \frac{1}{\K_{N,\be}(\mu)}e^{-\be N\Fr_N(\XN,\mu)}\,\dd\mu^{\otimes N}(\XN).
\end{align}
When $\mu=\mu_\be$ is the thermal equilibrium measure (see, e.g., \cite[Section~2.5, especially Proposition~2.19]{SerfatyLN}), characterized by the Euler--Lagrange relation for the macroscopic/mean-field free energy
\begin{equation}
  \dd\mu_\be = Z_{\mu_\be}^{-1}e^{-\be(V+h_{\mu_\be})}\,\dd x,
\end{equation}
the condition $U=\be(V+h_\mu)$ above is exactly satisfied.  Hence, the modulated Gibbs measure coincides with the original canonical ensemble:
$\Q_{N,\be}(\mu_\be)=\mathbb P_{N,\be}$.
If $\be>0$, the associated modulated free energy
\begin{align}\label{eq:ENdef}
\Er_N(f_N,\mu) \coloneqq \frac{1}{\be N}\Hr(f_N\vert \mu^{\otimes N}) + \E_{\XN\sim f_N}\brak[\big]{\Fr_N(\XN,\mu)},
\qquad f_N\in\P((\R^\d)^N),
\end{align}
satisfies
\begin{align}\label{eq:intro-entropy-rewrite}
  \Er_N(f_N,\mu)
  = \frac{1}{\be N}\left(
    \Hr(f_N\vert \Q_{N,\be}(\mu))-\log \K_{N,\be}(\mu)
  \right).
\end{align}
Jensen's inequality and \eqref{eq:FNdef} give the order-one lower bound
\begin{align}\label{eq:KlowerJensen}
\K_{N,\be}(\mu)\ge \exp\Big(-\be N\E_{\mu^{\otimes N}}[\Fr_N(\XN,\mu)]\Big)=\exp\big(\be I_\mu\big),
\end{align}
provided that $\mu$ has finite energy.  The three questions of interest to us are boundedness of \eqref{eq:intro-K-main} uniformly in $N$, identification of its limit, and sharpness of the assumptions under which such a bound is possible.

\medskip\noindent\textbf{Hilbert--Schmidt centered kernels and limiting normalization.}

The limiting normalization is governed by the quadratic fluctuation operator associated with the Hoeffding-centered interaction.  If $\tl{W}\in L^2(\mu^{\otimes2})$, then $\tl{W}$ defines a Hilbert--Schmidt operator on $L^2(\mu)$, which we denote by $T_\mu:L^2(\mu)\to L^2(\mu)$:
\begin{align}\label{eq:TmuDef}
T_\mu f(x)\coloneqq \int_{\R^\d} \tl{W}(x,y)f(y)\,\dd\mu(y).
\end{align}

We write $(\lambda_j)_{j\ge1}$ for the nonzero eigenvalues of $T_\mu$, repeated with algebraic multiplicity.  For $\be\ge0$, define
\begin{align}\label{eq:detTwoDef}
\det_2(\mathrm{Id}+\be T_\mu)
\coloneqq \prod_{j\ge1}(1+\be\lambda_j)e^{-\be\lambda_j}.
\end{align}
This is the Carleman--Fredholm, or regularized, determinant; see, for example, Simon \cite[Chapter~9]{Simon2005}.  Unlike the ordinary Fredholm determinant $\det(\mathrm{Id}+\be T_\mu)=\prod_j(1+\be\lambda_j)$, which requires $T_\mu$ to be trace class, $\det_2$ is well-defined for Hilbert--Schmidt operators.  The subscript $2$ records the second-order regularization: the linear trace contribution is subtracted.  In the repulsive positive-definite cases below, $T_\mu$ is nonnegative on the centered subspace, so the real logarithm is unambiguous and
\begin{align}\label{eq:detTwoLogDef}
\log\det_2(\mathrm{Id}+\be T_\mu)
=\sum_{j\ge1}\pa*{\log(1+\be\lambda_j)-\be\lambda_j}.
\end{align}
In these positive-definite cases, we set
\begin{align}\label{eq:KinftyDef}
\K_{\infty,\be}(\mu)\coloneqq e^{\be I_\mu}\det_2(\mathrm{Id}+\be T_\mu)^{-1/2}.
\end{align}
At the product level, the centered two-body term is the canonical statistic
\begin{align}\label{eq:intro-U-statistic}
W_N\coloneqq \frac1N\sum_{1\le i\ne j\le N}\tl{W}(\xi_i,\xi_j),
\qquad \xi_1,\ldots,\xi_N\stackrel{\mathrm{iid}}{\sim}\mu .
\end{align}
This is the fluctuation variable whose Laplace transform is exactly the centered two-body factor in the modulated partition function.  The origin of the determinant is transparent from the Gaussian-chaos limit.  Writing $(Z_j)_{j\ge1}$ for independent standard Gaussians, the series $\sum_j\lambda_j(Z_j^2-1)$ converges in $L^2$ and
\begin{align}\label{eq:intro-det2-heuristic}
\E\brak*{\exp\left\{-\frac{\be}{2}\sum_j\lambda_j(Z_j^2-1)\right\}}
=\prod_j e^{\be\lambda_j/2}(1+\be\lambda_j)^{-1/2}
=\det_2(\mathrm{Id}+\be T_\mu)^{-1/2}.
\end{align}
For a finite-rank kernel, the corresponding limit of $W_N$ follows immediately from the multivariate central limit theorem; the general Hilbert--Schmidt case is obtained by spectral approximation.  The main theorem below proves this normalization for the singular interactions to which we now specialize.

\medskip\noindent\textbf{Logarithmic and Riesz interactions.}

For the log/Riesz specialization, we take \(W(x,y)=\g(x-y)\), where
\begin{align}\label{eq:gDef}
  \g(x)=
  \begin{cases}
    -\log |x|, & \s=0,\\
    \frac1\s |x|^{-\s}, & 0<\s<\d.
  \end{cases}
\end{align}
The objects $h_\mu$, $I_\mu$, $\Fr_N$, $\K_{N,\be}$, $\Q_{N,\be}$, and $T_\mu$ introduced above are henceforth understood for this interaction.  With an abuse of notation, we use the same symbol for an absolutely continuous measure and for its Lebesgue density; thus statements such as $\mu\in L^p$, $\|\mu\|_{L^p}$, and pointwise bounds on $\mu$ refer to this density.

Up to a normalizing constant, these interactions are characterized as fundamental solutions of the fractional Laplacian, $(-\Delta)^{\frac{\d-\s}{2}}\g=c_{\d,\s}\delta_0$.  We restrict to the \emph{potential} regime $\s<\d$, in which $\g$ is locally integrable, thereby excluding the \emph{hypersingular} case $\s\ge\d$, studied for instance in \cite{HSST2020,hardin_large_2018}.  The case $\s=\d-2$ corresponds to the classical \emph{Coulomb} potential from electrostatics/gravitation.  Based on this threshold, it is convenient to call the interaction \emph{sub-Coulomb} if $\s<\d-2$ and \emph{super-Coulomb} if $\s>\d-2$.  More generally, interactions of the form \eqref{eq:gDef} play an important role in physics \cite{dauxois_dynamics_2002}, approximation theory \cite{borodachov_discrete_2019}, and machine learning \cite{altekruger_neural_2023,hertrich_generative_2023,hertrich_wasserstein_2023,hagemann_posterior_2023}.  In the singular regime considered below, these same kernels are also central in the analysis of Coulomb/Riesz gases and singular mean-field limits.  For this interaction, the modulated energy $\Fr_N$ originated in the study of the statistical mechanics of Coulomb/Riesz gases \cite{SandierSerfaty2015,RougerieSerfaty2016,PetracheSerfaty2017,LebleSerfaty2017,Serfaty2023} and later was used in the derivation of mean-field dynamics \cite{Duerinckx2016,Serfaty2020,NRS2021} and following works; see \cite[Chapter~4]{SerfatyLN} for a comprehensive discussion of the modulated energy. For a nonsingular positive-definite kernel, the quadratic form $\iint_{(\R^\d)^2} \g(x-y)\,\dd(\nu-\mu)(x)\,\dd(\nu-\mu)(y)$ is the squared maximum mean discrepancy (MMD) between $\nu$ and $\mu$ \cite{GrettonBorgwardtRaschScholkopfSmola2012,MuandetFukumizuSriperumbudurScholkopf2017}.  Thus, $\Fr_N$ may be viewed, up to the conventional factor $1/2$, as the diagonal-renormalized log/Riesz analogue of this MMD energy with $\nu=\mu_N$.

The positive definiteness of the repulsive kernels implies that $T_\mu$ is nonnegative on the centered subspace.  For $0<\s<\d/2$ and $\mu\in L^\infty$, the local singularity satisfies $\g\in L^2_{\mathrm{loc}}$ and hence $\tl\g\in L^2(\mu^{\otimes2})$; the logarithmic singularity is locally square-integrable under the logarithmic integrability assumptions stated below.  Thus, the range $\s<\d/2$ is precisely the Hilbert--Schmidt range for the Riesz kernels.

\begin{remark}[Torus and kinetic variants]\label{rem:torus-kinetic-variants}
All of the results below also hold on the flat torus $\T^\d$ where \eqref{eq:gDef} is replaced by the periodic (spectral) logarithmic/Riesz potential.  We do not state the torus case separately because the proofs are unchanged after these notational substitutions.

The static spatial results also have an immediate kinetic variant.  If the starting Hamiltonian $\mathcal H_N$ is augmented by the standard quadratic kinetic energy, then the velocity variables tensorize as an independent Maxwellian factor.  Consequently, the spatial modulated partition-function and entropic commutator estimates carry over verbatim for spatial observables.  The diffusive mean-field convergence application below is specific to the overdamped dynamics and is not part of this kinetic extension.  For a kinetic mean-field convergence result that exploits precisely this Maxwellian factorization, see the work of Duerinckx and Jabin \cite{DuerinckxJabin2026} discussed in \cref{subsec:intro-related-work}.
\end{remark}

\subsection{Partition-function estimates}
\label{subsec:main-results}
\label{subsec:partition-function-estimates}

Our first set of results controls the normalization of the modulated Gibbs measure.  The main theorem gives the fixed-temperature uniform bound and the regularized-determinant limit.  The Hilbert--Schmidt threshold is sharp, as quantified after the theorem.  The convergence rate in \eqref{eq:partitionLogFallbackStatement} is likewise not expected to be optimal; see \cref{rem:rateNotOptimal}.

\begin{thm}[Repulsive logarithmic/Riesz case]\label{thm:partitionFunctionBound}
Fix $\be\ge0$ and assume $0\le\s<\d/2$ and $\mu\in L^\infty\cap\P(\R^\d)$; if $\s=0$, assume in addition
\begin{align}\label{eq:logEnergyAssumption}
\int_{(\R^\d)^2}|\log |x-y||\,\dd\mu^{\otimes2}(x,y)<\infty.
\end{align}
\begin{enumerate}[(i)]
\item\emph{(Uniform bound.)}  There is a finite constant $C_{\mathrm{ub}}(\be,\mu)$ such that
\begin{align}\label{eq:partitionUniformBoundStatement}
\sup_{N\ge1}\K_{N,\be}(\mu)\le C_{\mathrm{ub}}(\be,\mu).
\end{align}
\item\emph{(Quantitative convergence.)}  If $\s=0$, assume moreover that
\begin{align}\label{eq:logRateExtraAssumptions}
\int_{\R^\d}e^{-4\be h_\mu}\,\dd\mu<\infty,
\qquad \tl\g\in L^2(\mu^{\otimes2}).
\end{align}
Then there are constants $C_{\mathrm{rate}}(\be,\mu)<\infty$ and, in the logarithmic case, $c_{\log}(\be,\mu)>0$ such that, for all $N\ge1$,
\begin{align}\label{eq:partitionLogFallbackStatement}
\Big|\K_{N,\be}(\mu)-\K_{\infty,\be}(\mu)\Big|
\le C_{\mathrm{rate}}(\be,\mu)
\begin{cases}
(\log(e+N))^{-\frac{\d-2\s}{2\s}}, & 0<\s<\d/2,\\[0.4em]
(e+N)^{-c_{\log}(\be,\mu)}, & \s=0.
\end{cases}
\end{align}
\end{enumerate}
\end{thm}

The threshold sharpness is quantitative.  For every $\be>0$ and $\d/2\le\s<\d$, \cref{prop:partition-function-sharpness} gives constants $c,C>0$ such that, for all sufficiently large $N$,
\begin{align}\label{eq:introQuantitativeSharpness}
{
\begin{aligned}
\K_{N,\be}(\mu)&\ge C^{-1}(\log N)^c, && \s=\d/2,\\
\log\K_{N,\be}(\mu)&\ge c(\log N)^{2-\d/\s}-C, && \d/2<\s<\d.
\end{aligned}}
\end{align}
These explicit lower rates are not expected to be optimal; see \cref{rem:expectedSharpUltravioletGrowth}.

\begingroup
For $\be>0$, define the macroscopic free energy
\begin{align}\label{eq:macroscopicFreeEnergy}
\mathcal F_{\be,V}(\mu)
\coloneqq
\int_{\R^\d}V\,\dd\mu+I_\mu
+\frac1\be\int_{\R^\d}\log\mu\,\dd\mu .
\end{align}
If $\mu_\be$ satisfies the thermal Euler--Lagrange relation, then
\begin{align}
\int_{\R^\d}\log\mu_\be\,\dd\mu_\be
=-\log Z_{\mu_\be}
-\be\int_{\R^\d}V\,\dd\mu_\be
-2\be I_{\mu_\be},
\end{align}
so $-\be\mathcal F_{\be,V}(\mu_\be)=\log Z_{\mu_\be}+\be I_{\mu_\be}$.  Substitution into the exact reduction \eqref{eq:intro-Z-to-K-reduction}, followed by \cref{thm:partitionFunctionBound}(ii), gives the following equilibrium consequence.

\begin{cor}[Equilibrium free-energy correction]\label{cor:equilibriumFreeEnergyCorrection}
Fix $\be>0$, and let $\mu_\be$ be an absolutely continuous thermal equilibrium measure satisfying
\begin{align}
 \dd\mu_\be=Z_{\mu_\be}^{-1}e^{-\be(V+h_{\mu_\be})}\,\dd x
\end{align}
and the hypotheses of \cref{thm:partitionFunctionBound}(ii).  Then the canonical partition function satisfies the exact identity
\begin{align}\label{eq:equilibriumCanonicalSplitting}
 Z_{N,\be}
 =e^{-\be N\mathcal F_{\be,V}(\mu_\be)}
 \K_{N,\be}(\mu_\be),
\end{align}
and consequently
\begin{align}\label{eq:equilibriumFreeEnergyExpansion}
\log Z_{N,\be}
={}&-\be N\mathcal F_{\be,V}(\mu_\be)
+\be I_{\mu_\be}
-\frac12\log\det_2(\mathrm{Id}+\be T_{\mu_\be})
+O_{\be,\mu_\be}(r_{N,\s}),
\end{align}
where
\begin{align}
 r_{N,\s}
 \coloneqq
 \begin{cases}
 (\log(e+N))^{-\frac{\d-2\s}{2\s}},&0<\s<\d/2,\\[0.3em]
 (e+N)^{-c_{\log}(\be,\mu_\be)},&\s=0.
 \end{cases}
\end{align}
\end{cor}

The entropy rewriting \eqref{eq:intro-entropy-rewrite} also gives the exact variational identity
\begin{align}
\inf_{f_N\in\P((\R^\d)^N)}\Er_N(f_N,\mu)
=-\frac1{\be N}\log\K_{N,\be}(\mu),
\end{align}
with the infimum attained by $\Q_{N,\be}(\mu)$.  Thus, the determinant limit identifies the coefficient of the order-$N^{-1}$ modulated-free-energy floor appearing in the diffusive application.
\endgroup

\begin{remark}[On the convergence rate]\label{rem:rateNotOptimal}
The emphasis in \cref{thm:partitionFunctionBound}(ii) is on the convergence itself, valid up to the sharp threshold $\s<\d/2$; the rate in \eqref{eq:partitionLogFallbackStatement} is the one the method yields without further effort, and we have not sought to optimize it.  We expect the optimal rate to be of order $1/N$ in the logarithmic case ---attainable by a tailored version of Stein's method--- and to be $\s$-dependent in the range $0<\s<\d/2$.  This is the subject of ongoing work by the authors.
\end{remark}

\begin{remark}[On the exponential integrability of $h_\mu$]\label{rem:oneBodyExpIntegrability}
Since $\mu\in L^\infty$, the potential $h_\mu$ is always bounded above, so the first condition in \eqref{eq:logRateExtraAssumptions} constrains only its lower tail, that is, the decay of $\mu$.
\end{remark}

For later use in the repulsive logarithmic/Riesz case, fix an admissible uniform partition-function constant
\begin{align}\label{eq:CpfDef}
C_{\mathrm{pf}}(\be,\mu)\coloneqq \max\{e,C_{\mathrm{ub}}(\be,\mu)\}.
\end{align}
The dependence of admissible choices of $C_{\mathrm{ub}}$, $C_{\mathrm{pf}}$, the attractive constants $\be_0$ and $C_{\mathrm{ub}}^{\mathrm{att}}$, and the later commutator constants is summarized in \cref{subsec:quantitative-bookkeeping}.

The same fixed-temperature estimate also interpolates with the sharp pointwise lower bound for the modulated energy as the inverse temperature grows with $N$.  The precise statement, including the specialization to power-law temperature scales and the microscopic-endpoint comparison, is recorded in \cref{cor:partitionFunctionTemperatureInterpolation} and proved in \cref{subsec:partition-function-temperature-interpolation}.

Next, we turn to the attractive logarithmic case, corresponding to replacing $\g$ with $\g_{\mathrm{att}}\coloneqq\log|x|$.  We complete the partition-function discussion with a high-temperature, or small-inverse-temperature, uniform bound and a quantitative determinant limit, proved in \cref{sec:attractive-log-small-beta}; the superscript $\mathrm{att}$ denotes the corresponding quantities defined above with $\g$ replaced by $\g_{\mathrm{att}}$.  The restriction to the logarithmic kernel is essential.  For the attractive Riesz sign, i.e.\ the kernel $-\g$ with $0<\s<\d$, the modulated partition function is infinite for every $\be>0$ and every $N\ge2$ whenever $\mu$ has positive density on some open set, since the near-diagonal factor $e^{c|x-y|^{-\s}}$, $c>0$, is not locally integrable.  The attractive logarithmic weight is instead only polynomial, of order $|x-y|^{-c\be}$ near the diagonal, and this is precisely the integrability that the smallness $\be<\be_0$ exploits.

\begingroup
\begin{thm}[Attractive logarithmic case, small $\be$]\label{thm:attractive-log-small-beta-partition}
Let $\mu\in\P(\R^\d)$ be such that $\g_{\mathrm{att}}\in L^1(\mu^{\otimes2})$, and assume that there is a constant $C_{\mathrm{mom}}<\infty$ such that
\begin{align}\label{eq:attMomentAssumption}
\sup_{y\in\R^\d}\int_{\R^\d}|\tl{\g}_{\mathrm{att}}(x,y)|^p\,\dd\mu(x)
\le C_{\mathrm{mom}}^p\,p!
\qquad\text{for every }p\ge1.
\end{align}
Then there is a universal constant $C_\ast<\infty$, the constant of the moment bound \eqref{eq:attMomentBound} in the proof, such that, setting
\begin{align}\label{eq:attBeta0Def}
\be_0\coloneqq\frac{1}{C_\ast C_{\mathrm{mom}}},
\end{align}
for every fixed $0\le\be<\be_0$ with $e^{2\be h_\mu^{\mathrm{att}}}\in L^1(\mu)$, there is a finite constant $C_{\mathrm{ub}}^{\mathrm{att}}(\be,\mu)$ satisfying
\begin{align}\label{eq:attUniformBoundStatement}
\sup_{N\ge1}\K_{N,\be}^{\mathrm{att}}(\mu)\le C_{\mathrm{ub}}^{\mathrm{att}}(\be,\mu).
\end{align}

Moreover, assume $\mu\in L^\infty(\R^\d)$, let $T_\mu^{\mathrm{att}}$ be the Hilbert--Schmidt operator with kernel $\tl\g_{\mathrm{att}}$, and set $\K_{\infty,\be}^{\mathrm{att}}(\mu)\coloneqq e^{\be I_\mu^{\mathrm{att}}}\det_2\pa*{\mathrm{Id}+\be T_\mu^{\mathrm{att}}}^{-1/2}$.  There is a threshold
\begin{align}\label{eq:attQuantitativeThreshold}
0<\be_{\mathrm q}=\be_{\mathrm q}(\d,\|\mu\|_{L^\infty},C_{\mathrm{mom}})\le\frac{\be_0}{2}
\end{align}
such that, for every fixed $0<\be<\be_{\mathrm q}$ with $e^{4\be h_\mu^{\mathrm{att}}}\in L^1(\mu)$, there are constants $C_{\mathrm{rate}}^{\mathrm{att}}(\be,\mu)<\infty$ and $c_{\mathrm{att}}(\be,\mu)>0$ satisfying
\begin{align}\label{eq:attQuantitativeDeterminantRate}
\left|
\K_{N,\be}^{\mathrm{att}}(\mu)
-\K_{\infty,\be}^{\mathrm{att}}(\mu)
\right|
\le C_{\mathrm{rate}}^{\mathrm{att}}(\be,\mu)(e+N)^{-c_{\mathrm{att}}(\be,\mu)}
\end{align}
for every $N\ge1$.  The determinant defining $\K_{\infty,\be}^{\mathrm{att}}(\mu)$ is positive.
\end{thm}
For a periodic uniform background, the one-body term vanishes and the additional exponential-integrability hypothesis in the quantitative assertion is automatic.
\endgroup

\begingroup
\begin{remark}[On the perturbative threshold and periodic uniform backgrounds]\label{rem:attractiveSharpPeriodicRegime}
The sufficient threshold $\be_0$ furnished by the preceding factorial-moment argument is not claimed to be sharp and should not be identified with the global-minimization, stability, or collision thresholds of the attractive logarithmic gas.  For the periodic interaction with uniform background, \cref{subsec:periodic-attractive-phase-diagram} records the exact one-dimensional formula, a dimension-dependent conjectural phase diagram, and self-contained critical and supercritical lower bounds.  In forthcoming work \cite{DGRPhaseTransitions2026}, the authors prove the sharp determinant limit throughout $0\le\be<2\d$ in dimensions $2\le\d\le6$, using the sharp zero-defect log-HLS/Beckner--Onofri inequality established by Lei and the third author \cite{LeiRosenzweigSharpLogHLS2026}; the latter work also proves that, in dimension eleven, the uniform phase loses global minimality strictly before its stability threshold.  The full-subcritical argument is nonperturbative: it separates a finite-cell macroscopic contribution, controlled by the sharp entropy--energy inequality, from a conditionally centered microscopic remainder controlled by canonical-array and occupancy estimates.
\end{remark}
\endgroup

The moment hypothesis \eqref{eq:attMomentAssumption} keeps the supremum over the second variable outside the $x$-integration.  It is therefore weaker than
\begin{equation}
\int_{\R^\d}\pa*{\sup_{y\in\R^\d}|\tl\g_{\mathrm{att}}(x,y)|}^p\,\dd\mu(x)\le C^p p!,
\end{equation}
which takes the pointwise supremum before integrating and is the norm considered by Jabin--Wang \cite{JabinWang2018}; that stronger norm excludes the Keller--Segel kernel.

The exponential-integrability condition mirrors \cref{rem:oneBodyExpIntegrability}, with the opposite sign.  Since the attractive potential grows at infinity, it constrains the upper tail of $h_\mu^{\mathrm{att}}$, and hence the decay of $\mu$.  Indeed, $\g_{\mathrm{att}}\in L^1(\mu^{\otimes2})$ implies $\int_{\R^\d}\log(1+|y|)\,\dd\mu(y)<\infty$, while
\begin{align}
h_\mu^{\mathrm{att}}(x)
\le \log(1+|x|)+\int_{\R^\d}\log(1+|y|)\,\dd\mu(y).
\end{align}
Consequently, $e^{2\be h_\mu^{\mathrm{att}}}\in L^1(\mu)$ whenever $\int_{\R^\d}|x|^{2\be}\,\dd\mu(x)<\infty$.  The local singularity, which may drive $h_\mu^{\mathrm{att}}$ to $-\infty$, enters with the favorable sign.

The proof of the uniform bound and determinant limit in \cref{thm:partitionFunctionBound} begins with the canonical statistic \eqref{eq:intro-U-statistic}.  Positive-definite truncation and a low/high-frequency decomposition reduce the singular kernel to data controlled by the master layer-cake estimate \cref{prop:masterLayerCake}; the high-frequency $L^2$, mixed, and operator norms are small precisely when $\s<\d/2$.  The logarithmic case uses the same mechanism after the infrared part has been paired with the one-body term through the cancellation in \cref{lem:logLargeScaleCancellation}.  The resulting uniform exponential moments are proved in \cref{prop:uniformBoundRepulsive}.

The determinant limit then follows by Hilbert--Schmidt approximation and the bounded-feature Laplace estimate \cref{prop:boundedFeatureQuadLaplace}, while \cref{lem:oneBodyReduction} shows that the centered one-body average contributes a lower-order error.  Beyond the Hilbert--Schmidt threshold, \cref{lem:quantitativeDyadicCompression} replaces the centered kernel by its conditional expectation on a dyadic spatial partition whose scale is coupled to $N$.  Conditional Jensen reduces the original partition function to the resulting finite-rank coarse-grained kernel, whose Hilbert--Schmidt norm diverges; the quantitative finite-particle estimate in \cref{prop:boundedFeatureQuadLaplace} then gives the explicit lower rates in \cref{prop:partition-function-sharpness}.

The attractive logarithmic theorem follows a genuinely different perturbative route, since positivity of the interaction is no longer available.  After the one-body factor is separated, centering forces the surviving moments of the pair statistic to involve repeated indices; decoupling by \cref{prop:decoup} and the factorial-moment hypothesis \eqref{eq:attMomentAssumption} make the resulting exponential series summable for $\be<\be_0$.  A feature-space expansion for the positive quadratic exponential, combined with logarithmic truncation, gives the quantitative determinant convergence in \eqref{eq:attQuantitativeDeterminantRate}.

\subsection{Entropic commutators and diffusive mean-field estimates}
\label{subsec:commutator-estimates}
Our second set of results concerns so-called commutator estimates for the variations of the modulated energy along transport fields. Let us briefly explain the context and motivation for these estimates.

The modulated energy approach to mean-field convergence consists of establishing a Gronwall relation for $\Fr_N(\XN^t,\mu^t)$, where $\XN^t$ is a solution of the microscopic system and $\mu^t$ is a solution of the limiting mean-field equation.  An essential point in this approach is to control the derivative of $\Fr_N$ along a transport $v:\R^\d\rightarrow\R^\d$.  In the modulated-free-energy method, this same derivative is the transport commutator left after the relative Fisher-information term has been put in favorable sign.

For applications to the mean-field limit, $v$ is the velocity field of the limiting mean-field dynamics.  For applications to central limit theorems (CLTs) for the fluctuations (the next-order description), $v$ is the gradient of a test function evolved along the adjoint linearized mean-field flow, and these inequalities are at the core of the ``transport'' approach to fluctuations of canonical Gibbs ensembles \cite{LebleSerfaty2018,Serfaty2023,peilen_local_2025}.  The same quantity also appears in the loop, or Dyson--Schwinger, equations in the random matrix theory literature \cite{BorotGuionnet2013OneCut,BorotGuionnet2024MultiCut,BorotGuionnetKozlowski2015}.  These identities are obtained by differentiating the Gibbs integral under infinitesimal changes of variables and thus admit a transport interpretation.

The higher-order commutators below are the corresponding higher variations of the modulated energy.  As in the first-order case, deterministic commutator estimates seek to control these quantities by $C_v(\Fr_N(\XN,\mu)+C_\mu N^{-\alpha})$ \cite{RS2024comm,HCRS2025}.  Second-order estimates are important for the aforementioned transport method for studying the fluctuations of Coulomb/Riesz gases.  Estimates beyond second order are also useful, allowing one to obtain finer estimates on the fluctuations of Riesz gases to treat a broader class of interactions \cite{peilen_local_2025} and also to compute the asymptotics of arbitrary-order cumulants \cite{rosenzweig_cumulants_nodate}.

We now use the partition-function control to prove the averaged, entropic counterpart of these deterministic commutator bounds when the point configuration is distributed according to the modulated Gibbs measure.  The main point is to prove a uniform unnormalized exponential-moment bound under $\mu^{\otimes N}$, then convert it by Donsker--Varadhan into an entropic estimate for arbitrary $N$-particle laws.

For a Lipschitz transport $v:\R^\d\rightarrow\R^\d$ and an integer $n\ge1$, define
\begin{align}\label{eq:Andef}
\mathsf{A}_n(v,\XN,\mu) \coloneqq \int_{(\R^\d)^2\setminus\triangle}\nab^{\otimes n}\g(x-y):(v(x)-v(y))^{\otimes n}\,\dd(\mu_N-\mu)^{\otimes 2}(x,y).
\end{align}
Since $\mathsf{A}_n(-v,\XN,\mu)=(-1)^n\mathsf{A}_n(v,\XN,\mu)$, a sign can be absorbed into $v$ only when $n$ is odd.  We therefore keep the explicit sign parameter $\sigma\in\{\pm 1\}$ in the all-order estimate.

\begin{thm}\label{thm:entropicCommutator}
Fix an integer $n\ge1$. Assume $0\le\s<\d/2$, $\mu\in L^\infty\cap\P(\R^\d)$, and $v:\R^\d\to\R^\d$ is globally Lipschitz; if $\s=0$, assume in addition \eqref{eq:logEnergyAssumption}.  Then there exist a threshold $\be_{\mathrm{loc}}=\be_{\mathrm{loc}}(\mu,v,n)>0$ and, for every $\be>0$, constants $\vep_{\be,n}>0$ and $C_{\mathrm{em}}(\be,\mu,v,n)<\infty$, made explicit in \cref{rem:entropicCommutatorConstants}, such that, for every $\sigma\in\{\pm 1\}$,
\begin{align}\label{eq:QNexpMomentAn}
\sup_{N\ge1}\E_{\XN\sim\Q_{N,\be}(\mu)}\Big[e^{\sigma\vep_{\be,n} N\mathsf{A}_n(v,\XN,\mu)}\Big]\le C_{\mathrm{em}}(\be,\mu,v,n).
\end{align}
Consequently, for every $f_N\in\P((\R^\d)^N)$,
\begin{align}\label{eq:entropicCommutatorEstimate}
\Big|\E_{\XN\sim f_N}\Big[N\mathsf{A}_n(v,\XN,\mu)\Big]\Big|\le C_{\be,n,v}N\Er_N(f_N,\mu)+C_{\be,\mu,v,n},
\end{align}
where $\Er_N$ is the modulated free energy \eqref{eq:ENdef} and
\begin{align}\label{eq:CnvExplicit}
C_{\be,n,v}\coloneqq\frac{\be}{\vep_{\be,n}}
\qquad\text{and}\qquad
C_{\be,\mu,v,n}\coloneqq\frac{1}{\vep_{\be,n}}
\pa*{\log C_{\mathrm{pf}}(\be,\mu)+\log C_{\mathrm{em}}(\be,\mu,v,n)}.
\end{align}
\end{thm}

Although this dependence is suppressed in the notation, $\be_{\mathrm{loc}}$, $\vep_{\be,n}$, and hence $C_{\be,n,v}$ depend on the background $\mu$ as well as on $\be$, $v$, and $n$.  We write $C_{\be,n,v}(\mu)$ when this dependence must be displayed; the detailed dependence of the exponential-moment and entropic constants is recorded in \cref{subsec:quantitative-bookkeeping}.

\medskip\noindent\textbf{Diffusive mean-field application.}
\label{subsec:application-diffusive-mean-field-convergence}
We now turn to the application to mean-field convergence and propagation of chaos for interacting systems of diffusive gradient flows.  Consider the overdamped $N$-particle diffusion
\begin{align}\label{eq:diffusiveNParticleSystem}
\dd x_i^t
=
-\frac1N\sum_{\substack{1\le j\le N\\ j\ne i}}\nabla \g(x_i^t-x_j^t)\,\dd t
+\sqrt{\frac2\be}\,\dd B_i^t,
\qquad i\in[N],
\end{align}
where $B_1,\ldots,B_N$ are independent Brownian motions.  

Establishing the mean-field limit means proving convergence, as $N\to\infty$, of the empirical measure
$\mu_N^t\coloneqq N^{-1}\sum_{i=1}^N\delta_{x_i^t}$ associated to a solution of \eqref{eq:diffusiveNParticleSystem}.  Assuming the initial empirical measures converge to a sufficiently regular density $\mu^0$, a formal It\^o calculation leads one to expect that, for $t>0$, $\mu_N^t$ converges to the solution of the McKean--Vlasov equation
\begin{align}\label{eq:diffusiveMeanFieldEquation}
\partial_t\mu^t
=
\frac1\be\Delta\mu^t+\div\pa*{\mu^t\nabla h_{\mu^t}},
\qquad h_{\mu^t}=\g\ast\mu^t,
\end{align}
which is completely deterministic, the noise having been averaged out to become diffusion.

Convergence of the empirical measure is qualitatively equivalent to propagation of molecular chaos; we refer the reader to \cite{Sznitman1991,HaurayMischler2014} for classical treatments and \cite{ChaintronDiez2022a,ChaintronDiez2022b} for modern reviews.  This latter notion means that, if the initial law $f_N^0$ is $\mu^0$-chaotic, in the sense that the $k$-point marginals $f_{N;k}^0\rightharpoonup(\mu^0)^{\otimes k}$ as $N\to\infty$ for every fixed $k$, then the evolved law $f_N^t$ is $\mu^t$-chaotic.

For positive-temperature gradient dynamics, the modulated energy combines naturally with relative entropy in the form of the modulated free energy\footnote{In the sub-Coulomb Riesz regime, the mean-field limit can also be treated by the modulated energy alone, as shown in \cite{RS2021}.  We use the modulated-free-energy formulation here because it is the setting in which the partition-function and entropic-commutator estimates enter.} \cite{BreschJabinWang2019,BreschJabinWang2019PKS,BreschJabinWang2023,RS2023lsi,CdCRS2025}.  More generally, the modulated-free-energy method combines two complementary controls: relative entropy controls the marginals, while the modulated energy controls the empirical measure.  Differentiating $\Er_N(f_N^t,\mu^t)$ along the $N$-particle flow and along \eqref{eq:diffusiveMeanFieldEquation} produces a relative Fisher-information term with favorable sign and a transport commutator.  For the singular logarithmic and Riesz interactions considered here, the principal analytic input is the control of this commutator furnished by \cref{thm:entropicCommutator}; the corresponding deterministic first-variation estimates are developed in \cite{HCRS2025,Rosenzweig2026Commutators}.  The Fisher-information term is discarded below, while the commutator is exactly the $n=1$ case of \cref{thm:entropicCommutator}.  For the repulsive logarithmic interaction, the first two authors showed that the commutator contribution in the modulated-free-energy differential inequality is only $O(N^{-1})$ after time integration, eliminating the extra factor of $\log N$ present in cruder deterministic estimates \cite{DelgadinoGvalani2025}.  The present theorem proves the corresponding conclusion in the locally square-integrable repulsive log/Riesz regime using the partition-function and entropic-commutator estimates developed in this paper.

\begin{thm}[Diffusive mean-field convergence]\label{thm:diffusiveMeanFieldConvergence}
Fix $T_\ast\in(0,\infty]$ and $\be>0$.  Let $(\mu^t)_{0\le t<T_\ast}$ solve \eqref{eq:diffusiveMeanFieldEquation}, and assume that this solution admits a globally\footnote{For sufficiently regular solutions, this velocity is
\(
u^t=-\nabla h_{\mu^t}-\be^{-1}\nabla\log\mu^t\).} Lipschitz velocity field $u^t$ satisfying
\begin{align}\label{eq:diffusiveMeanFieldContinuityVelocity}
\partial_t\mu^t+\div(\mu^t u^t)=0.
\end{align}
Let $(f_N^t)_{0\le t<T_\ast}$ be a curve of $N$-particle probability laws and write
\begin{align}\label{eq:ENtDef}
\Er_N^t\coloneqq \Er_N(f_N^t,\mu^t).
\end{align}
Assume that $t\mapsto\Er_N^t$ is absolutely continuous and that the pair $(f_N^t,\mu^t)$ satisfies the modulated-free-energy dissipation inequality
\begin{align}\label{eq:abstractMFEdiss}
\frac{\dd}{\dd t}\Er_N^t
\le \frac1N\Big|\E_{\XN\sim f_N^t}\Big[N\mathsf{A}_1(u^t,\XN,\mu^t)\Big]\Big|, \qquad \text{a.e. } t\in(0,T_\ast).
\end{align}
Assume also that, for a.e. $t\in(0,T_\ast)$, the hypotheses of \cref{thm:entropicCommutator} are satisfied with $n=1$, $\mu=\mu^t$, and $v=u^t$, at the fixed inverse temperature $\be$.

Set
\begin{align}\label{eq:MFEGronwallFactors}
\mathcal A^t
&\coloneqq \exp\pa*{\int_0^t C_{\be,1,u^\tau}(\mu^\tau)\,\dd\tau},\\
\mathcal B^t
&\coloneqq \int_0^t \exp\pa*{\int_\tau^t C_{\be,1,u^r}(\mu^r)\,\dd r}
C_{\be,\mu^\tau,u^\tau,1}\,\dd\tau,
\end{align}
where the constants are those of \cref{thm:entropicCommutator}.  If these quantities are finite for every $t<T_\ast$, then
\begin{align}\label{eq:MFEdiffusiveGronwall}
\Er_N(f_N^t,\mu^t)
\le \mathcal A^t\Er_N(f_N^0,\mu^0)+\frac{\mathcal B^t}{N},
\qquad 0\le t<T_\ast.
\end{align}
\end{thm}

The $O(1/N)$ scale in the estimate \eqref{eq:MFEdiffusiveGronwall} is the iid-chaotic scale.  The diagonal-exclusion correction in the modulated energy is already of order $N^{-1}$, and one should not expect a generally smaller additive floor in this topology.  This does not contradict the sharpness statement in the work of Chodron de Courcel, Serfaty, and the third author \cite{CdCRS2025}.  Their $N^{\s/\d-1}$ scale is sharp when the temperature is allowed to vanish with $N$, or when the initial randomization already produces a modulated free energy of order $N^{\s/\d-1}$, which is the case in the usual log/Riesz gas temperature scaling or the corresponding low-temperature regime; see, e.g., \cite[Section~3.2, especially (3.2.4)]{SerfatyLN}.  In the genuinely diffusive regime considered here, with fixed positive noise and iid-scale initial modulated free energy, the additive error is of the same order as the initial error.

\begin{remark}[Uniform-in-time propagation and generation]\label{rem:uniform-generation-chaos}
Evidently, the same estimate is uniform in time if the coefficients generated by the full velocity $u^t=-\nabla h_{\mu^t}-\be^{-1}\nabla\log\mu^t$ are integrable strongly enough that $\sup_{t\ge0}\mathcal A^t<\infty$ and $\sup_{t\ge0}\mathcal B^t<\infty$.  This is the abstract form of the usual uniform-in-time propagation-of-chaos mechanism in the modulated-free-energy method.  If, in addition, the nonpositive Fisher-information term in the dissipation identity is retained and a uniform modulated logarithmic Sobolev inequality for $\Q_{N,\be}(\mu^t)$ is available, then the same argument gives generation of chaos down to the $O(1/N)$ diffusive floor; this is the mechanism emphasized in \cite{RS2023lsi}.
\end{remark}

\begingroup

The proof of \cref{thm:entropicCommutator} uses the same exponential-moment mechanism as the partition-function theorem with the transport variation inserted.  The local numerator estimate \cref{lem:localNumeratorComm} absorbs the short-range commutator into the modulated energy, and the H\"older bootstrap in \cref{subsec:removal-smallness-assumption} removes the small-inverse-temperature restriction by invoking the partition-function bound at a larger inverse temperature.  The Donsker--Varadhan formula then gives the entropic inequality.  Inserting its first-order case into the modulated-free-energy dissipation inequality and applying Gronwall's lemma proves \cref{thm:diffusiveMeanFieldConvergence}.  The corresponding interpolation between the fixed-temperature entropic estimate and the deterministic pointwise estimate is stated and proved in \cref{cor:entropicCommutatorTemperatureInterpolation}.

\subsection{Static and dynamical fluctuations}
\label{subsec:static-dynamic-fluctuations}

We first record the static fluctuation consequence of the partition-function convergence.  For a real-valued test function $\phi\in L^2(\mu)$, set
\begin{align}\label{eq:linearEmpiricalFluctuationDef}
S_N(\phi)
\coloneqq \sqrt N\pa*{\int_{\R^\d}\phi\,\dd\mu_N-\int_{\R^\d}\phi\,\dd\mu},
\qquad
\bar\phi\coloneqq\phi-\int_{\R^\d}\phi\,\dd\mu,
\end{align}
and define
\begin{align}\label{eq:staticFluctuationCovariance}
\Sigma_{\be,\mu}(\phi,\psi)
\coloneqq
\left\langle \bar\phi,
(\mathrm{Id}+\be T_\mu)^{-1}\bar\psi
\right\rangle_{L^2(\mu)}.
\end{align}
Since $T_\mu$ is bounded, nonnegative, and self-adjoint on $L_0^2(\mu)$, the covariance form $\Sigma_{\be,\mu}$ is well defined and continuous on $L^2(\mu)\times L^2(\mu)$.

\begin{prop}[Static linear-statistics CLT]\label{prop:linearEmpiricalFluctuationCLT}
Assume the hypotheses of \cref{thm:partitionFunctionBound}, including \eqref{eq:logRateExtraAssumptions} when $\s=0$.  Fix $\be\ge0$ and real-valued functions $\phi_1,\ldots,\phi_k\in L^2(\mu)$.
\begingroup

Then, for every $\theta=(\theta_1,\ldots,\theta_k)\in\R^k$,
\begin{align}\label{eq:linearEmpiricalFluctuationCF}
\lim_{N\to\infty}
\E_{\XN\sim\Q_{N,\be}(\mu)}
\brak*{\exp\pa*{\iu\sum_{i=1}^k\theta_iS_N(\phi_i)}}
=
\exp\pa*{-\frac12\sum_{i,j=1}^k\theta_i\theta_j
\Sigma_{\be,\mu}(\phi_i,\phi_j)}.
\end{align}
Consequently, $(S_N(\phi_1),\ldots,S_N(\phi_k))$ converges in law under $\Q_{N,\be}(\mu)$ to a centered Gaussian vector with covariance matrix $(\Sigma_{\be,\mu}(\phi_i,\phi_j))_{i,j=1}^k$.
\endgroup

If, in addition, $\phi_1,\ldots,\phi_k\in L^\infty(\mu)$, then, for every $\theta\in\R^k$,
\begin{align}\label{eq:linearEmpiricalFluctuationMGF}
\lim_{N\to\infty}
\E_{\XN\sim\Q_{N,\be}(\mu)}
\brak*{\exp\pa*{\sum_{i=1}^k\theta_iS_N(\phi_i)}}
=
\exp\pa*{\frac12\sum_{i,j=1}^k\theta_i\theta_j
\Sigma_{\be,\mu}(\phi_i,\phi_j)}.
\end{align}
\end{prop}

\begin{remark}[Equilibrium and dynamical interpretation]\label{rem:staticCLTInterpretation}
When $\mu=\mu_\be$ is the thermal equilibrium law, $\Q_{N,\be}(\mu_\be)=\mathbb P_{N,\be}$, so \cref{prop:linearEmpiricalFluctuationCLT} is a central limit theorem for the original canonical ensemble.  \begingroup
The $L^2(\mu)$ test class is the natural and essentially optimal class for a general $\sqrt N$-scale Gaussian theorem with the covariance above.  Indeed, since $T_\mu$ is nonnegative,
\begin{align}
\frac{1}{1+\be\|T_\mu\|_{\mathrm{op}}}
\|\bar\phi\|_{L^2(\mu)}^2
\le
\Sigma_{\be,\mu}(\phi,\phi)
\le
\|\bar\phi\|_{L^2(\mu)}^2.
\end{align}
Thus, the covariance norm is equivalent to the centered $L^2(\mu)$ norm.  The boundedness assumption in the final assertion is used only to justify convergence of the two-sided moment generating functions.  The proposition also identifies the initial Gibbs characteristic function entering the dynamical central limit theorem below.
\endgroup
\end{remark}

We now turn to the fluctuation process associated with \eqref{eq:diffusiveNParticleSystem}--\eqref{eq:diffusiveMeanFieldEquation}.  For a bounded test $\phi$, write
\begin{align}\label{eq:dynamicFluctuationStatistic}
S_N^t(\phi)
\coloneqq
\sqrt N\pa*{\int_{\R^\d}\phi\,\dd\mu_N^t-\int_{\R^\d}\phi\,\dd\mu^t}.
\end{align}
For $\nu\in\P(\R^\d)$, define the adjoint linearized operator
\begin{align}\label{eq:adjointLinearizedOperator}
(\mathcal L_{\be,\nu}\phi)(x)
\coloneqq
\frac1\be\Delta\phi(x)
-\int_{\R^\d}\nabla\g(x-y)\cdot
\pa*{\nabla\phi(x)-\nabla\phi(y)}\,\dd\nu(y).
\end{align}

\begin{thm}[Dynamical linear-statistics CLT]\label{thm:dynamicLinearStatisticsCLT}\label{def:dynamicCLTAdmissibility}
Let $0\le\s<\d/2$, $\be>0$, and $T>0$.  Let $(\mu^t)_{0\le t\le T}$ be a mean-field trajectory and $(f_N^t)_{0\le t\le T}$ the corresponding particle laws.  Assume the following.
\begin{enumerate}[(i)]
\item The particle system and mean-field equation are well posed in a class for which the time-dependent It\^o formulas used below are valid; the Brownian motions are independent of the initial configuration; $f_N^0=\Q_{N,\be}(\mu^0)$; and the estimate \eqref{eq:MFEdiffusiveGronwall} holds.  Moreover,
\begin{align}\label{eq:dynamicUniformBackgroundControls}
\sup_{0\le t\le T}
\pa*{\|\mu^t\|_{L^\infty}+\mathcal A^t+\mathcal B^t}
<\infty.
\end{align}
If $\s=0$, assume in addition that
\begin{align}\label{eq:dynamicUniformLogEnergy}
\inf_{0\le t\le T} I_{\mu^t}>-\infty.
\end{align}
\item There are constants $0<c_{\be,T}\le C_{\be,T}<\infty$ and $C_{2\be,T}<\infty$ such that
\begin{align}\label{eq:dynamicUniformPartitionBounds}
&\inf_{N\ge1}\inf_{0\le t\le T}\K_{N,\be}(\mu^t)\ge c_{\be,T},
\qquad
\sup_{N\ge1}\sup_{0\le t\le T}\K_{N,\be}(\mu^t)\le C_{\be,T},\\
&\sup_{N\ge1}\sup_{0\le t\le T}\K_{N,2\be}(\mu^t)\le C_{2\be,T},
\end{align}
and $\mu^0$ satisfies the hypotheses of \cref{prop:linearEmpiricalFluctuationCLT}.
\item There is a real linear test space $\mathscr D_T\subset C_b^2(\R^\d)$ such that, for every $t\in[0,T]$ and $\phi\in\mathscr D_T$, the terminal-value problem
\begin{align}\label{eq:backwardAdjointEquation}
-\partial_r f^r=\mathcal L_{\be,\mu^r}f^r,
\qquad 0\le r\le t,
\qquad f^t=\phi,
\end{align}
has a unique real-valued classical solution, denoted by $f^r=\Uprop^{r,t}\phi$, for which $f$, $\nabla f$, $\nabla^{\otimes2}f$, and $\partial_r f$ are jointly continuous and
\begin{align}\label{eq:dynamicAdjointRegularity}
\sup_{0\le r\le t}
\pa*{
\|f^r\|_{L^\infty}
+\|\nabla f^r\|_{L^\infty}
+\|\nabla^{\otimes2}f^r\|_{L^\infty}
+\|\partial_r f^r\|_{L^\infty}}
<\infty.
\end{align}
The bound in \eqref{eq:dynamicAdjointRegularity} is uniform over each fixed finite family of terminal pairs $(t,\phi)$.
\end{enumerate}
Fix $k\ge1$, times $t_1,\ldots,t_k\in[0,T]$, and tests $\phi_1,\ldots,\phi_k\in\mathscr D_T$.  Then
\begin{align}
\pa*{S_N^{t_1}(\phi_1),\ldots,S_N^{t_k}(\phi_k)}
\Longrightarrow (G_1,\ldots,G_k),
\end{align}
where the vector on the right is centered Gaussian with covariance
\begin{align}\label{eq:dynamicFluctuationCovariance}
\E[G_aG_b]
={}&\Sigma_{\be,\mu^0}\pa*{\Uprop^{0,t_a}\phi_a,\Uprop^{0,t_b}\phi_b}\\
&+\frac2\be\int_0^{t_a\wedge t_b}
\int_{\R^\d}\nabla\Uprop^{r,t_a}\phi_a\cdot
\nabla\Uprop^{r,t_b}\phi_b\,\dd\mu^r\,\dd r.
\end{align}
\end{thm}

The uniform background controls in hypothesis~(i), the partition-function bounds in hypothesis~(ii), and the Hessian bound in hypothesis~(iii) imply the family-uniform exponential commutator estimate \eqref{eq:dynamicUniformCommutatorBounds} used below; see \cref{rem:entropicCommutatorConstants,lem:dynamicEntropyTransfers}.

\begin{remark}[Backward adjoint regularity]\label{rem:HRSAdjointRegularity}
The backward-adjoint regularity required in hypothesis~(iii) of \cref{thm:dynamicLinearStatisticsCLT} is proved in the forthcoming work of Huang, Serfaty, and the third author \cite{HuangRosenzweigSerfaty2026}.
\end{remark}

\begin{remark}[Toward a path-space central limit theorem]\label{rem:pathSpaceCLT}
\Cref{thm:dynamicLinearStatisticsCLT} identifies the finite-dimensional distributions of any possible path-space limit.  Consequently, if the fluctuation fields are tight in a separable continuous-path space in which the admissible tests form a measure-determining class, and if the Gaussian process characterized by \eqref{eq:dynamicFluctuationCovariance} has trajectories in that space, then every subsequential limit has the same law.  Thus, the remaining step toward a path-space central limit theorem is the corresponding tightness and trajectory-regularity estimate.
\end{remark}

\begingroup
For the static theorem, the linear tilt is first combined with the joint Gaussian/degenerate-chaos limit and the determinant normalization to prove the bounded-test Laplace-transform formula.  A truncation argument, using a uniform $L^2(\mu^{\otimes N})$ bound on the density of $\Q_{N,\be}(\mu)$ relative to $\mu^{\otimes N}$, then extends the weak central limit theorem to arbitrary $L^2(\mu)$ tests; see \cref{subsec:static-linear-statistics-proof}.  \endgroup For the dynamical theorem, the backward adjoint equation cancels the linearized drift and the finite-$N$ identity \cref{prop:finiteNAdjointDuality} expresses each terminal statistic as an initial statistic, a martingale, and a transport-commutator remainder.  The static central limit theorem identifies the initial term, while \cref{lem:dynamicEntropyTransfers}, using the family-uniform form of the entropic commutator estimate, makes the remainder negligible.  A characteristic-function argument, together with convergence of the martingales' predictable quadratic covariations to their deterministic mean-field limits, gives the time-integrated part of the covariance in \eqref{eq:dynamicFluctuationCovariance}.
\endgroup

\subsection{Relation to previous work}
\label{subsec:intro-related-work}

There is a large literature on full Gibbs partition functions for logarithmic, Coulomb, and Riesz gases.  At the macroscopic level, first-order free-energy asymptotics and empirical-measure large deviations identify the mean-field variational problem and its equilibrium measures; at the next orders, microscopic screening, renormalized energies, empirical-field large deviations, local laws, rigidity, and Coulomb/Riesz free-energy expansions enter the description.  For a review of this body of results, see Serfaty's lecture notes \cite[Chapters~3, 5, and~8--13]{SerfatyLN}.  A further distinction is the temperature scale.  In many of the microscopic equilibrium results, the relevant inverse-temperature scale is the one at which nearest-neighbor Riesz costs enter the Gibbs weight; in the present normalization this means $\be=\be_N$ of order $N^{1-\s/\d}$, and $\be_N\sim N$ in the logarithmic case.  The present paper instead works in the fixed-temperature diffusive regime $\be_N=O(1)$, which is a high-temperature regime from that perspective.  Those results are substantially more refined at the microscopic level than the present work.  They concern the full confined ensemble and its equilibrium structure.  On the other hand, these results do not imply those of the present paper because of the differences in temperature regimes.  In the present paper, the singularity of the interaction is not controlled through the screening and renormalized-energy mechanisms that describe microscopic equilibrium structure.  The residual normalization is instead a high-temperature fluctuation problem for the centered two-body statistic, and its sharp behavior is governed by the Hilbert--Schmidt threshold and the Carleman--Fredholm determinant.

\begingroup
Central limit theorems for equilibrium linear statistics of logarithmic,
Coulomb, and Riesz gases have been obtained in several complementary
temperature and spatial regimes.  At the diffusive, or high-temperature,
scale considered here, one-dimensional logarithmic ensembles were treated
for circular and real beta ensembles in
\cite{HardyLambert2021,NakanoTrinh2018,
DworaczekGueraMemin2024,MazzucaMemin2024}, generally for smooth or
polynomial test functions and, in some cases, with quantitative
normal-approximation estimates.  At stronger, non-diffusive interaction
scales, quantitative or mesoscopic central limit theorems for
one-dimensional beta ensembles were proved in
\cite{BekermanLebleSerfaty2018,LambertLedouxWebb2019,
BekermanLodhia2018,Lambert2021,BourgadeModyPain2022}, while
two-dimensional Coulomb fluctuations at the conventional interaction
scale were treated at macroscopic and mesoscopic scales in
\cite{LebleSerfaty2018,BauerschmidtBourgadeNikulaYau2019}.
Serfaty's fluctuation and free-energy theory includes broad, possibly
particle-number-dependent temperature regimes, including the diffusive
regime in dimension two; the corresponding higher-dimensional central
limit theorem is conditional on a no-phase-transition assumption and a
sufficiently accurate free-energy expansion \cite{Serfaty2023}.
Recent work of Peilen and Serfaty treats super-Coulombic Riesz gases in
arbitrary dimension, in a restricted range of inverse powers and inverse
temperatures, and obtains a central limit theorem at small mesoscopic
scales \cite{peilen_local_2025}.  Boursier's fixed-temperature circular
Riesz theorem is a particularly relevant low-regularity result, covering
interval counts and certain algebraically singular test functions
\cite{BoursierCircularRiesz}.  By contrast,
\cref{prop:linearEmpiricalFluctuationCLT} applies at the diffusive scale,
for arbitrary prescribed bounded backgrounds satisfying the
partition-function hypotheses, to any fixed finite family of \begingroup real-valued
square-integrable test functions throughout the logarithmic/Riesz
Hilbert--Schmidt range.  For bounded tests, it also gives convergence of
the two-sided moment generating functions for every real tilt.  \endgroup It does
not address particle-number-dependent mesoscopic tests, convergence of
the full spatial fluctuation field, or a quantitative normal-approximation
rate.
\endgroup

A separate literature in one dimension gives all-order expansions of log-gas and beta-ensemble partition functions by loop equations and topological recursion; see, for example, \cite{BorotGuionnet2013OneCut,BorotGuionnet2024MultiCut,BorotGuionnetKozlowski2015,BekermanLebleSerfaty2018,DworaczekGuera2024}.  We emphasize that these results use the special structure of the one-dimensional logarithmic potential, which does not carry over to higher dimensions or to Riesz interactions.

The dynamical motivation comes from the modulated-energy and modulated-free-energy approach to singular mean-field limits.  The method was developed for Coulomb/Riesz and related singular interactions in \cite{Serfaty2020,Duerinckx2016,NRS2021,RS2021,JabinWang2018,BreschJabinWang2019,BreschJabinWang2023}, and the modulated Gibbs viewpoint was used in \cite{RS2023lsi} to connect propagation and generation of chaos with uniform modulated logarithmic Sobolev inequalities.  In this framework, \eqref{eq:intro-entropy-rewrite} shows that $\log\K_{N,\be}(\mu)$ is the exact defect in changing the reference entropy from $\mu^{\otimes N}$ to the modulated Gibbs measure.  The present paper supplies the corresponding order-one partition-function estimate in the locally square-integrable repulsive log/Riesz regime.

\begingroup
Fluctuation-process central limit theorems for regular mean-field
particle systems are classical; see, for instance,
\cite{ShigaTanaka1985,FernandezMeleard1997}, and
\cite{DuerinckxGlauber2021} for a modern, quantitative perspective.  For
singular diffusions, Wang, Zhao, and Zhu prove convergence of the
fluctuation process to a generalized Ornstein--Uhlenbeck process, with
the two-dimensional Biot--Savart law as a principal application, while
Cecchin and Nikolaev obtain quantitative weak convergence in negative
Sobolev spaces and treat both viscous-vortex and repulsive-Coulomb
models \cite{WangZhaoZhu2023,CecchinNikolaev2025}.  Duerinckx and
Jabin establish correlation estimates and a fixed-time central limit
theorem for Langevin particle systems with square-integrable singular
interactions; their analysis is presented for underdamped dynamics,
while the corresponding correlation estimates extend, with minor
adaptations, to overdamped systems
\cite[Remark~1.2]{DuerinckxJabin2025Correlations}.  Their hierarchical
argument does not, however, identify time correlations along
trajectories.  In the special one-dimensional logarithmic setting,
global fluctuation results for Dyson dynamics were obtained by Bender
\cite{Bender2008} and Unterberger \cite{Unterberger2018}, while
Nakano, Trinh, and Trinh \cite{NakanoTrinhTrinh2023} prove
process-level central limit theorems for moment observables of
high-temperature beta-Dyson and beta-Laguerre systems.
\Cref{thm:dynamicLinearStatisticsCLT} differs in starting from a
correlated modulated Gibbs law and in treating the full
Hilbert--Schmidt logarithmic/Riesz class under the stated mean-field
and backward-adjoint assumptions.  It identifies joint linear
statistics at finitely many times, and hence their time correlations,
but does not provide a quantitative convergence rate or path-space
tightness.  Once the static fluctuation theorem and the entropic
commutator estimate are available, the final probabilistic stage of
the proof is comparatively direct and elementary: the backward
adjoint equation cancels the linearized drift, entropy control removes
the commutator remainder, and a finite-dimensional martingale
characteristic-function argument identifies the covariance.
\endgroup

Several works are closer at the level of partition functions.  Ben Arous and Brunaud treated the regular case of bounded continuous mean-field interactions by Laplace-method arguments, obtaining nondegenerate Gaussian fluctuations and qualitative limits involving the Carleman--Fredholm determinant for the modulated partition function \cite{BenArousBrunaud1990}.  The present theorem can be viewed as the singular log/Riesz analogue with the sharp Hilbert--Schmidt threshold made explicit.  An earlier logarithmic equilibrium predecessor is the work of Grotto and Romito \cite{GrottoRomito2020}, who treat the two-dimensional case on the flat torus with uniform background.  They prove convergence of the corresponding partition functions and Gaussian fluctuations.  Although their paper is written in the two-dimensional setting, the argument may be generalized, with some modifications, to higher-dimensional logarithmic kernels.  Their work does not, however, provide the commutator estimates or the modulated-free-energy mean-field application developed here.  The first two authors recently proved the corresponding arbitrary-background logarithmic result: for repulsive logarithmic kernels, on both the torus and the whole space, they obtain an $N$-uniform modulated partition-function bound for arbitrary reference laws with bounded density \cite{DelgadinoGvalani2025}.  The proof is Nelson/$\varphi^4_2$-inspired \cite{Nelson1966Quartic}: a Donsker--Varadhan reduction is combined with heat regularization, layer-cake estimates, Besov control, and high-moment correlation inequalities.  The present argument uses a different $U$-statistic route, with positive-definite truncations and GLZ tail bounds, and in addition identifies the Carleman--Fredholm determinant limit and the sharp log/Riesz Hilbert--Schmidt threshold.

During the preparation of the present manuscript, Duerinckx and Jabin obtained closely related high-temperature estimates \cite{DuerinckxJabin2026}.  Their main analysis is presented on the torus but is extendable to the whole space as explained in Appendix~A of their work.  In the homogeneous torus setting, they prove the uniform square-integrable bound, the qualitative determinant limit describing the normalization of equilibrium quadratic fluctuations, and qualitative divergence outside the square-integrable class.  Their main result concerns small fluctuations of the Newtonian dynamics around Gibbs equilibrium and proves convergence to the linearized Vlasov equation, reducing weighted $L^p$ estimates on the kinetic Gibbs equilibrium to spatial partition-function bounds through the Maxwellian factorization described in \cref{rem:torus-kinetic-variants}.  The proof is math-physics inspired, relying on a cluster/Mayer expansion \cite{Brydges1984Cluster}.  Duerinckx and Jabin note that their main result extends to arbitrary inverse temperature $\be<\infty$ as soon as the corresponding spatial partition functions remain uniformly bounded; see \cite[Remark~1.2]{DuerinckxJabin2026}.  The present paper supplies the needed partition-function estimates at every fixed $\be<\infty$ for the repulsive logarithmic/Riesz interactions in the Hilbert--Schmidt regime.  More generally, it addresses the modulated partition function around an arbitrary bounded reference density $\mu$ in the whole-space setting, with the torus variant described in \cref{rem:torus-kinetic-variants}, and proves quantitative convergence to the Carleman--Fredholm determinant associated with $T_\mu$.  We also prove sharpness of the Hilbert--Schmidt threshold, with explicit nonoptimal divergence rates at and above it.  The proof uses canonical $U$-statistics and decoupling inequalities rather than a cluster expansion.

\begingroup
Zhenfu Wang has communicated to the first author that he has extended the partition-function estimates of Duerinckx and Jabin to all fixed inverse temperatures in their homogeneous setting, by refining the same cluster/Mayer-expansion proof with a related decomposition of the interaction \cite{WangPrivate2026}. The communicated work does not contain analogues of the entropic commutator estimates here and consequently does not yield the corresponding order-$N^{-1}$ modulated-free-energy estimate or the static and dynamical central limit theorems for joint linear statistics.
\endgroup

\begingroup
On the attractive side, Grotto recently considered the two-dimensional periodic logarithmic interaction with uniform background, in the equivalent negative-temperature point-vortex formulation.  For sufficiently small attractive inverse temperature, he proves qualitative convergence of the partition function to the corresponding Carleman--Fredholm determinant and convergence of the fluctuation field to the associated Gaussian energy--enstrophy ensemble \cite{Grotto2026NegativeTemperature}.  In contrast, the quantitative conclusion of \cref{thm:attractive-log-small-beta-partition} gives a power rate of determinant convergence for arbitrary bounded backgrounds satisfying the stated factorial-moment and exponential-integrability hypotheses.  His proof is perturbative and based on a Mayer/cluster expansion.  The work \cite{Grotto2026NegativeTemperature} does not address the entropic commutator estimates developed here or their applications to quantitative dynamical mean-field limits and dynamical central limit theorems for joint linear statistics at finitely many times.
\endgroup

Finally, we note that the proof belongs to the probabilistic theory of degenerate $U$-statistics.  After Hoeffding centering, \eqref{eq:intro-U-statistic} is a canonical degree-two statistic.  Finite-rank kernels reduce to the classical Gaussian-chaos limit for degenerate quadratic forms \cite{Serfling1980,deJong1990}.  The singular infinite-dimensional problem is controlled by decoupling and exponential inequalities for canonical $U$-statistics, in particular the estimates of de la Pe\~{n}a and Montgomery-Smith and those of Gin\'{e}, Lata\l{}a, and Zinn \cite{dPmS1995,GLZ2000}.  This explains both the appearance of the determinant \eqref{eq:intro-det2-heuristic} and the sharpness of the $L^2(\mu^{\otimes2})$ threshold.

\subsection{Proof strategy and logical structure}
\label{subsec:intro-proof-architecture}

This subsection records the logical dependence among the main results.  More detailed roadmaps of the individual arguments are given at the beginning of the corresponding proof sections.

The main Hilbert--Schmidt branch of the argument is organized around a common exponential-moment input.  The master layer-cake estimate of \cref{sec:probabilistic-truncation-preliminaries}, together with the deterministic truncation package in \cref{sec:truncation-estimates}, yields uniform product-space exponential moments for the centered interaction.  Hilbert--Schmidt approximation and the bounded-feature Laplace estimate then give the regularized-determinant limit.  Applying the same mechanism to a transport variation gives the exponential commutator estimate; Donsker--Varadhan converts it into \cref{thm:entropicCommutator}, and the first-order case closes the Gronwall inequality in \cref{thm:diffusiveMeanFieldConvergence}.

The fluctuation results form the second downstream branch.  The determinant limit first identifies the static tilted Laplace transforms for bounded tests and hence the covariance in \cref{prop:linearEmpiricalFluctuationCLT}.  A uniform $L^2(\mu^{\otimes N})$ bound on the Gibbs density, followed by truncation, extends the weak central limit theorem to arbitrary $L^2(\mu)$ tests.  For the dynamical theorem, the backward adjoint equation cancels the linearized drift and expresses each terminal linear statistic as the sum of an initial statistic, a martingale, and a transport-commutator remainder.  The static central limit theorem identifies the initial term, while entropy transfer and the exponential commutator estimate make the remainder negligible.  A characteristic-function argument, together with convergence of the martingales' predictable quadratic covariations to their deterministic mean-field limits, produces the time-integrated part of the limiting covariance.  The hypotheses in \cref{thm:dynamicLinearStatisticsCLT} provide the time-dependent It\^o identities and backward adjoint evolution, while the uniform density and logarithmic-energy controls, together with the adjoint Hessian bound, allow \cref{thm:entropicCommutator} to be applied uniformly along the mean-field trajectory.

Two arguments are logically separate from this chain.  The attractive logarithmic theorem follows a perturbative route, since positivity of the interaction is no longer available.  After the one-body factor is separated, centering forces the surviving moments of the pair statistic to involve repeated indices; decoupling and the factorial-moment hypothesis make the resulting exponential series summable at sufficiently small inverse temperature.  A feature-space expansion for the positive quadratic exponential, combined with logarithmic truncation, gives the quantitative determinant convergence in \eqref{eq:attQuantitativeDeterminantRate}.  Beyond the Hilbert--Schmidt threshold, the centered kernel is replaced by its conditional expectation on a dyadic spatial partition whose scale is coupled to $N$.  Conditional Jensen reduces the original partition function to the resulting finite-rank coarse-grained kernel, whose Hilbert--Schmidt norm diverges; the bounded-feature Laplace estimate in \cref{prop:boundedFeatureQuadLaplace} then yields the quantitative divergence rates.  In the logarithmic Hilbert--Schmidt argument, the further non-generic ingredient is the infrared cancellation in \cref{lem:logLargeScaleCancellation}.

\subsection{Organization of the paper}

The remainder of the paper is organized as follows.  \Cref{sec:probabilistic-truncation-preliminaries} develops the master layer-cake principle for degenerate $U$-statistics, and \cref{sec:truncation-estimates} collects the deterministic truncation estimates.  \Cref{sec:partition-function-estimate-proof} proves the repulsive logarithmic/Riesz partition-function estimates, their temperature interpolation, and the quantitative sharpness of the Hilbert--Schmidt threshold.  \Cref{sec:attractive-log-small-beta} proves the attractive logarithmic uniform bound and quantitative determinant limit, and then treats the periodic uniform-background phase diagram.  \Cref{sec:entropic-commutator-estimate-proof} proves the entropic commutator estimate and the diffusive mean-field application, while \cref{sec:static-dynamic-clt-proofs} proves the static and dynamical fluctuation theorems.  \Cref{sec:further-applications-future-directions} records consequences and open problems, and \cref{subsec:quantitative-bookkeeping} gives the detailed dependence of constants.

\subsection{Statement on AI use}

The authors used generative AI tools (OpenAI's ChatGPT and Codex, and Anthropic's Claude) during the development and preparation of this paper to identify potentially relevant literature, explore and test mathematical arguments, improve the exposition, and check internal consistency and cross-references.  The authors treated all AI-generated output as provisional material requiring independent verification and did not rely on it as an authority, checking literature suggestions against the relevant sources and independently working through mathematical suggestions before deciding whether to incorporate them into the paper.  The authors made all final decisions concerning the manuscript and take full responsibility for its contents.

\endgroup

\section{A master layer-cake estimate for degenerate \texorpdfstring{$U$}{U}-statistics}
\label{sec:probabilistic-truncation-preliminaries}
This section develops the probabilistic engine of the paper.  The decoupling and GLZ tail bounds for degenerate $U$-statistics, recalled in \cref{subsec:decoupling-glz-tail-estimates}, are packaged in \cref{subsec:master-exponential-moment} into a master layer-cake estimate for exponential moments, which is the result applied in \Cref{sec:partition-function-estimate-proof,sec:entropic-commutator-estimate-proof}.

\subsection{Decoupling and GLZ tail bounds}
\label{subsec:decoupling-glz-tail-estimates}

We recall from Gin\'{e}, Lata\l{}a, and Zinn \cite[Theorem 3.3]{GLZ2000} the following tail bound for $U$-statistics.
\begin{prop}\label{prop:UstatTail}
There exists a universal constant $L>0$ with the following property.  Let $\{h_{ij}\}_{1\le i,j\le n}$ be bounded, completely degenerate kernels, let $\{\xi_{i}^{(1)},\xi_{j}^{(2)}\}_{1\le i,j\le n}$ be independent random variables in a measurable space $(S,\mathscr{S})$, and define
\begin{align}
A &\coloneqq \max_{1\le i,j\le n} \|h_{ij}\|_{L^\infty},
\qquad\qquad
C^2 \coloneqq  \sum_{1\le i,j\le n} \E\brak[\big]{h_{ij}^2(\xi_i^{(1)},\xi_j^{(2)})},\\
B^2 &\coloneqq \max_{1\le i,j\le n}\ \max\brac[\bigg]{\Big\|\sum_{i'=1}^n \E\brak[\big]{h_{i'j}^2(\xi_{i'}^{(1)},\cdot)}\Big\|_{L^\infty},\ \Big\|\sum_{j'=1}^n \E\brak[\big]{h_{ij'}^2(\cdot,\xi_{j'}^{(2)})}\Big\|_{L^\infty}},\\
D &\coloneqq \sup\ \E\brak[\Bigg]{\sum_{1\le i,j\le n}h_{ij}(\xi_i^{(1)},\xi_j^{(2)})\, f_i(\xi_i^{(1)})\, g_j(\xi_j^{(2)})},
\end{align}
where the supremum in $D$ runs over all $f_i,g_j:S\to\R$ with $\E[\sum_i f_i^2(\xi_i^{(1)})]\le 1$ and $\E[\sum_j g_j^2(\xi_j^{(2)})]\le 1$.  Then, for every $t>0$,
\begin{align}\label{eq:UstatTailBound}
\mathbb{P}\Bigg(\Big|\sum_{1\le i,j\le n}h_{ij}(\xi_i^{(1)},\xi_j^{(2)})\Big| \ge t\Bigg)
\le L\exp\pa*{-\frac{1}{L}\min\brac*{\pa*{\frac{t}{C}}^{2},\ \frac{t}{D},\ \pa*{\frac{t}{B}}^{2/3},\ \pa*{\frac{t}{A}}^{1/2}}}.
\end{align}
\end{prop}

\begin{remark}\label{rem:Ustattail}
In all applications below, the variables are iid with law $\mu$ and a single symmetric kernel is repeated over the pairs of indices: $h_{ij}=h$ for $i\ne j$, and $h_{ii}=0$.  Exploiting this repetition, the four quantities of \cref{prop:UstatTail} reduce to $N$-weighted norms of $h$:
\begin{align}
A = \|h\|_{L^\infty},\qquad
B^2 = (N-1)\|h\|_{L^\infty L^2(\mu)}^2,\qquad
C^2 =N(N-1)\|h\|_{L^2(\mu^{\otimes2})}^2,\qquad
D \le N\|\Tc_h\|_{\mathrm{op}},
\end{align}
where $\|h\|_{L^\infty L^2(\mu)}\coloneqq\sup_x\pa[\big]{\int_{S} h(x,y)^2\,\dd\mu(y)}^{1/2}$ and $\Tc_hf\coloneqq\int_{S} h(\cdot,y)f(y)\,\dd\mu(y)$, with operator norm $\|\Tc_h\|_{\mathrm{op}}$ on $L^2(\mu)$.  These are the norms entering the master layer-cake estimate below (\cref{prop:masterLayerCake}), which consumes \cref{prop:UstatTail} in the decoupled form \eqref{eq:Ustattail}.
\end{remark}

We recall from de la Pe\~{n}a and Montgomery-Smith \cite[Theorem 1]{dPmS1995} the following decoupling inequality.

\begin{prop}\label{prop:decoup}
For every $k\ge2$ there exists a constant $C_k>0$, depending only on $k$, with the following property.  Let $\{\xi_i\}_{i}$ be a sequence of independent random variables in a measurable space $(S,\mathscr{S})$, let $\{\xi_i^{(j)}\}_{i}$, $1\le j\le k$, be independent copies of $\{\xi_i\}$, and let $f_{i_1\cdots i_k}:S^{k}\to E$ be measurable functions with values in a Banach space $(E,\|\cdot\|)$.  Then, for every $n\ge k$ and every $t>0$,
\begin{multline}\label{eq:decoupTail}
\mathbb{P}\Bigg(\Big\|\sum_{1\le i_1\ne \cdots\ne i_k\le n} f_{i_1\cdots i_k}(\xi_{i_1}^{(1)},\ldots,\xi_{i_k}^{(1)})\Big\| \ge t \Bigg) \\
\le C_k\,\mathbb{P}\Bigg(C_k\Big\|\sum_{1\le i_1\ne \cdots\ne i_k\le n} f_{i_1\cdots i_k}(\xi_{i_1}^{(1)},\ldots,\xi_{i_k}^{(k)})\Big\| \ge t \Bigg),
\end{multline}
the summations ranging over all tuples $(i_1,\ldots,i_k)\in[n]^k$ with distinct entries.
\end{prop}

\begin{remark}[Moment decoupling with an explicit constant]\label{rem:decoupMoment}
For $k=2$ and canonical kernels, the companion moment comparison holds with the explicit constant $4$: if $f_{ij}:S^2\to E$, $1\le i\ne j\le n$, are such that $f_{ij}(\xi_i,\xi_j)$ and $f_{ij}(\xi^{(1)}_i,\xi^{(2)}_j)$ are Bochner integrable and
\begin{align}\label{eq:decoupCanonical}
\E\brak[\big]{f_{ij}(\xi_i,y)}=0
\quad\text{for a.e.\ }y,
\qquad
\E\brak[\big]{f_{ij}(x,\xi_j)}=0
\quad\text{for a.e.\ }x,
\end{align}
then, for every nondecreasing convex $\Phi:[0,\infty)\to[0,\infty)$,
\begin{align}\label{eq:decoupBound}
\E\,\Phi\pa[\Bigg]{\Big\|\sum_{1\le i\ne j\le n}f_{ij}(\xi_i,\xi_j)\Big\|}
\le
\E\,\Phi\pa[\Bigg]{4\,\Big\|\sum_{1\le i\ne j\le n}f_{ij}(\xi_i^{(1)},\xi_j^{(2)})\Big\|}.
\end{align}
This form, with $\Phi(x)=x^{2k}$, is used in the moment computation of \cref{sec:attractive-log-small-beta}.
\end{remark}

\subsection{The master layer-cake estimate}
\label{subsec:master-exponential-moment}

The partition-function bounds of \cref{sec:partition-function-estimate-proof} and the local numerator estimate of \cref{sec:entropic-commutator-estimate-proof} both rest on the same mechanism: a layer-cake decomposition in which, at height $u$, the singular and infrared parts of the interaction are removed by deterministic estimates, and the tail of the remaining centered $U$-statistic is controlled by the following decoupled form of the GLZ bound.  Applying \cref{prop:decoup} with $k=2$ and $E=\R$ and then \cref{prop:UstatTail}, we obtain that if $\xi_1,\ldots,\xi_n$ are iid, then for all $t>0$,
\begin{align}
\mathbb{P}\Bigg(\Big|\sum_{1\le i_1\ne i_2\le n} h(\xi_{i_1},\xi_{i_2})\Big| \ge t\Bigg)
&\le C_2L\exp\pa*{-\frac{1}{L}\min\brac*{\pa*{\frac{t}{C_2C}}^{2},\ \frac{t}{C_2D},\ \pa*{\frac{t}{C_2B}}^{2/3},\ \pa*{\frac{t}{C_2A}}^{1/2}}}  \nn\\
&\le L'\exp\Bigg(-\frac{1}{L'}\min\Bigg\{\pa*{\frac{t}{N\|h\|_{L^2(\mu^{\otimes2})}}}^{2},\ \frac{t}{N\|\Tc_h\|_{\mathrm{op}}},\nn\\
&\qquad\qquad\qquad\qquad\qquad\quad \pa*{\frac{t}{N^{1/2}\|h\|_{L^\infty L^2(\mu)}}}^{2/3},\ \pa*{\frac{t}{\|h\|_{L^\infty}}}^{1/2}\Bigg\}\Bigg), \label{eq:Ustattail}
\end{align}
for a universal constant $L'>0$, where the final line is by \cref{rem:Ustattail} and a relabeling of constants.  Hereafter, we recycle notation and write $L$ in lieu of $L'$.

We state and prove the layer-cake mechanism once; in the statement, $\tl q$ denotes the $\mu$-centering \eqref{eq:tildedef} of a kernel $q$, and all norms of kernels are computed on $S=\supp\mu$.

\begin{prop}[Master layer-cake estimate]\label{prop:masterLayerCake}
Let $\mu\in\P(\R^\d)$, and let $L\ge1$ be the universal constant in \eqref{eq:Ustattail}.  Let $g:(\R^\d)^2\rightarrow\R$ be a symmetric, measurable two-body kernel and $w:\R^\d\to\R$ a measurable one-body function, and consider the statistics
\begin{align}\label{eq:masterXiDef}
\Xi_N(\XN)\coloneqq -\frac{1}{2N}\sum_{1\le i\ne j\le N}\tl g(x_i,x_j)+\frac1N\sum_{i=1}^N w(x_i),
\qquad N\ge1.
\end{align}
Let $U_0\ge1$ be a threshold, and suppose that for every $N\ge1$ and every $u\ge U_0$ there exist a bounded symmetric kernel $q_{N,u}$ and a number $\mathsf m_{N,u}>0$ such that the following hold.
\begin{enumerate}[(M1)]
\item\emph{(Event reduction)}  For $\mu^{\otimes N}$-a.e.\ configuration $\XN$,
\begin{align}\label{eq:masterEventReduction}
\Xi_N(\XN)\ge u
\quad\Longrightarrow\quad
-\frac{1}{2N}\sum_{1\le i\ne j\le N}\tl q_{N,u}(x_i,x_j)\ge u.
\end{align}
\item\emph{(Deterministic diagonal bound)}  For every configuration $\XN$,
\begin{align}\label{eq:masterDiagonal}
-\frac{1}{2N}\sum_{1\le i\ne j\le N}\tl q_{N,u}(x_i,x_j)\le \mathsf m_{N,u}.
\end{align}
\item\emph{(Norm bounds)}
\begin{align}\label{eq:masterNorms}
\|\tl q_{N,u}\|_{L^\infty L^2(\mu)}^2\le \frac1{L},
\qquad
\|\tl q_{N,u}\|_{L^2(\mu^{\otimes2})}^2\le \frac1{L},
\qquad
\|\Tc_{\tl q_{N,u}}\|_{\mathrm{op}}\le\frac1{L}.
\end{align}
\end{enumerate}
Suppose, finally, that there exists $N_0\ge1$ such that
\begin{align}\label{eq:masterLayerCondition}
2L^2\,\sup_{u\ge U_0}\ \min\brac[\big]{u,\ \mathsf m_{N,u}}\,\pa[\big]{1+\|\tl q_{N,u}\|_{L^\infty}}\le N
\qquad\text{for every }N\ge N_0,
\end{align}
the supremum being finite also for $N<N_0$.  Then, writing $\mathsf S_N$ for the supremum in \eqref{eq:masterLayerCondition},
\begin{align}\label{eq:masterConclusion}
\sup_{N\ge1}\ \E_{\mu^{\otimes N}}\brak[\big]{e^{\Xi_N}}
\le \max\brac[\Big]{\max_{1\le N<N_0}e^{U_0+\mathsf S_N},\ e^{U_0}+L},
\end{align}
the inner maximum being absent when $N_0=1$.
\end{prop}

\begin{remark}[Uniformity in families]\label{rem:masterUniform}
The bound \eqref{eq:masterConclusion} depends on the data only through $L$, $U_0$, and the suprema $\mathsf S_N$ in \eqref{eq:masterLayerCondition}, monotonically.  Consequently, if $\{\Xi_N^{(\lambda)}\}_{\lambda\in\Lambda}$ satisfy (M1)--(M3) for each $\lambda\in\Lambda$ with a common threshold $U_0$, and if there are $\lambda$-independent majorants $\mathsf m^{(\lambda)}_{N,u}\le\mathsf m_{N,u}$ and $\|\tl q^{(\lambda)}_{N,u}\|_{L^\infty}\le\mathsf a_{N,u}$ with, for some $N_0\ge1$,
\begin{align}\label{eq:masterUniformBound}
2L^2\,\sup_{u\ge U_0}\ \min\brac[\big]{u,\ \mathsf m_{N,u}}\,\pa[\big]{1+\mathsf a_{N,u}}\le N
\qquad\text{for every }N\ge N_0,
\end{align}
and if, for each $N<N_0$, the corresponding supremum $\mathsf S_N^{(\lambda)}$ is bounded uniformly in $\lambda\in\Lambda$, then $\sup_{\lambda\in\Lambda}\sup_{N\ge1}\E_{\mu^{\otimes N}}\brak[\big]{e^{\Xi_N^{(\lambda)}}}<\infty$.  In \cref{sec:partition-function-estimate-proof}, $\lambda$ is the near-field truncation scale $\eta\in[0,\eta_0]$ in the Riesz case and the infrared cutoff $T\ge2$ in the logarithmic case.
\end{remark}

\begin{proof}[\textup{\textbf{Proof of \cref{prop:masterLayerCake}}}]
Fix $N_0$ as in the statement.  We will use repeatedly the following observation: for every $u\ge U_0$ with $u>\mathsf m_{N,u}$, the event $\{\Xi_N\ge u\}$ is $\mu^{\otimes N}$-null, by (M1) and (M2).

\textbf{The case $N<N_0$ ($L^\infty$ bound).}  By the observation, $\Xi_N$ can reach a height $u\ge U_0$ only if $u\le\mathsf m_{N,u}$, and then $u=\min\{u,\mathsf m_{N,u}\}\le\mathsf S_N$ by the definition of $\mathsf S_N$.  Hence $\Xi_N\le\max\{U_0,\mathsf S_N\}\le U_0+\mathsf S_N$ $\mu^{\otimes N}$-a.e., and
\begin{align}
\E_{\mu^{\otimes N}}\brak[\big]{e^{\Xi_N}}\le e^{U_0+\mathsf S_N},
\qquad N<N_0.
\end{align}

\textbf{The case $N\ge N_0$ (layer-cake bound).}  We claim that
\begin{align}\label{eq:masterTailClaim}
\mu^{\otimes N}\pa{\Xi_N\ge u}\le L\,e^{-2u}
\qquad\text{for every }u\ge U_0.
\end{align}
By the observation, we may assume $u\le\mathsf m_{N,u}$, in which case \eqref{eq:masterLayerCondition} gives the lower bound
\begin{align}\label{eq:masterNLowerBound}
 2L^2\,u\,\pa[\big]{1+\|\tl q_{N,u}\|_{L^\infty}}\le2L^2\,\mathsf S_N\le N.
\end{align}
By (M1) and the decoupled GLZ tail estimate \eqref{eq:Ustattail}, applied to the bounded canonical kernel $\tl q_{N,u}$ at level $t=2Nu$,
\begin{align}
\mu^{\otimes N}\pa{\Xi_N\ge u}
&\le \mu^{\otimes N}\Bigg(\Big|\sum_{1\le i\ne j\le N}\tl q_{N,u}(x_i,x_j)\Big|\ge 2Nu\Bigg)\nn\\
&\le L\exp\pa*{-\frac{1}{L}\min\brac*{
\frac{4u^2}{\|\tl q\|_{L^2(\mu^{\otimes2})}^2},\
\frac{2u}{\|\Tc_{\tl q}\|_{\mathrm{op}}},\
\frac{N^{1/3}(2u)^{2/3}}{\|\tl q\|_{L^\infty L^2(\mu)}^{2/3}},\
\pa*{\frac{2Nu}{\|\tl q\|_{L^\infty}}}^{1/2}}},
\end{align}
where, within this proof, we abbreviate $\tl q=\tl q_{N,u}$.  We check that each of the four entries is at least $2Lu$, so that
\begin{align}\label{eq:masterMinBound}
2Lu\le
\min\brac*{
\frac{4u^2}{\|\tl q\|_{L^2(\mu^{\otimes2})}^2},\
\frac{2u}{\|\Tc_{\tl q}\|_{\mathrm{op}}},\
\frac{N^{1/3}(2u)^{2/3}}{\|\tl q\|_{L^\infty L^2(\mu)}^{2/3}},\
\pa*{\frac{2Nu}{\|\tl q\|_{L^\infty}}}^{1/2}} .
\end{align}
The bounds for the first two entries follow from \eqref{eq:masterNorms} and $1\le U_0\le u$.  The third entry satisfies $2Lu\le N^{1/3}(2u)^{2/3}\|\tl q\|_{L^\infty L^2(\mu)}^{-2/3}$ if and only if $2L^3\|\tl q\|_{L^\infty L^2(\mu)}^2u\le N$, which follows by \eqref{eq:masterNorms} and \eqref{eq:masterNLowerBound}.  The fourth entry satisfies $2Lu\le(2Nu/\|\tl q\|_{L^\infty})^{1/2}$ if and only if $2L^2\|\tl q\|_{L^\infty}u\le N$, which again holds by \eqref{eq:masterNLowerBound}.  This proves \eqref{eq:masterMinBound}, hence the claim \eqref{eq:masterTailClaim}, and the layer-cake formula gives
\begin{align}
\E_{\mu^{\otimes N}}\brak[\big]{e^{\Xi_N}}
\le e^{U_0}+\int_{U_0}^{\infty}e^u\,\mu^{\otimes N}\pa{\Xi_N\ge u}\,\dd u
\le e^{U_0}+L\int_{U_0}^{\infty}e^{-u}\,\dd u
\le e^{U_0}+L,
\qquad N\ge N_0.
\end{align}

Taking the maximum of the two bounds gives \eqref{eq:masterConclusion}.
\end{proof}

\section{Truncation estimates for the log/Riesz kernels}
\label{sec:truncation-estimates}

This section collects the deterministic truncation estimates for the interaction kernels, used in \Cref{sec:partition-function-estimate-proof,sec:entropic-commutator-estimate-proof}: the positive-definite physical-space truncations imported from \cite{HCRS2025} (\cref{subsec:positive-definite-truncations-riesz-kernel}), and the Fourier truncations and deterministic diagonal bounds built on them (\cref{subsec:fourier-truncations-deterministic-diagonal-bounds}).

\subsection{Positive-definite truncations}
\label{subsec:positive-definite-truncations-riesz-kernel}

The following positive-definite truncation package is extracted from the third author's work with Hess-Childs and Serfaty \cite[Subsection 2.1, Lemma 2.2, Proposition 2.9]{HCRS2025}. The lemma is valid for the full range $0<\s<\d$; the stricter condition $\s<\d/2$ imposed in this paper enters only when the singular remainder or the centered kernel is estimated in $L^2$.

In the sequel, every free truncation parameter $\eta$ is taken in $0<\eta\le\eta_0$, where $\eta_0>0$ denotes a fixed small truncation scale, unless explicitly stated otherwise.

\begin{lemma}[Potential truncation]\label{lem:HCRStruncations}
Let $\phi\in\mathcal S(\R^\d)$ be nonnegative, radial, decreasing, and satisfy $\wh{\phi}\ge0$,\footnote{A normalized Gaussian is an admissible choice.} and set $\phi_t(x)\coloneqq t^{-\d}\phi(x/t)$.

For $0<\s<\d$, set
\begin{align}
\mathsf{c}_{\phi,\d,\s}^{-1}
&= \s \int_0^\infty r^\s \phi(r)\,\frac{\dd r}{r},\\
\g(x)
&=\mathsf{c}_{\phi,\d,\s}\int_0^\infty t^{\d-\s}\phi_t(x)\,\frac{\dd t}{t}, \qquad x\ne0.
\end{align}
For $\eta>0$, define the scale-$\eta$ truncated potential and potential difference
\begin{align}\label{eq:HCRStruncDef}
\g_\eta(x)\coloneqq \mathsf{c}_{\phi,\d,\s}\int_\eta^\infty t^{\d-\s}\phi_t(x)\,\frac{\dd t}{t}, \qquad
\f_\eta(x)\coloneqq \g(x)-\g_\eta(x)=\mathsf{c}_{\phi,\d,\s}\int_0^\eta t^{\d-\s}\phi_t(x)\,\frac{\dd t}{t}.
\end{align}
Then, by radial monotonicity,
\begin{align}\label{eq:HCRSdiagPD}
\|\g_\eta\|_{L^\infty}=\g_\eta(0)=C_{\phi,\d,\s}\eta^{-\s}, \qquad \g_\eta\ge0, \qquad \wh{\g_\eta}\ge0, \qquad \wh{\f_\eta}\ge0,
\end{align}
with the Fourier inequalities understood in the sense of tempered distributions.  Since $\g=\g_\eta+\f_\eta$, this implies
\begin{align}\label{eq:HCRSFourierOrder}
0\le \wh{\g_\eta}\le \wh{\g}.
\end{align}
Moreover, the potential difference satisfies
\begin{align}\label{eq:HCRSfetaDecay}
0\le \f_\eta(x)&\le C_\gamma \frac{\eta^{\gamma-\s}}{|x|^\gamma}, \qquad \gamma>\s,\nn\\
\f_\eta(x)&\ge c\g(x),\qquad |x|\le\eta.
\end{align}
In particular, if $\mu\in L^\infty$, then
\begin{align}\label{eq:HCRSconvBound}
\|\f_\eta\ast \mu\|_{L^\infty}
\le \|\mu\|_{L^\infty}\int_{\R^\d}\f_\eta(x)\,\dd x
\le C_\mu \eta^{\d-\s}.
\end{align}

For $\s=0$, let
\begin{align}\label{eq:logCphiTDef}
\mathsf{c}_{\phi,\d,0}\coloneqq \phi(0)^{-1},\qquad
C_{\phi,T}\coloneqq \phi(0)\log T-\int_0^\infty \log r\,\phi'(r)\,\dd r,
\end{align}
set, for $0<\eta<T$,
\begin{align}\label{eq:logFiniteTTruncDef}
\g_{\eta,T}(x)\coloneqq \mathsf{c}_{\phi,\d,0}\pa*{\int_\eta^T t^\d\phi_t(x)\,\frac{\dd t}{t}-C_{\phi,T}},\qquad
\f_\eta(x)\coloneqq \g(x)-\g_\eta(x)=\mathsf{c}_{\phi,\d,0}\int_0^\eta t^\d\phi_t(x)\,\frac{\dd t}{t},
\end{align}
where $\g_\eta$ is the limit of $\g_{\eta,T}$ as $T\to\infty$ in the sense of tempered distributions.  Then
\begin{align}\label{eq:logTruncPropsBasic}
\g_\eta(0)=-\log\eta+O(1),\qquad \wh{\g_\eta}\ge0\quad\hbox{on zero-mean test functions},\qquad \wh{\f_\eta}\ge0,
\end{align}
\begin{align}\label{eq:logfetaDecay}
0\le \f_\eta(x)\le C_\gamma\min\Bigg\{1+\pa*{-\log\frac{|x|}{\eta}}_+,\frac{\eta^\gamma}{|x|^\gamma}\Bigg\},\qquad \gamma>0,
\end{align}
and consequently, if $\mu\in L^\infty$,
\begin{align}\label{eq:logfetaConv}
\|\f_\eta\ast\mu\|_{L^\infty}\le C_\mu\eta^\d.
\end{align}
For every fixed $n\ge1$, it also holds that
\begin{align}\label{eq:logfetaDerivative}
|x|^n|\nab^{\otimes n}\f_\eta(x)|\le C_n\f_{c_n\eta}(x)
\end{align}
for a dimensional constant $c_n\ge2$.
\end{lemma}

\subsection{Fourier truncations and positive-definite diagonal bounds}
\label{subsec:fourier-truncations-deterministic-diagonal-bounds}

We use here the truncation facts already recorded in
\eqref{eq:HCRSdiagPD}-\eqref{eq:HCRSconvBound}: positive definiteness and
Fourier order, the diagonal and $L^\infty$ bounds, and the short-range convolution estimate.
For the Riesz kernel, with our normalization of Fourier transform,
\begin{align}
\wh{\g}(\xi)=c_{\d,\s}|\xi|^{-(\d-\s)}, \qquad \xi\in \R^\d\setminus \brac{0},
\end{align}
for a positive constant $c_{\d,\s}$.  Given $K\ge1$, define the low- and high-frequency pieces of $\g_\eta$ by
\begin{align}\label{eq:FourierLowHighDef}
\wh{\g_{\eta,\le K}}(\xi) \coloneqq \indic_{|\xi|\le K}\wh{\g_\eta}(\xi), \qquad \wh{\g_{\eta,>K}}(\xi) \coloneqq \indic_{|\xi|>K}\wh{\g_\eta}(\xi).
\end{align}
By \eqref{eq:HCRSdiagPD}, both kernels are positive definite.

We record here the elementary deterministic estimate which will later give the improved diagonal bound.
\begin{lemma}[Positive-definite diagonal bound]\label{lem:PDdeterministicDiagonalBound}
Let $q$ be a bounded positive-definite kernel, and let $\tl{q}$ denote its $\mu$-centering with respect to a probability measure $\mu$.  Then
\begin{align}\label{eq:PDcenteredLinfty}
\|\tl q\|_{L^\infty}\le 4\sup_x q(x,x),
\end{align}
and, for every $N\ge1$ and every configuration $\XN=(x_1,\ldots,x_N)$,
\begin{align}
\sum_{1\le i,j\le N}\tl{q}(x_i,x_j)\ge0.
\end{align}
Consequently,
\begin{align}\label{eq:PDdeterministicDiagonalBound}
-\frac{1}{N}\sum_{1\le i\ne j\le N}\tl{q}(x_i,x_j)
\le \frac{1}{N}\sum_{i=1}^N \tl{q}(x_i,x_i)
\le \sup_x \tl{q}(x,x)
\le 4\sup_x q(x,x).
\end{align}
\end{lemma}
\begin{proof}
Since the two-point Gram matrix associated to $q$ is positive semidefinite, the nonnegativity of its determinant gives $|q(x,y)|^2 \le q(x,x)q(y,y)$, so $\|q\|_{L^\infty}\le M\coloneqq\sup_x q(x,x)$.  Expanding the centering identity \eqref{eq:tildedef} then bounds each of the four terms of $\tl q(x,y)$ by $M$, which proves \eqref{eq:PDcenteredLinfty} and the last inequality in \eqref{eq:PDdeterministicDiagonalBound}.

Next, by \eqref{eq:tildedef} and the definition \eqref{eq:muNdef} of $\mu_N$,
\begin{align}
\sum_{1\le i,j\le N}\tl{q}(x_i,x_j)
=N^2\iint_{(\R^\d)^2} q(x,y)\,\dd(\mu_N-\mu)(x)\,\dd(\mu_N-\mu)(y),
\end{align}
which is nonnegative by positive definiteness of $q$.  This gives
\begin{align}
-\frac1N\sum_{1\le i\ne j\le N}\tl{q}(x_i,x_j)
\le \frac1N\sum_{i=1}^N\tl{q}(x_i,x_i)
\le \sup_x\tl{q}(x,x).
\end{align}
\end{proof}

For the logarithmic proofs below, we also need a finite-$T$ frequency decomposition.  The point is to keep the one-body term in $-\be N\Fr_N$ until after the infrared part of the interaction has been estimated.  Constants may depend on $\d$, on the fixed truncation profile $\phi$, and on the displayed bounds for $\mu$, but never on $N$, $u$, $\eta$, or the auxiliary infrared cutoff $T$.

Fix once and for all a radial, radially nonincreasing function $\chi\in C_c^\infty(\R^\d)$ such that $0\le\chi\le1$, $\chi=1$ on $\{|\xi|\le1\}$, and $\chi=0$ on $\{|\xi|\ge2\}$, and set $\chi_K(\xi)\coloneqq\chi(\xi/K)$.  For $0<\eta<T$, we assign the renormalizing constant in \eqref{eq:logFiniteTTruncDef} to the low-frequency part and define
\begin{align}\label{eq:logFreqDef}
\wh{\g}_{\eta,T,\le K}
&\coloneqq \chi_K(\xi)\,\mathsf{c}_{\phi,\d,0}\int_\eta^T t^\d\wh\phi(t\xi)\,\frac{\dd t}{t}-\mathsf{c}_{\phi,\d,0}C_{\phi,T}\delta_0,\\
\wh{\g}_{\eta,T,>K}
&\coloneqq (1-\chi_K(\xi))\,\mathsf{c}_{\phi,\d,0}\int_\eta^T t^\d\wh\phi(t\xi)\,\frac{\dd t}{t}.
\label{eq:logFreqDefHigh}
\end{align}
Thus, $\g_{\eta,T}=\g_{\eta,T,\le K}+\g_{\eta,T,>K}$.  Since the Dirac mass at the origin is invisible to zero-average test functions, both pieces are positive definite in every centered quadratic form in which they are used below.  We also write
\begin{align}
\g_T\coloneqq \f_\eta+\g_{\eta,T},\qquad h_{\mu,T}\coloneqq \g_T\ast\mu,
\end{align}
noting that $\g_T$ is independent of the auxiliary choice of $\eta$.

\begin{lemma}[Frequency-truncation estimates]\label{lem:logLargeScaleCancellation}
Let $K\ge1$, and recall the $L^\infty L^2(\mu)$ and integral-operator norms of \cref{rem:Ustattail}.
\begin{enumerate}[(i)]
\item\emph{(Riesz case $0<\s<\d/2$.)}  The centered kernels of the splitting \eqref{eq:FourierLowHighDef} satisfy, uniformly in $0<\eta\le\eta_0$,
\begin{align}
\|\tl\g_{\eta,>K}\|_{L^2(\mu^{\otimes2})}^2 &\le C_\mu \int_{|\xi|>K}\wh{\g}(\xi)^2\,\dd\xi \le C_\mu K^{2\s-\d},\label{eq:CtailFourier}\\
\|\tl\g_{\eta,>K}\|_{L^\infty L^2(\mu)}^2 &\le C_\mu \int_{|\xi|>K}\wh{\g}(\xi)^2\,\dd\xi \le C_\mu K^{2\s-\d},\label{eq:BtailFourier}\\
\|\Tc_{\tl\g_{\eta,>K}}\|_{\mathrm{op}} &\le C_\mu \sup_{|\xi|>K}\wh{\g}(\xi) \le C_\mu K^{-(\d-\s)},\label{eq:DtailFourier}
\end{align}
together with the low-mode diagonal bound
\begin{align}\label{eq:lowModeDiagBound}
\sup_x \tl{\g}_{\eta,\le K}(x,x) \le 4\int_{|\xi|\le K}\wh{\g}(\xi)\,\dd\xi \le C_1 K^\s,
\end{align}
for a constant $C_1=C_1(\d,\s)$.
\item\emph{(Logarithmic case $\s=0$.)}  Assume \eqref{eq:logEnergyAssumption}.  There is a constant $C_{\mu,K}<\infty$ such that, for every $N\ge1$, every $T\ge2$, every $0<\eta\le\eta_0$, and every configuration $\XN$,
\begin{align}\label{eq:logLargeScaleCancellation}
\frac1N\sum_{i=1}^N h_{\mu,T}(x_i)
-\frac1{2N}\sum_{1\le i\ne j\le N}\tl\g_{\eta,T,\le K}(x_i,x_j)
\le C_{\mu,K}.
\end{align}
Moreover,
\begin{align}\label{eq:logHighFreqNorms}
\|\tl\g_{\eta,T,>K}\|_{L^2(\mu^{\otimes2})}^2+\|\tl\g_{\eta,T,>K}\|_{L^\infty L^2(\mu)}^2&\le C_\mu K^{-\d},
\qquad
\|\Tc_{\tl\g_{\eta,T,>K}}\|_{\mathrm{op}}\le C_\mu K^{-\d},\nn\\
\|\tl\g_{\eta,T,>K}\|_{L^\infty}&\le C\pa[\big]{1+\log(1/\eta)}.
\end{align}
\end{enumerate}
\end{lemma}

\begin{proof}
\emph{(i)}  The bounds \eqref{eq:CtailFourier}--\eqref{eq:DtailFourier} follow from Plancherel, the Fourier order \eqref{eq:HCRSFourierOrder}, and $\mu\in L^\infty$; the tail integrals are finite exactly because $2\s<\d$.  For \eqref{eq:lowModeDiagBound}: since $|\phi_{\xi,\mu}|\le2$ for the centered characters $\phi_{\xi,\mu}(x)\coloneqq e^{2\pi\iu x\cdot\xi}-\int_{\R^\d} e^{2\pi\iu y\cdot\xi}\,\dd\mu(y)$, the representation
\begin{align}
\tl{\g}_{\eta,\le K}(x,y)=\int_{|\xi|\le K}\wh{\g_{\eta}}(\xi)\,\phi_{\xi,\mu}(x)\ol{\phi_{\xi,\mu}(y)}\,\dd\xi
\end{align}
gives $\sup_x\tl\g_{\eta,\le K}(x,x)\le4\int_{|\xi|\le K}\wh{\g}(\xi)\,\dd\xi\le C_1K^{\s}$, using \eqref{eq:HCRSFourierOrder} again.

\emph{(ii)}  Let $q\coloneqq\g_{\eta,T,\le K}$.  Since $q$ is conditionally positive definite,

\begin{align}\label{eq:logCancellationAlgebra}
\frac1N\sum_{i=1}^N h_{\mu,T}(x_i)-\frac1{2N}\sum_{1\le i\ne j\le N}\tl q(x_i,x_j)
\le \frac1N\sum_{i=1}^N\Big(h_{\mu,T}(x_i)+\frac12\tl q(x_i,x_i)\Big).
\end{align}
Expanding the centered diagonal gives
\begin{align}\label{eq:logDiagonalExpand}
h_{\mu,T}(x)+\frac12\tl q(x,x)
=(\g_T-q)\ast\mu(x)+\frac12q(0)+\frac12\int_{(\R^\d)^2}q(z-w)\,\dd\mu^{\otimes2}(z,w).
\end{align}
\begingroup
Here, $\g_T-q=\f_\eta+\g_{\eta,T,>K}$.  The contribution of $\f_\eta$ is bounded in $L^\infty$ by \eqref{eq:logfetaConv}.  For the high-frequency term, the multiplier in \eqref{eq:logFreqDefHigh} is smooth at $\xi=0$ and satisfies
\begin{align}\label{eq:logHighSymbolBounds}
|\partial_\xi^\alpha \wh\g_{\eta,T,>K}(\xi)|
\le C_{\alpha,K}|\xi|^{-\d-|\alpha|},
\qquad \xi\ne0,
\end{align}
for $|\alpha|\le\d+2$, uniformly in $T$ and $\eta$.  A smooth dyadic decomposition of $\{|\xi|\ge K\}$, followed by rescaling and repeated integration by parts on each annulus, therefore gives the absolute kernel bound
\begin{align}\label{eq:logHighKernelMajorant}
|\g_{\eta,T,>K}(x)|
\le C_K\Big(1+(-\log|x|)_+\Big)\indic_{|x|\le1}
+C_K\langle x\rangle^{-\d-2}.
\end{align}
Since $\mu\in L^1\cap L^\infty$, this implies, uniformly in $T\ge2$ and $0<\eta\le\eta_0$, that
\begin{align}
\|\g_{\eta,T,>K}\ast\mu\|_{L^\infty}\le C_{\mu,K}.
\end{align}

For the scalar terms in \eqref{eq:logDiagonalExpand}, the nonconstant Fourier multiplier defining $q$ is real, radial, and nonnegative.  Consequently, the nonconstant part of $q$ is a real translation-invariant positive-definite kernel and is therefore maximized at the origin.  Since subtracting the constant term in \eqref{eq:logFreqDef} does not affect this comparison, $q(x)\le q(0)$ for every $x\in\R^\d$, and hence
\begin{align}
\int_{(\R^\d)^2}q(z-w)\,\dd\mu^{\otimes2}(z,w)\le q(0).
\end{align}
Set
\begin{align}
F_K(t)\coloneqq\int_{\R^\d}\chi_K(\xi)t^\d\wh\phi(t\xi)\,\dd\xi
=\int_{\R^\d}\chi_K(z/t)\wh\phi(z)\,\dd z.
\end{align}
Fourier inversion and \eqref{eq:logCphiTDef} give
\begin{align}
q(0)
=\mathsf c_{\phi,\d,0}\left(
\int_\eta^T F_K(t)\,\frac{\dd t}{t}
-\phi(0)\log T+\int_0^\infty\log r\,\phi'(r)\,\dd r
\right).
\end{align}
Here $0\le F_K(t)\le\phi(0)$ for every $t>0$, while $F_K(t)\le C_Kt^\d$ for $0<t\le1$.  Splitting at $t=1$ if $\eta\le1$, and using $\int_\eta^T F_K(t)\,\dd t/t\le\phi(0)\log(T/\eta)$ if $\eta>1$, gives $q(0)\le C_K$ uniformly in $T$ and $\eta$.  The preceding pointwise bound therefore yields
\begin{align}
\frac12q(0)
+\frac12\int_{(\R^\d)^2}q(z-w)\,\dd\mu^{\otimes2}(z,w)
\le C_K
\end{align}
uniformly in $T$ and $\eta$.
\endgroup
Combining these bounds with \eqref{eq:logCancellationAlgebra}--\eqref{eq:logDiagonalExpand} proves \eqref{eq:logLargeScaleCancellation}.

It remains to record the $U$-statistic norms.  From \eqref{eq:logFreqDef}, the bounds $0\le1-\chi_K\le1$, the support condition $1-\chi_K=0$ on $\{|\xi|\le K\}$, and the change of variables $\zeta=t\xi$, we have
\begin{align}\label{eq:logMultiplierOrder}
0\le \wh\g_{\eta,T,>K}(\xi)\le C|\xi|^{-\d}\indic_{|\xi|\ge K}.
\end{align}
Plancherel and $\mu\in L^\infty$ give
\begin{align}
\|\tl\g_{\eta,T,>K}\|_{L^2(\mu^{\otimes2})}^2+\|\tl\g_{\eta,T,>K}\|_{L^\infty L^2(\mu)}^2
\le C_\mu\int_{|\xi|\ge K}|\xi|^{-2\d}\,\dd\xi\le C_\mu K^{-\d},
\end{align}
and the operator norm is bounded by the same Fourier multiplier estimate,
\begin{align}
\|\Tc_{\tl\g_{\eta,T,>K}}\|_{\mathrm{op}}\le C_\mu\sup_{|\xi|\ge K}|\xi|^{-\d}\le C_\mu K^{-\d}.
\end{align}
Finally, by positive definiteness and \eqref{eq:logTruncPropsBasic},
\begin{align}
\|\tl\g_{\eta,T,>K}\|_{L^\infty}\le C\g_\eta(0)\le C(1+\log(1/\eta)).
\end{align}
This completes the proof.
\end{proof}

\section{Repulsive partition-function estimates}
\label{sec:partition-function-estimate-proof}

This section proves \cref{thm:partitionFunctionBound} and \cref{cor:partitionFunctionTemperatureInterpolation}.  The proof separates into three independent estimates, which we state now as a roadmap.  For $\xi_1,\ldots,\xi_N$ iid with law $\mu$, define the centered two-body statistic and the one-body average
\begin{align}\label{eq:WNdefRate}
W_N\coloneqq \frac1N\sum_{1\le i\ne j\le N}\tl\g(\xi_i,\xi_j),
\qquad
R_N\coloneqq\frac1N\sum_{i=1}^N\psi_\mu(\xi_i),
\qquad
\psi_\mu\coloneqq h_\mu-2I_\mu.
\end{align}
The exact identity
\begin{align}\label{eq:roadmapIdentity}
\K_{N,\be}(\mu)
=e^{\be I_\mu}\,\E_{\mu^{\otimes N}}\Big[e^{-\frac{\be}{2}W_N+\be R_N}\Big]
\end{align}
splits the modulated partition function into a centered two-body part and a one-body remainder.  The proof of \cref{thm:partitionFunctionBound} then consists of:
\begin{enumerate}[(a)]
\item a \emph{uniform exponential-moment bound} for the centered two-body statistics, proved in \cref{subsec:uniform-bound-log-riesz} by verifying, separately for the Riesz and logarithmic kernels, the hypotheses of the master layer-cake estimate of \cref{subsec:master-exponential-moment} (\cref{prop:masterLayerCake});
\item a \emph{regularized-determinant limit with rate} for $\E[e^{-\frac\be2 W_N}]$, proved in \cref{subsec:regularized-determinant-limit} by Hilbert--Schmidt truncation, with the uniform bound of part (a) as the key input;
\item a \emph{one-body reduction} showing that $R_N$ contributes an $O(N^{-1/2})$ error, proved in \cref{subsec:one-body-reduction}.
\end{enumerate}
These are assembled into the proof of \cref{thm:partitionFunctionBound} at the end of \cref{subsec:one-body-reduction}.  The temperature-interpolation corollary advertised in the introduction is proved in \cref{subsec:partition-function-temperature-interpolation}, and the sharpness of the threshold $\s<\d/2$ in \cref{subsec:partition-function-sharpness}.  The attractive logarithmic case is proved in \cref{sec:attractive-log-small-beta}, immediately following this section.

\subsection{Uniform bound for the repulsive log/Riesz kernels}
\label{subsec:uniform-bound-log-riesz}
\label{subsec:uniform-bound-positive-riesz}
\label{subsec:uniform-bound-logarithmic}

We now verify the hypotheses of \cref{prop:masterLayerCake} for the repulsive kernels \eqref{eq:gDef}.  Throughout this subsection, $\be>0$ is fixed (the case $\be=0$ of the statements below being trivial), $\mu\in L^\infty\cap\P(\R^\d)$ under the standing density convention, and all constants may depend on $\be$, $\d$, $\s$, and the relevant bounds for $\mu$, but never on $N$, $u$, $\eta$, $t$, or the infrared cutoff $T$.  For $0\le\eta\le\eta_0$, let $W_{N,\eta}$ denote the statistic \eqref{eq:WNdefRate} with $\g$ replaced by the truncation $\g_\eta$ of \cref{lem:HCRStruncations}, with the convention $\g_0\coloneqq\g$, so that $W_{N,0}=W_N$.

\begin{prop}[Uniform exponential-moment bounds]\label{prop:uniformBoundRepulsive}
Assume $0\le\s<\d/2$, $\mu\in L^\infty\cap\P(\R^\d)$, and $\be\ge0$ is fixed.  If $\s=0$, assume in addition \eqref{eq:logEnergyAssumption}.  Then
\begin{align}\label{eq:CubDef}
\sup_{N\ge1}\K_{N,\be}(\mu)\le C_{\mathrm{ub}}(\be,\mu)<\infty.
\end{align}
If moreover, in the case $\s=0$, $\int_{\R^\d}e^{-2\be h_\mu}\,\dd\mu<\infty$, then also
\begin{align}\label{eq:centeredFamilyExpMoment}
\sup_{N\ge1}\ \sup_{0\le\eta\le\eta_0}\ \E_{\mu^{\otimes N}}\brak*{e^{-\frac{\be}{2}W_{N,\eta}}}\le C_{\mathrm{cf}}(\be,\mu)<\infty.
\end{align}
\end{prop}

The proof occupies the remainder of this subsection: the case $0<\s<\d/2$ is treated first, then the case $\s=0$.  In each case we specify the data $(g,w)$ in \eqref{eq:masterXiDef}, the threshold $U_0$, and the kernels $q_{N,u}$, verify (M1)--(M3), and check the growth condition \eqref{eq:masterLayerCondition}; throughout, we write $\mathsf S_N$ for the supremum appearing there.  In both cases the cap $\mathsf m_{N,u}$ will be an explicit number, at least $1$, that also bounds the $L^\infty$ norm, $\|\tl q_{N,u}\|_{L^\infty}\le\mathsf m_{N,u}$; hence $1+\|\tl q_{N,u}\|_{L^\infty}\le2\mathsf m_{N,u}$ and $\mathsf S_N\le2\sup_{u\ge U_0}\min\{u,2\mathsf m_{N,u}\}\,\mathsf m_{N,u}$.

\subsubsection{The case \texorpdfstring{$0<\s<\d/2$}{0<s<d/2}}

\begin{proof}[Proof of \cref{prop:uniformBoundRepulsive}, case $0<\s<\d/2$: \eqref{eq:CubDef} and \eqref{eq:centeredFamilyExpMoment}]
Throughout the proof we assume, as we may, $\eta_0\le1$.

\step{Step 1: Reduction to the centered statistics.}  Both \eqref{eq:CubDef} and \eqref{eq:centeredFamilyExpMoment} follow from the family bound
\begin{align}\label{eq:fullSingularExpGoal}
\sup_{N\ge1}\ \sup_{0\le\eta\le\eta_0}\ \E_{\XN\sim\mu^{\otimes N}}\brak*{e^{\Xi_N^{(\eta)}}}\le C_{\be,\mu},
\qquad
\Xi_N^{(\eta)}\coloneqq -\frac{\be}{2}W_{N,\eta}
= -\frac{\be}{2N}\sum_{1\le i\ne j\le N}\tl\g_\eta(x_i,x_j).
\end{align}

\par\smallskip\noindent  Since $\mu\in L^\infty$ and $0<\s<\d$, splitting the integral at unit distance gives $h_\mu=\g\ast\mu\in L^\infty$ with $\|h_\mu\|_{L^\infty}\le C(1+\|\mu\|_{L^\infty})$.  By \eqref{eq:FNdef},
\begin{align}
-\be N\Fr_N(\XN,\mu)
= -\frac{\be}{2N}\sum_{1\le i\ne j\le N}\tl\g(x_i,x_j)+\frac{\be}{N}\sum_{i=1}^Nh_\mu(x_i)-\be I_\mu
\le -\frac{\be}{2}W_N+\be\|h_\mu\|_{L^\infty},
\end{align}
using $I_\mu\ge0$; hence $\K_{N,\be}(\mu)\le e^{\be\|h_\mu\|_{L^\infty}}\E[e^{-\frac\be2W_N}]$, so the case $\eta=0$ of \eqref{eq:fullSingularExpGoal} gives \eqref{eq:CubDef}, while \eqref{eq:fullSingularExpGoal} is precisely \eqref{eq:centeredFamilyExpMoment}.

\step{Step 2: Choice of the data.}
With $C_0$ the near-field constant in \eqref{eq:HCRSconvBound}, $C_\mu$ and $C_1$ the constants in \eqref{eq:CtailFourier}--\eqref{eq:DtailFourier} and \eqref{eq:lowModeDiagBound} of \cref{lem:logLargeScaleCancellation}(i), $C_2$ the diagonal constant in \eqref{eq:HCRSdiagPD}, and $C_4\coloneqq\max\brac[\big]{8C_2,\ \be^{-1}}$, define the frequency cutoff and the threshold
\begin{align}
K_{\be,\mu}&\coloneqq\max\brac[\Big]{\pa[\big]{4L\,\be^2C_\mu}^{\frac1{\d-2\s}},\ \pa[\big]{4L\,\be C_\mu}^{\frac1{\d-\s}},\ 1},\label{eq:chooseKbeta}\\
U_0&\coloneqq \frac{3\be}{2}\,C_1K_{\be,\mu}^{\s}+1,\label{eq:UbetaRieszDef}
\end{align}
and, for $N\ge1$, $u\ge U_0$, and $\eta\in[0,\eta_0]$, the adaptive physical scales, kernels, and caps
\begin{align}
\delta_{N,u}&\coloneqq \min\Bigg\{\eta_0,\ \pa*{\frac{u}{8C_0\be N}}^{\frac{1}{\d-\s}}\Bigg\},
\qquad
\bar\delta\coloneqq\max\{\delta_{N,u},\eta\},\label{eq:etaNuNoSmallBeta}\\
q_{N,u}&\coloneqq 2\be\,\g_{\bar\delta,>K_{\be,\mu}},
\qquad
\mathsf m_{N,u}\coloneqq C_4\,\be\,\delta_{N,u}^{-\s},\label{eq:rieszKernelDef}
\end{align}
the hypotheses of \cref{prop:masterLayerCake} hold for each fixed $\eta\in[0,\eta_0]$, with $g=\be\g_\eta$ and $w=0$ (so that $\Xi_N=\Xi_N^{(\eta)}$), and the majorant condition \eqref{eq:masterUniformBound} of \cref{rem:masterUniform} holds uniformly in $\eta\in[0,\eta_0]$.

\par\smallskip\noindent  Fix $\eta\in[0,\eta_0]$.  We use the three-scale decomposition, valid since $\bar\delta\ge\eta$,
\begin{align}\label{eq:threeScaleDecomposition}
\g_\eta=\pa[\big]{\f_{\bar\delta}-\f_\eta}+\g_{\bar\delta,\le K_{\be,\mu}}+\g_{\bar\delta,>K_{\be,\mu}},
\end{align}
where $\f_{\bar\delta}-\f_\eta=\g_\eta-\g_{\bar\delta}\ge0$ by \eqref{eq:HCRStruncDef} (and $\f_0\coloneqq0$).

\textbf{(M1).}  For the near-field part, we distinguish the two cases of $\bar\delta=\max\{\delta_{N,u},\eta\}$.  If $\eta\ge\delta_{N,u}$, then $\bar\delta=\eta$, so $\f_{\bar\delta}-\f_\eta=0$ and the near-field part vanishes.  If instead $\eta<\delta_{N,u}$, then $\bar\delta=\delta_{N,u}$, and the nonnegativity of $\f_{\delta_{N,u}}-\f_\eta$, the centering identity \eqref{eq:tildedef}, and \eqref{eq:HCRSconvBound} at scale $\delta_{N,u}$ give, for all $x,y$,
\begin{align}
-\pa[\big]{\tl\f_{\delta_{N,u}}-\tl\f_\eta}(x,y)
\le \pa[\big]{\f_{\delta_{N,u}}-\f_\eta}\ast\mu(x)+\pa[\big]{\f_{\delta_{N,u}}-\f_\eta}\ast\mu(y)
\le 2C_0\delta_{N,u}^{\d-\s}.
\end{align}
In both cases, for every configuration,
\begin{align}\label{eq:nearFieldAdaptiveNoSmallBeta}
-\frac{\be}{2N}\sum_{1\le i\ne j\le N}\pa[\big]{\tl\f_{\bar\delta}-\tl\f_\eta}(x_i,x_j)
\le C_0\be N\delta_{N,u}^{\d-\s}\le \frac{u}{8},
\end{align}
the last inequality by the definition \eqref{eq:etaNuNoSmallBeta} of $\delta_{N,u}$.
For the low-frequency part, $\g_{\bar\delta,\le K}$ is positive definite by \eqref{eq:FourierLowHighDef}, so \cref{lem:PDdeterministicDiagonalBound} and \eqref{eq:lowModeDiagBound} give, for every configuration,
\begin{align}\label{eq:lowModeDeterministic}
-\frac{\be}{2N}\sum_{1\le i\ne j\le N}\tl{\g}_{\bar\delta,\le K_{\be,\mu}}(x_i,x_j) \le \frac{\be}{2}C_1 K_{\be,\mu}^\s\le \frac{u}{3},
\end{align}
the last inequality by \eqref{eq:UbetaRieszDef} and $u\ge U_0$.  Combining \eqref{eq:threeScaleDecomposition}--\eqref{eq:lowModeDeterministic}, on the event $\{\Xi_N^{(\eta)}\ge u\}$,
\begin{align}
-\frac{1}{2N}\sum_{1\le i\ne j\le N}\tl q_{N,u}(x_i,x_j)
\ge 2\pa*{u-\frac u8-\frac u3}=\frac{13u}{12}\ge u,
\end{align}
which is (M1).

\textbf{(M2) and (M3).}  By \eqref{eq:CtailFourier}--\eqref{eq:DtailFourier}, the norms of $\tl q_{N,u}=2\be\tl\g_{\bar\delta,>K_{\be,\mu}}$ satisfy, using the choice \eqref{eq:chooseKbeta} of $K_{\be,\mu}$,
\begin{align}
\|\tl q_{N,u}\|_{L^\infty L^2(\mu)}^2,\ \|\tl q_{N,u}\|_{L^2(\mu^{\otimes2})}^2\le 4\be^2C_\mu K_{\be,\mu}^{2\s-\d}\le\frac1{L},
\qquad
\|\Tc_{\tl q_{N,u}}\|_{\mathrm{op}}\le 2\be C_\mu K_{\be,\mu}^{-(\d-\s)}\le\frac1{2L}.
\end{align}
This is (M3).  Moreover, since $\g_{\bar\delta,>K}$ is positive definite with $\g_{\bar\delta,>K}(0)\le\g_{\bar\delta}(0)\le C_2\bar\delta^{-\s}\le C_2\delta_{N,u}^{-\s}$, the $L^\infty$ bound \eqref{eq:PDcenteredLinfty} gives $\|\tl\g_{\bar\delta,>K}\|_{L^\infty}\le4\,\g_{\bar\delta,>K}(0)$, hence
\begin{align}\label{eq:diagonalEtaConstraint}
\|\tl q_{N,u}\|_{L^\infty}\le 8\be\,\g_{\bar\delta,>K}(0)\le 8C_2\,\be\,\delta_{N,u}^{-\s}\le\mathsf m_{N,u},
\end{align}
since $C_4\ge8C_2$; note also that $\mathsf m_{N,u}\ge C_4\be\ge1$, since $\delta_{N,u}\le\eta_0\le1$.  Since $q_{N,u}$ is positive definite, \cref{lem:PDdeterministicDiagonalBound} gives $-\frac{1}{2N}\sum_{1\le i\ne j\le N}\tl q_{N,u}(x_i,x_j)\le\frac12\|\tl q_{N,u}\|_{L^\infty}\le\mathsf m_{N,u}$ for every configuration, which is (M2).  Consequently, only layer heights $u\le \mathsf m_{N,u}=C_4\be\,\delta_{N,u}^{-\s}$ contribute.  Combined with the definition \eqref{eq:etaNuNoSmallBeta} of $\delta_{N,u}$, this constraint gives, after adjusting $C_{\be,\mu}$,
\begin{align}\label{eq:diagonalNConstraint}
U_0\le u\le 2\mathsf m_{N,u}
\quad\Longrightarrow\quad
u\le C_{\be,\mu}\,N^{\s/\d}.
\end{align}
Indeed, if $\delta_{N,u}=\eta_0$ this is immediate.  If instead the second entry of the minimum in \eqref{eq:etaNuNoSmallBeta} is active, then $u\le 2C_4\be(8C_0\be N/u)^{\s/(\d-\s)}$, that is, $u^{\d/(\d-\s)}\le C_{\be,\mu}N^{\s/(\d-\s)}$.

\textbf{(Sublinear growth \eqref{eq:masterLayerCondition}).}  This is the second place where $\s<\d/2$ enters.  We claim that
\begin{align}\label{eq:allTermsBeatNoSmallBeta}
\mathsf S_N\le 2\sup_{u\ge U_0}\ \min\brac[\big]{u,2\mathsf m_{N,u}}\,\mathsf m_{N,u}\le C_{\be,\mu}\,N^{2\s/\d},
\qquad\text{uniformly in }\eta\in[0,\eta_0].
\end{align}
Fix $u\ge U_0$.  If $\delta_{N,u}=\eta_0$, then $\mathsf m_{N,u}=C_4\be\,\eta_0^{-\s}=O_{\be,\mu}(1)$, so $\min\{u,2\mathsf m_{N,u}\}\,\mathsf m_{N,u}\le2\mathsf m_{N,u}^2\le C_{\be,\mu}$.  Suppose instead $\delta_{N,u}=(u/8C_0\be N)^{1/(\d-\s)}$.  If $u\le2\mathsf m_{N,u}$, then, using \eqref{eq:diagonalEtaConstraint} and then \eqref{eq:diagonalNConstraint},
\begin{align}
\min\brac[\big]{u,2\mathsf m_{N,u}}\,\mathsf m_{N,u}
\le u\,\mathsf m_{N,u}
&= C_4\be\,u\pa*{\frac{8C_0\be N}{u}}^{\frac{\s}{\d-\s}}
= C_{\be,\mu}\,u^{\frac{\d-2\s}{\d-\s}}\,N^{\frac{\s}{\d-\s}}\nn\\
&\le C_{\be,\mu}\pa[\big]{N^{\s/\d}}^{\frac{\d-2\s}{\d-\s}}\,N^{\frac{\s}{\d-\s}}
= C_{\be,\mu}\,N^{2\s/\d} .
\end{align}
If instead $u>2\mathsf m_{N,u}$, then the algebra of \eqref{eq:diagonalNConstraint}, run in reverse, gives $u\ge c_{\be,\mu}\,N^{\s/\d}$, so
\begin{align}
\min\brac[\big]{u,2\mathsf m_{N,u}}\,\mathsf m_{N,u}
\le 2\mathsf m_{N,u}^2
= C_{\be,\mu}\pa*{\frac{N}{u}}^{\frac{2\s}{\d-\s}}
\le C_{\be,\mu}\pa[\big]{N^{1-\s/\d}}^{\frac{2\s}{\d-\s}}
= C_{\be,\mu}\,N^{2\s/\d}.
\end{align}
This proves \eqref{eq:allTermsBeatNoSmallBeta}; since $2\s/\d<1$ and neither $\mathsf m_{N,u}$ nor \eqref{eq:allTermsBeatNoSmallBeta} depends on $\eta$, the majorant condition \eqref{eq:masterUniformBound} holds uniformly in $\eta\in[0,\eta_0]$.  Step 2 is proved.

\step{Step 3: Conclusion.}
By Step 2 and \cref{rem:masterUniform},
\begin{align}
\sup_{0\le\eta\le\eta_0}\ \sup_{N\ge1}\ \E\brak[\big]{e^{\Xi_N^{(\eta)}}}<\infty.
\end{align}
This proves \eqref{eq:fullSingularExpGoal}, hence, by Step 1, both \eqref{eq:CubDef} and \eqref{eq:centeredFamilyExpMoment}.
\end{proof}

\subsubsection{The case \texorpdfstring{$\s=0$}{s=0}}

\begin{proof}[Proof of \cref{prop:uniformBoundRepulsive}, logarithmic case]
\leavevmode\par
\step{Step 1: Finite-$T$ reduction.}  Let $\Fr_{N,T}$ denote the modulated energy obtained from $\Fr_N$ by replacing $\g$, $h_\mu$, $I_\mu$ with the infrared truncations $\g_T$, $h_{\mu,T}$, $I_{\mu,T}$ introduced in \cref{subsec:fourier-truncations-deterministic-diagonal-bounds}; it satisfies the identity
\begin{align}\label{eq:logFiniteTExponent}
-\be N\Fr_{N,T}(\XN,\mu)
=\frac{\be}{N}\sum_{i=1}^N h_{\mu,T}(x_i)-\frac{\be}{2N}\sum_{1\le i\ne j\le N}\tl\g_T(x_i,x_j)-\be I_{\mu,T}.
\end{align}
Then \eqref{eq:CubDef} follows from
\begin{align}\label{eq:logFiniteTPartitionGoal}
\sup_{T\ge2}\ \sup_{N\ge1}\ \E_{\mu^{\otimes N}}\Big[e^{-\be N\Fr_{N,T}(\XN,\mu)}\Big]\le C_{\be,\mu} .
\end{align}

\par\smallskip\noindent  By \eqref{eq:logEnergyAssumption}, $\Fr_{N,T}(\XN,\mu)\to\Fr_N(\XN,\mu)$ for $\mu^{\otimes N}$-a.e.\ $\XN$ as $T\to\infty$, so \eqref{eq:logFiniteTPartitionGoal} implies \eqref{eq:CubDef} by Fatou's lemma.

\step{Step 2: Choice of the data.}
Set $C_I(\mu)\coloneqq\sup_{T\ge2}|I_{\mu,T}|$, finite by \eqref{eq:logEnergyAssumption}.  With $C_0'$ the near-field constant in \eqref{eq:logfetaConv}, $C_\mu$ and $C$ the constants in the high-frequency estimates \eqref{eq:logHighFreqNorms}, $C_{\mu,K}$ the constant of the large-scale cancellation estimate \eqref{eq:logLargeScaleCancellation}, and $C_3\coloneqq\max\brac[\big]{2C,\ \be^{-1}}$, define the frequency cutoff and the threshold
\begin{align}
K_{\be,\mu}&\coloneqq\max\brac[\Big]{\pa[\big]{4L\,\be^2C_\mu}^{1/\d},\ \pa[\big]{4L\,\be C_\mu}^{1/\d},\ 1},\label{eq:logChooseK}\\
U_0&\coloneqq 3\be\pa[\big]{C_{\mu,K_{\be,\mu}}+C_I(\mu)}+1,\label{eq:UbetaLogDef}
\end{align}
and, for $N\ge1$, $u\ge U_0$, and $T\ge2$, the adaptive physical scales, kernels, and caps
\begin{align}\label{eq:logEtaNu}
\delta_{N,u}&\coloneqq \min\Bigg\{\eta_0,\ \pa*{\frac{u}{8C_0'\be N}}^{1/\d}\Bigg\},\nn\\
q_{N,u}&\coloneqq 2\be\,\g_{\delta_{N,u},T,>K_{\be,\mu}},
\qquad
\mathsf m_{N,u}\coloneqq \be\,C_3\pa[\big]{1+\log(1/\delta_{N,u})},
\end{align}
the hypotheses of \cref{prop:masterLayerCake} hold for each fixed $T\ge2$, with $g=\be\g_T$ and $w=\be(h_{\mu,T}-I_{\mu,T})$ (so that $\Xi_{N}=-\be N\Fr_{N,T}$ by \eqref{eq:logFiniteTExponent}), and the majorant condition \eqref{eq:masterUniformBound} of \cref{rem:masterUniform} holds uniformly in $T\ge2$.

\par\smallskip\noindent  Fix $T\ge2$ and decompose $\g_T=\f_{\delta_{N,u}}+\g_{\delta_{N,u},T,\le K_{\be,\mu}}+\g_{\delta_{N,u},T,>K_{\be,\mu}}$.

\textbf{(M1).}  By \cref{lem:logLargeScaleCancellation}(ii), the one-body term pairs with the low-frequency block:
\begin{align}\label{eq:logLowBoundPartition}
\frac{\be}{N}\sum_{i=1}^N h_{\mu,T}(x_i)-\frac{\be}{2N}\sum_{1\le i\ne j\le N}\tl\g_{\delta_{N,u},T,\le K_{\be,\mu}}(x_i,x_j)
\le \be\,C_{\mu,K_{\be,\mu}} .
\end{align}
This pairing, rather than a bound on the two terms separately, is what allows the infrared cutoff $T$ to be removed at the end; neither term is bounded uniformly in $T$ on its own.  By \eqref{eq:logfetaConv} and the centering argument used in \eqref{eq:nearFieldAdaptiveNoSmallBeta},
\begin{align}\label{eq:logShortPartition}
-\frac{\be}{2N}\sum_{1\le i\ne j\le N}\tl\f_{\delta_{N,u}}(x_i,x_j)
\le C_0'\be N\delta_{N,u}^\d\le\frac u8 .
\end{align}
Combining \eqref{eq:logFiniteTExponent}--\eqref{eq:logShortPartition} with $\be|I_{\mu,T}|\le\be C_I(\mu)$ and \eqref{eq:UbetaLogDef}, for $u\ge U_0$,
\begin{align}\label{eq:logEventReductionPartition}
\brac[\big]{-\be N\Fr_{N,T}\ge u}
\subset
\brac*{-\frac{\be}{2N}\sum_{1\le i\ne j\le N}\tl\g_{\delta_{N,u},T,>K_{\be,\mu}}(x_i,x_j)\ge u-\frac u8-\frac u3\ge \frac{u}{2}},
\end{align}
on which $-\frac{1}{2N}\sum_{1\le i\ne j\le N}\tl q_{N,u}(x_i,x_j)\ge u$; this is (M1).

\textbf{(M2) and (M3).}  By \eqref{eq:logHighFreqNorms} and the choice \eqref{eq:logChooseK} of $K_{\be,\mu}$,
\begin{align}
\|\tl q_{N,u}\|_{L^\infty L^2(\mu)}^2,\ \|\tl q_{N,u}\|_{L^2(\mu^{\otimes2})}^2\le 4\be^2C_\mu K_{\be,\mu}^{-\d}\le\frac1{L},
\qquad
\|\Tc_{\tl q_{N,u}}\|_{\mathrm{op}}\le 2\be C_\mu K_{\be,\mu}^{-\d}\le\frac1{2L},
\end{align}
which is (M3); moreover, uniformly in $T$, $\|\tl q_{N,u}\|_{L^\infty}\le2\be\,C\pa[\big]{1+\log(1/\delta_{N,u})}\le\mathsf m_{N,u}$ by \eqref{eq:logHighFreqNorms} and $C_3\ge2C$, and $\mathsf m_{N,u}\ge\be C_3\ge1$.  Since $\g_{\delta,T,>K}$ is positive definite in centered quadratic forms, \cref{lem:PDdeterministicDiagonalBound} again gives $-\frac{1}{2N}\sum_{1\le i\ne j\le N}\tl q_{N,u}(x_i,x_j)\le\frac12\|\tl q_{N,u}\|_{L^\infty}\le\mathsf m_{N,u}$ for every configuration, which is (M2).  Since $u\ge1$ forces $\delta_{N,u}\ge\min\{\eta_0,(8C_0'\be N)^{-1/\d}\}$, uniformly in $u\ge U_0$ and $T\ge2$,
\begin{align}\label{eq:logLayerThresholdPartition}
\mathsf m_{N,u}=\be C_3\pa[\big]{1+\log(1/\delta_{N,u})}\le C_{\be,\mu}\log(e+N).
\end{align}

\textbf{(Sublinear growth \eqref{eq:masterLayerCondition}).}  Since $1+\|\tl q_{N,u}\|_{L^\infty}\le2\mathsf m_{N,u}$, the bound \eqref{eq:logLayerThresholdPartition} gives
\begin{align}
\mathsf S_N\le \sup_{u\ge U_0}\ 4\mathsf m_{N,u}^2\le C_{\be,\mu}\pa[\big]{\log(e+N)}^{2},
\qquad\text{uniformly in }T\ge2,
\end{align}
so, since $\mathsf m_{N,u}$ does not depend on $T$, the majorant condition \eqref{eq:masterUniformBound} holds uniformly in $T\ge2$.  Step 2 is proved.

\step{Step 3: Conclusion for \eqref{eq:CubDef}.}
By Step 2 and \cref{rem:masterUniform},
\begin{align}
\sup_{T\ge2}\ \sup_{N\ge1}\ \E_{\mu^{\otimes N}}\Big[e^{-\be N\Fr_{N,T}}\Big]<\infty,
\end{align}
which is \eqref{eq:logFiniteTPartitionGoal}; Step 1 then gives \eqref{eq:CubDef}.

\step{Step 4: The centered family \eqref{eq:centeredFamilyExpMoment}.}
Assume now $\int_{\R^\d}e^{-2\be h_\mu}\,\dd\mu<\infty$.  Fix $\eta\in[0,\eta_0]$ and let $h_\mu^{(\eta)}\coloneqq\g_\eta\ast\mu=h_\mu-\f_\eta\ast\mu$ and $I_\mu^{(\eta)}\coloneqq\frac12\iint_{(\R^\d)^2}\g_\eta\,\dd\mu^{\otimes2}$ denote the potential and energy of the truncated kernel.  Applying the splitting identity \eqref{eq:FNdef} to the kernel $\g_\eta$,
\begin{align}\label{eq:WNetaModulatedIdentity}
-\frac{\be}{2}W_{N,\eta}
=-\be N\Fr_N^{(\eta)}(\XN,\mu)-\frac{\be}{N}\sum_{i=1}^N h_\mu^{(\eta)}(x_i)+\be I_\mu^{(\eta)},
\end{align}
where $\Fr_N^{(\eta)}$ is the modulated energy with kernel $\g_\eta$.  By \eqref{eq:logEnergyAssumption} and \eqref{eq:logfetaConv}, $h_\mu^{(\eta)}\ge h_\mu-C_0'\eta_0^\d$ and $|I_\mu^{(\eta)}|\le C_\mu$, both uniformly in $\eta$.  Since the one-body sum in \eqref{eq:WNetaModulatedIdentity} is an empirical average of iid samples, two applications of Jensen's inequality give
\begin{align}\label{eq:oneBodyJensenFamily}
\E_{\mu^{\otimes N}}\brak[\Big]{e^{-\frac{2\be}{N}\sum_{i=1}^N h_\mu^{(\eta)}(x_i)}}
=\E_\mu\brak[\big]{e^{-\frac{2\be}N h_\mu^{(\eta)}}}^N
\le \E_\mu\brak[\big]{e^{-2\be h_\mu^{(\eta)}}}
\le e^{2\be C_0'\eta_0^\d}\,\E_\mu\brak[\big]{e^{-2\be h_\mu}},
\end{align}
which is finite by \eqref{eq:logRateExtraAssumptions}, uniformly in $\eta$ and $N$.  Hence, by \eqref{eq:WNetaModulatedIdentity}, Cauchy--Schwarz, and \eqref{eq:oneBodyJensenFamily},
\begin{align}\label{eq:familyCauchySchwarz}
\E_{\mu^{\otimes N}}\brak[\big]{e^{-\frac\be2 W_{N,\eta}}}
\le e^{\be C_\mu+\be C_0'\eta_0^\d}\,
\E_{\mu^{\otimes N}}\brak[\big]{e^{-2\be N\Fr_N^{(\eta)}}}^{1/2}
\,\E_\mu\brak[\big]{e^{-2\be h_\mu}}^{1/2},
\end{align}
and it therefore suffices to bound $\E[e^{-2\be N\Fr_N^{(\eta)}}]$ uniformly in $N$ and $\eta$.

This follows by repeating Steps 1--3 with the kernel $\g_{\eta}$ in place of $\g$ and $2\be$ in place of $\be$; we indicate the only changes.  At finite $T$, one applies \cref{prop:masterLayerCake} with $g=2\be\g_{\eta,T}$, where $\g_{\eta,T}=\g_T-\f_\eta$, and $w=2\be(h_{\mu,T}^{(\eta)}-I_{\mu,T}^{(\eta)})$, with one-body potential $h_{\mu,T}^{(\eta)}=h_{\mu,T}-\f_\eta\ast\mu$; a.e.\ convergence $\Fr_{N,T}^{(\eta)}\to\Fr_N^{(\eta)}$ again follows from \eqref{eq:logEnergyAssumption} since $0\le\f_\eta\le\f_{\eta_0}$.  In the proof of Step 2, one uses the effective scale $\bar\delta\coloneqq\max\{\delta_{N,u},\eta\}$ and the decomposition
\begin{align}
\g_{\eta,T}=\pa[\big]{\f_{\bar\delta}-\f_\eta}+\g_{\bar\delta,T,\le K_{2\be,\mu}}+\g_{\bar\delta,T,>K_{2\be,\mu}} .
\end{align}
The near-field part obeys \eqref{eq:logShortPartition} verbatim (it vanishes if $\eta\ge\delta_{N,u}$).  The low-frequency pairing \eqref{eq:logLowBoundPartition} is applied with $h_{\mu,T}$ and costs the additional bounded term $\frac{2\be}{N}\sum_i\f_\eta\ast\mu(x_i)\le2\be C_0'\eta_0^\d$, which is absorbed into the constant $C_{\mu,K_{2\be,\mu}}$ in \eqref{eq:UbetaLogDef}.  The high-frequency verification of (M2)--(M3) and of the sublinear growth \eqref{eq:masterLayerCondition} is unchanged, with $\bar\delta\ge\delta_{N,u}$ giving $1+\log(1/\bar\delta)\le1+\log(1/\delta_{N,u})$.  \Cref{prop:masterLayerCake}, in the uniform form of \cref{rem:masterUniform}, and Fatou's lemma then give \eqref{eq:centeredFamilyExpMoment}.
\end{proof}

\begin{remark}[Role of the frequency truncation]\label{rem:frequencyTruncationRole}
The role of the frequency cutoff $K_{\be,\mu}$ is to replace a high-temperature smallness condition on $\be$ by the choice of a $(\be,\mu)$-dependent parameter: the low frequencies cost a finite deterministic offset, reflected in the thresholds \eqref{eq:UbetaRieszDef} and \eqref{eq:UbetaLogDef}, while the high frequencies have the mixed $L^\infty L^2(\mu)$ norm, the $L^2(\mu^{\otimes2})$ norm, and the integral-operator norm small enough for \eqref{eq:masterNorms} exactly because $\wh\g\in L^2(|\xi|>1)$ in the range $\s<\d/2$.  The remaining $L^\infty$ norm is controlled by the adaptive physical truncation together with the deterministic diagonal bound of \cref{lem:PDdeterministicDiagonalBound}, which is available before any Fourier cutoff is introduced.  No smallness assumption on $\be$ is used.  Note also that the restriction $\s<\d/2$ enters the verification exactly twice: in the finiteness of the tail integrals \eqref{eq:CtailFourier}--\eqref{eq:BtailFourier}, and in the exponent $2\s/\d<1$ in \eqref{eq:allTermsBeatNoSmallBeta}.
\end{remark}

We conclude this subsection with the interpolated exponential-moment bound needed for the determinant limit.  For $0<\eta\le\eta_0$ and $t\in[0,1]$, set $\g_{\eta,t}\coloneqq\g_\eta+t\f_\eta$ and let $W_{N,\eta,t}$ denote the statistic \eqref{eq:WNdefRate} with kernel $\g_{\eta,t}$.

\begin{lemma}[Interpolated kernels]\label{lem:interpolatedExpMoment}
Under the hypotheses of \cref{prop:uniformBoundRepulsive}, including \eqref{eq:logRateExtraAssumptions} when $\s=0$, for every fixed $\be\ge0$,
\begin{align}\label{eq:interpolatedExpMomentBound}
\sup_{N\ge1}\ \sup_{0<\eta\le\eta_0}\ \sup_{0\le t\le1}\
\E_{\mu^{\otimes N}}\Big[e^{-\be W_{N,\eta,t}}\Big]\le C_{\mathrm{cf}}(2\be,\mu).
\end{align}
\end{lemma}

\begin{proof}
Since $\f_\eta=\g-\g_\eta$, the interpolated kernel is the convex combination $\g_{\eta,t}=(1-t)\g_\eta+t\g$, and the Hoeffding centering \eqref{eq:tildedef} is linear in the kernel.  Hence
\begin{align}
W_{N,\eta,t}=(1-t)W_{N,\eta}+tW_N
\end{align}
pointwise, and H\"older's inequality with exponents $\frac1{1-t}$ and $\frac1t$ gives
\begin{align}
\E\Big[e^{-\be W_{N,\eta,t}}\Big]
\le \E\Big[e^{-\be W_{N,\eta}}\Big]^{1-t}\,\E\Big[e^{-\be W_{N}}\Big]^{t}
\le C_{\mathrm{cf}}(2\be,\mu),
\end{align}
by \eqref{eq:centeredFamilyExpMoment} applied at inverse temperature $2\be$.  In the logarithmic case, this requires precisely the first condition in \eqref{eq:logRateExtraAssumptions}.
\end{proof}
\subsection{Regularized determinant limit}
\label{subsec:regularized-determinant-limit}
\label{subsec:completion-quantitative-partition-function}

Having proved uniform integrability of the exponential weights, we now identify the limiting normalization constant for the centered two-body statistic.  If $\be=0$, then $\E_{\mu^{\otimes N}}[1]=1=\det_2(\mathrm{Id})^{-1/2}$, so the remaining argument assumes $\be>0$.

The expansion at a fixed truncation scale rests on the following bounded-feature quadratic Laplace estimate, which is applied after truncating the kernel.

\subsubsection{Feature-space quadratic Laplace estimate}
\label{subsec:bounded-feature-quadratic-laplace-estimate}

We first record the fixed-kernel estimate used after truncating the interaction.  The formulation allows the squared feature norm $x\mapsto\|k(x)\|_{\mathcal H}^2$ to be unbounded, provided it satisfies the exponential-moment condition \eqref{eq:expFeatureMomentDef}; this is needed for logarithmic truncations against noncompact backgrounds.

\begingroup

\begin{prop}[Feature-space quadratic Laplace estimate]\label{prop:boundedFeatureQuadLaplace}
Assume that $\tl\g$ admits a Hilbert-space feature representation: there are a separable real Hilbert space $\mathcal H$ and a strongly measurable map $k:\R^\d\to\mathcal H$ such that
\begin{align}\label{eq:boundedFeatureHyp}
\tl\g(x,y)=\langle k(x),k(y)\rangle_{\mathcal H}
\quad\text{for $\mu^{\otimes2}$-a.e. }(x,y),
\qquad
\int_{\R^\d}k\,\dd\mu=0.
\end{align}
Set
\begin{align}\label{eq:expFeatureMomentDef}
D(x)\coloneqq\|k(x)\|_{\mathcal H}^2,
\qquad
\mathcal M_\be(k)\coloneqq
\int_{\R^\d}e^{\be D(x)/2}\pa*{1+D(x)^4}\,\dd\mu(x).
\end{align}
If $\mathcal M_\be(k)<\infty$, then there is a constant $C_\be<\infty$, depending only on $\be$, such that, for every $N\ge1$,
\begin{align}\label{eq:boundedFeatureQuadLaplaceRate}
\Big|\E_{\mu^{\otimes N}}\Big[e^{-\frac{\be}{2}W_N}\Big]
-\det_2(\mathrm{Id}+\be T_\mu)^{-1/2}\Big|
\le \frac{C_\be}{N}\pa*{1+\mathcal M_\be(k)}^3.
\end{align}
In particular, if $k\in L^\infty(\mu;\mathcal H)$, then
\begin{align}\label{eq:boundedFeatureSpecialization}
\Big|\E_{\mu^{\otimes N}}\Big[e^{-\frac{\be}{2}W_N}\Big]
-\det_2(\mathrm{Id}+\be T_\mu)^{-1/2}\Big|
\le C_\be\exp\pa*{C_\be\be\|k\|_{L^\infty(\mu;\mathcal H)}^2}N^{-1}.
\end{align}
\end{prop}

\begin{proof}
The case $\be=0$ is immediate, so assume $\be>0$.  Let $\mathsf Z$ be an isonormal Gaussian process over $\mathcal H$, independent of the $\xi_i$'s, and write $\mathsf Z_x\coloneqq\mathsf Z(k(x))$.  Thus, $\mathsf Z_x$ is centered Gaussian with variance $D(x)$.  Define
\begin{align}
\bar D&\coloneqq\int_{\R^\d}D\,\dd\mu,
&Q(\mathsf Z)&\coloneqq\int_{\R^\d}\mathsf Z_x^2\,\dd\mu(x),\\
B(\mathsf Z)&\coloneqq\frac\be2\pa*{\bar D-Q(\mathsf Z)},
&R(\mathsf Z)&\coloneqq\be^{3/2}\int_{\R^\d}
\pa*{\frac12D(x)\mathsf Z_x-\frac16\mathsf Z_x^3}\,\dd\mu(x).
\end{align}
The moment assumption implies $D\in L^p(\mu)$ for every $1\le p\le4$, so all of these quantities are well-defined.  Moreover,
\begin{align}\label{eq:expFeatureMomentBounds}
\E_{\mathsf Z}\brak[\big]{B^2+R^2}
\le C_\be\pa*{1+\mathcal M_\be(k)}^2.
\end{align}
Indeed, Wick's formula gives
\begin{align}
\operatorname{Var}_{\mathsf Z}(Q)
=2\iint_{(\R^\d)^2}\langle k(x),k(y)\rangle_{\mathcal H}^2\,\dd\mu^{\otimes2}(x,y)
\le2\bar D^2,
\end{align}
while Jensen's inequality in $x$ and the Gaussian sixth-moment identity give $\E_{\mathsf Z}[R^2]\le C_\be\int_{\R^\d} D^3\,\dd\mu$.

\step{Step 1: Oscillatory Hubbard--Stratonovich transformation.}
With $S_N\coloneqq N^{-1/2}\sum_{i=1}^Nk(\xi_i)$, the feature representation gives
\begin{align}
W_N=\|S_N\|_{\mathcal H}^2-\frac1N\sum_{i=1}^ND(\xi_i).
\end{align}
Consequently,
\begin{align}\label{eq:boundedFeatureFourierRep}
\E_{\mu^{\otimes N}}\Big[e^{-\frac\be2W_N}\Big]
=\E_{\mathsf Z}\brak[\big]{A_N(\mathsf Z)^N},
\qquad
A_N(\mathsf Z)
\coloneqq\int_{\R^\d}
\exp\pa*{\iu\sqrt{\frac\be N}\mathsf Z_x+\frac\be{2N}D(x)}\,\dd\mu(x).
\end{align}
To justify Fubini, set
\begin{align}
a_N\coloneqq\int_{\R^\d} e^{\be D/(2N)}\,\dd\mu,
\qquad
A_\be\coloneqq\int_{\R^\d} e^{\be D/2}\,\dd\mu.
\end{align}
Since $r\mapsto r^{1/N}$ is concave,
\begin{align}\label{eq:expFeatureANBound}
|A_N(\mathsf Z)|\le a_N,
\qquad
a_N^N\le A_\be,
\qquad
e^{\be\bar D/2}\le A_\be.
\end{align}
The same bounds justify every interchange below.  Equivalently, one may first truncate to the sets $\{D\le m\}$ and then pass to the limit by dominated convergence using \eqref{eq:expFeatureMomentDef}.

\step{Step 2: Expansion of $A_N$.}
Set
\begin{align}
f(\tau)\coloneqq\int_{\R^\d}
\exp\pa*{\iu\tau\sqrt\be\,\mathsf Z_x+\frac{\tau^2\be}{2}D(x)}\,\dd\mu(x).
\end{align}
Then $A_N=f(N^{-1/2})$, $f(0)=1$, and the centering of $k$ gives $f'(0)=0$.  Direct differentiation yields $f''(0)/2=B$ and $f'''(0)/6=\iu R$.  Hence
\begin{align}\label{eq:boundedFeatureANExpansion}
A_N=1+\frac BN+\iu\frac R{N^{3/2}}+\frac{E_N}{N^2},
\end{align}
where $B$ is even and $R$ is odd under $\mathsf Z\mapsto-\mathsf Z$, and
\begin{align}\label{eq:expFeatureRemainderBound}
\sup_{N\ge1}\E_{\mathsf Z}|E_N|
\le C_\be\mathcal M_\be(k).
\end{align}
Indeed, if $g_x(\tau)=\iu\tau\sqrt\be\,\mathsf Z_x+\tau^2\be D(x)/2$, then
\begin{align}
(e^{g_x})''''=e^{g_x}\pa*{(g_x')^4+6(g_x')^2g_x''+3(g_x'')^2},
\end{align}
with $\operatorname{Re}g_x(\tau)\le\be D(x)/2$, $|g_x'(\tau)|\le\sqrt\be|\mathsf Z_x|+\be D(x)$, and $g_x''=\be D(x)$ for $0\le\tau\le1$.  Taking the Gaussian expectation and using $\E|\mathsf Z_x|^{2p}=C_pD(x)^p$ proves \eqref{eq:expFeatureRemainderBound}.

\step{Step 3: Passage to the Gaussian limit.}
Put $u_N=N^{-1}(B+\iu R/\sqrt N)$.  Then
\begin{align}
A_N^N-e^B
=\pa*{A_N^N-e^{Nu_N}}
+e^B\pa*{e^{\iu R/\sqrt N}-1-\iu R/\sqrt N}
+\frac\iu{\sqrt N}e^BR.
\end{align}
The last term has zero expectation by parity and is integrable by \eqref{eq:expFeatureANBound} and \eqref{eq:expFeatureMomentBounds}.  The middle term is bounded in expectation by $A_\be\E[R^2]/(2N)$.  For the first term, the telescoping identity gives, on $\{|u_N|\le1\}$,
\begin{align}
|A_N^N-e^{Nu_N}|
\le \frac{A_\be}{N}\pa*{|E_N|+e(|B|+|R|)^2}.
\end{align}
On $\{|u_N|>1\}$, one has $|B|+|R|>N$, while $|A_N|^N+|e^{Nu_N}|\le2A_\be$; hence the contribution of this event is at most
\begin{align}
2A_\be N^{-2}\E\brak[\big]{(|B|+|R|)^2}.
\end{align}
Combining these estimates yields
\begin{align}\label{eq:boundedFeatureLaplaceExpansion}
\Big|\E_{\mathsf Z}[A_N^N]-\E_{\mathsf Z}[e^B]\Big|
\le\frac{C_\be}{N}\pa*{1+\mathcal M_\be(k)}^3.
\end{align}

\step{Step 4: Identification of the leading term.}
The covariance operator
\begin{align}
\Sigma f\coloneqq\int_{\R^\d}\langle k(x),f\rangle_{\mathcal H}k(x)\,\dd\mu(x)
\end{align}
is positive and trace class, with $\operatorname{Tr}\Sigma=\bar D$, and has the same nonzero eigenvalues $(\lambda_j)_{j\ge1}$ as $T_\mu$.  Diagonalizing $\Sigma$ therefore gives
\begin{align}\label{eq:boundedFeatureGaussianDet}
\E_{\mathsf Z}[e^B]
=\prod_{j\ge1}e^{\be\lambda_j/2}(1+\be\lambda_j)^{-1/2}
=\det_2(\mathrm{Id}+\be T_\mu)^{-1/2}.
\end{align}
Together with \eqref{eq:boundedFeatureLaplaceExpansion}, this proves \eqref{eq:boundedFeatureQuadLaplaceRate}.  If $\|k\|_{L^\infty}\le M$, then $\mathcal M_\be(k)\le e^{\be M^2/2}(1+M^8)$, and \eqref{eq:boundedFeatureSpecialization} follows after increasing $C_\be$.
\end{proof}
\endgroup

\begin{prop}\label{prop:rieszLaplaceLogFallback}
Assume $0\le\s<\d/2$, $\mu\in L^\infty\cap\P(\R^\d)$, and $\be>0$ is fixed.  If $\s=0$, assume in addition \eqref{eq:logEnergyAssumption} and \eqref{eq:logRateExtraAssumptions}.  Then there are constants $C_{\be,\mu}<\infty$ and, in the case $\s=0$, $c_{\be,\mu}>0$ such that, for all $N\ge1$,
\begin{align}\label{eq:rieszLaplaceLogFallback}
\Big|\E_{\mu^{\otimes N}}\Big[e^{-\frac{\be}{2}W_N}\Big]
-\det_2(\mathrm{Id}+\be T_\mu)^{-1/2}\Big|
\le C_{\be,\mu}
\begin{cases}
(\log(e+N))^{-\frac{\d-2\s}{2\s}}, & 0<\s<\d/2,\\[0.4em]
(e+N)^{-c_{\be,\mu}}, & \s=0.
\end{cases}
\end{align}
\end{prop}

\begin{proof}
Throughout, $W_{N,\eta}$ and $T_{\mu,\eta}$ denote the analogues of \eqref{eq:WNdefRate} and \eqref{eq:TmuDef} with $\g$ replaced by $\g_\eta$, using the logarithmic truncation of \eqref{eq:logFiniteTTruncDef} when $\s=0$.  Set
\begin{align}
\alpha_\eta\coloneqq
\begin{cases}
\eta^{-\s}, & 0<\s<\d/2,\\
1+\log(1/\eta), & \s=0,
\end{cases}
\qquad\qquad
\rho_\eta=
\eta^{\d/2-\s}
\end{align}
and, with $C_{\mathrm L}=C_{\mathrm L}(\be,\mu)$ the constant furnished by Step 1 below, fix the truncation scale
\begin{align}\label{eq:etaNChoice}
\eta(N)\coloneqq
\begin{cases}
\pa[\big]{(2C_{\mathrm L})^{-1}\log N}^{-1/\s}, & 0<\s<\d/2,\\[0.3em]
N^{-a},\qquad a\coloneqq\min\brac[\big]{1,\ (2C_{\mathrm L})^{-1}}, & \s=0,
\end{cases}
\qquad\text{so that}\qquad
e^{C_{\mathrm L}\alpha_{\eta(N)}}\le e^{C_{\mathrm L}}N^{1/2}.
\end{align}

\step{Step 1: Laplace expansion at fixed scale.}  We claim that there is $C_{\mathrm L}=C_{\mathrm L}(\be,\mu)<\infty$ such that, for all $N\ge1$ and $0<\eta\le\eta_0$,
\begin{align}\label{eq:truncatedLaplaceRateEta}
\Big|\E_{\mu^{\otimes N}}\Big[e^{-\frac{\be}{2}W_{N,\eta}}\Big]
-\det_2(\mathrm{Id}+\be T_{\mu,\eta})^{-1/2}\Big|
\le C_{\mathrm L}\,e^{C_{\mathrm L}\alpha_\eta}\,N^{-1}.
\end{align}

\par\smallskip\noindent
\begingroup

Assume first $0<\s<\d/2$.  The centered truncated kernel has a bounded feature representation as in \cref{prop:boundedFeatureQuadLaplace}.  Indeed, if
\begin{align}
\g_\eta(x-y)=\langle\kappa_\eta(x),\kappa_\eta(y)\rangle_{\mathcal H_\eta},
\end{align}
then
\begin{align}\label{eq:truncatedFeatureCentering}
\tl\g_\eta(x,y)
=\Big\langle \kappa_\eta(x)-\int_{\R^\d}\kappa_\eta\,\dd\mu,
\kappa_\eta(y)-\int_{\R^\d}\kappa_\eta\,\dd\mu\Big\rangle_{\mathcal H_\eta},
\end{align}
so centering preserves positive definiteness and
\begin{align}
\sup_x\Big\|\kappa_\eta(x)-\int_{\R^\d}\kappa_\eta\,\dd\mu\Big\|_{\mathcal H_\eta}^2
\le4\g_\eta(0)\le C\eta^{-\s}=C\alpha_\eta.
\end{align}
The bounded specialization \eqref{eq:boundedFeatureSpecialization} gives \eqref{eq:truncatedLaplaceRateEta}.

Assume now $\s=0$.  The centered kernel $\tl\g_\eta$ is positive definite; let $k_\eta$ be its canonical feature map and set
\begin{align}\label{eq:logFeatureDiagonalIdentity}
D_\eta(x)
\coloneqq\|k_\eta(x)\|_{\mathcal H_\eta}^2
=\tl\g_\eta(x,x)
=\g_\eta(0)-2h_\mu^{(\eta)}(x)+2I_\mu^{(\eta)},
\end{align}
where $h_\mu^{(\eta)}=\g_\eta\ast\mu$ and $I_\mu^{(\eta)}=\frac12\iint_{(\R^\d)^2}\g_\eta\,\dd\mu^{\otimes2}$.  Since $D_\eta\in L^1(\mu)$ and $\tl\g_\eta$ is canonical, the feature map may be chosen with $\int_{\R^\d} k_\eta\,\dd\mu=0$.  By $h_\mu^{(\eta)}=h_\mu-\f_\eta\ast\mu$, \eqref{eq:logfetaConv}, the diagonal estimate \eqref{eq:logTruncPropsBasic}, and the uniform boundedness of $I_\mu^{(\eta)}$, there is $C_\mu<\infty$ such that
\begin{align}\label{eq:logFeatureDiagonalBound}
0\le D_\eta(x)\le C_\mu\alpha_\eta-2h_\mu(x)
\qquad\text{for $\mu$-a.e. }x.
\end{align}
Since $h_\mu$ is bounded above and $\int_{\R^\d} e^{-4\be h_\mu}\,\dd\mu<\infty$, the polynomial factor in \eqref{eq:expFeatureMomentDef} can be absorbed by this exponential moment, and
\begin{align}\label{eq:logFeatureExpMoment}
\mathcal M_\be(k_\eta)
=\int_{\R^\d} e^{\be D_\eta/2}\pa*{1+D_\eta^4}\,\dd\mu
\le C_{\be,\mu}e^{C_{\be,\mu}\alpha_\eta}.
\end{align}
Applying \cref{prop:boundedFeatureQuadLaplace} and enlarging the constant to absorb the fixed power in \eqref{eq:boundedFeatureQuadLaplaceRate} proves \eqref{eq:truncatedLaplaceRateEta}.
\endgroup

\step{Step 2: Determinant truncation error.}  It holds that
\begin{align}\label{eq:centeredFetaL2Rate}
\|\tl\f_\eta\|_{L^2(\mu^{\otimes2})}\le C_\mu\,\rho_\eta
\qquad\text{and}\qquad
\Big|\det_2(\mathrm{Id}+\be T_\mu)^{-1/2}
-\det_2(\mathrm{Id}+\be T_{\mu,\eta})^{-1/2}\Big|
\le C_{\be,\mu}\,\rho_\eta.
\end{align}

\par\smallskip\noindent  The short-range remainder satisfies
\begin{align}\label{eq:fetaL2Rate}
\|\f_\eta\|_{L^2(\mu^{\otimes2})}\le C_\mu\,\rho_\eta:
\end{align}
for $0<\s<\d/2$, this is the scaling in \eqref{eq:HCRStruncDef}, while for $\s=0$, one has $\f_\eta(x)=F(x/\eta)$ for a fixed $F\in L^2(\R^\d)$, using \eqref{eq:logfetaDecay}.  Since Hoeffding centering is bounded on $L^2(\mu^{\otimes2})$, the first bound in \eqref{eq:centeredFetaL2Rate} follows.  \begingroup

For the second, the continuity estimate of \cite[Theorem~9.2(c)]{Simon2005} shows that $A\mapsto\det_2(\mathrm{Id}+\be A)$ is Lipschitz on Hilbert--Schmidt balls.  On the cone of nonnegative self-adjoint Hilbert--Schmidt operators one also has
\begin{align}
\det_2(\mathrm{Id}+\be A)
\ge \exp\pa*{-\frac{\be^2}{2}\|A\|_{\mathrm{HS}}^2},
\end{align}
so composing with the inverse square root is locally Lipschitz there.  The asserted estimate now follows from
\begin{align}
\|T_\mu-T_{\mu,\eta}\|_{\mathrm{HS}}
=\|\tl\f_\eta\|_{L^2(\mu^{\otimes2})}.
\end{align}
\endgroup

\step{Step 3: Finite-$N$ truncation error.}  It holds that
\begin{align}\label{eq:finiteNTruncationRate}
\Big|\E_{\mu^{\otimes N}}\Big[e^{-\frac{\be}{2}W_N}\Big]
-\E_{\mu^{\otimes N}}\Big[e^{-\frac{\be}{2}W_{N,\eta}}\Big]\Big|
\le C_{\be,\mu}\,\rho_\eta .
\end{align}

\par\smallskip\noindent  With $\g_{\eta,t}=\g_\eta+t\f_\eta$ and $W_{N,\eta,t}$ as in \cref{lem:interpolatedExpMoment}, differentiating in $t$ and applying Cauchy--Schwarz gives
\begin{align}
\Big|\frac{\dd}{\dd t}\E_{\mu^{\otimes N}}\Big[e^{-\frac{\be}{2}W_{N,\eta,t}}\Big]\Big|
\le \frac{\be}{2}\,
\E_{\mu^{\otimes N}}\Big[\Big|\frac1N\sum_{i\ne j}\tl\f_\eta(\xi_i,\xi_j)\Big|^2\Big]^{1/2}
\E_{\mu^{\otimes N}}\Big[e^{-\be W_{N,\eta,t}}\Big]^{1/2},
\end{align}
and, since $\tl\f_\eta$ is canonical,
\begin{align}\label{eq:canonicalL2Identity}
\E_{\mu^{\otimes N}}\Big[\Big|\frac1N\sum_{i\ne j}\tl\f_\eta(\xi_i,\xi_j)\Big|^2\Big]
=\frac{2(N-1)}{N}\|\tl\f_\eta\|_{L^2(\mu^{\otimes2})}^2
\le 2\|\tl\f_\eta\|_{L^2(\mu^{\otimes2})}^2.
\end{align}
Combining \cref{lem:interpolatedExpMoment}, \eqref{eq:canonicalL2Identity}, and \eqref{eq:centeredFetaL2Rate}, and integrating $t$ from $0$ to $1$, proves \eqref{eq:finiteNTruncationRate}.

\step{Step 4: Conclusion.}  For $N=1$ and the finitely many $N$ with $\eta(N)>\eta_0$, the left-hand side of \eqref{eq:rieszLaplaceLogFallback} is bounded by a finite constant depending on $(\be,\mu)$, by \cref{lem:interpolatedExpMoment} and the estimate
\begin{align}
\det_2(\mathrm{Id}+\be T_\mu)^{-1/2}
\le \exp\pa*{\frac{\be^2}{4}\|T_\mu\|_{\mathrm{HS}}^2},
\end{align}
which follows from $x-\log(1+x)\le x^2/2$ for $x\ge0$, and is absorbed into $C_{\be,\mu}$, the right-hand side of \eqref{eq:rieszLaplaceLogFallback} being bounded below on finite sets.  For the remaining $N$, summing Steps 1--3 at the scale $\eta=\eta(N)$ and using \eqref{eq:etaNChoice},
\begin{align}
\Big|\E_{\mu^{\otimes N}}\Big[e^{-\frac{\be}{2}W_N}\Big]
-\det_2(\mathrm{Id}+\be T_\mu)^{-1/2}\Big|
\le C_{\be,\mu}\,\rho_{\eta(N)}
+C_{\mathrm L}e^{C_{\mathrm L}}\,N^{-1/2}.
\end{align}
If $0<\s<\d/2$, then $\rho_{\eta(N)}=\pa[\big]{(2C_{\mathrm L})^{-1}\log N}^{-\frac{\d-2\s}{2\s}}\le C_{\be,\mu}(\log(e+N))^{-\frac{\d-2\s}{2\s}}$, which dominates $N^{-1/2}$ and gives the first line of \eqref{eq:rieszLaplaceLogFallback}.  If $\s=0$, then $\rho_{\eta(N)}=N^{-a\d/2}\le C(e+N)^{-a\d/2}$, and the second line follows with $c_{\be,\mu}\coloneqq\min\brac[\big]{a\d/2,\ 1/2}$.
\end{proof}

\subsection{One-body reduction and proof of the main estimate}
\label{subsec:one-body-reduction}

\begin{lemma}[One-body reduction]\label{lem:oneBodyReduction}
Assume the hypotheses of \cref{prop:rieszLaplaceLogFallback}, and recall $\psi_\mu\coloneqq h_\mu-2I_\mu$ and $R_N\coloneqq\frac1N\sum_{i=1}^N\psi_\mu(\xi_i)$ from \eqref{eq:roadmapIdentity}.  Then, for all $N\ge1$,
\begin{multline}\label{eq:oneBodyReductionRate}
\Big|\K_{N,\be}(\mu)
-e^{\be I_\mu}\E_{\mu^{\otimes N}}\Big[e^{-\frac{\be}{2}W_N}\Big]\Big|\\
\le
\be\,\|\psi_\mu\|_{L^2(\mu)}
\pa[\Big]{C_{\mathrm{ub}}(2\be,\mu)^{1/2}
+C_{\mathrm{ub}}(4\be,\mu)^{1/4}\,\E_\mu\brak[\big]{e^{-4\be\psi_\mu}}^{1/4}}\,N^{-1/2}.
\end{multline}
\end{lemma}

\begin{proof}
Note first that the hypotheses guarantee $\psi_\mu\in L^2(\mu)$ and $e^{-4\be\psi_\mu}\in L^1(\mu)$: for $0<\s<\d/2$ both follow from $h_\mu\in L^\infty$ as in Step~1 of the Riesz case, and for $\s=0$ from \eqref{eq:logRateExtraAssumptions}, since $h_\mu$ is bounded above (\cref{rem:oneBodyExpIntegrability}) and $x^2\le C_\be(1+e^{-4\be x})$ for $x$ in a half-line $(-\infty,C]$.  By \eqref{eq:FNdef} and the definition \eqref{eq:WNdefRate} of $W_N$, the identity \eqref{eq:roadmapIdentity} holds:
\begin{align}\label{eq:KExactOneBodyReduction}
\K_{N,\be}(\mu)
=e^{\be I_\mu}\E_{\mu^{\otimes N}}\Big[e^{-\frac{\be}{2}W_N+\be R_N}\Big].
\end{align}
Since $\E_\mu[\psi_\mu]=0$,
\begin{align}\label{eq:RNvarianceRate}
\E_{\mu^{\otimes N}}[R_N^2]=\frac{\|\psi_\mu\|_{L^2(\mu)}^2}{N},
\end{align}
and, since $R_N$ is an empirical average of iid samples, two applications of Jensen's inequality give
\begin{align}\label{eq:RNexpJensen}
\E_{\mu^{\otimes N}}\brak[\big]{e^{-4\be R_N}}
=\E_\mu\brak[\big]{e^{-\frac{4\be}N\psi_\mu}}^N
\le\E_\mu\brak[\big]{e^{-4\be\psi_\mu}}.
\end{align}
Moreover, by Cauchy--Schwarz, \eqref{eq:KExactOneBodyReduction} at inverse temperature $4\be$, \eqref{eq:RNexpJensen}, and the uniform bound \eqref{eq:CubDef},
\begin{align}\label{eq:expMinusBetaWBound}
\E_{\mu^{\otimes N}}\Big[e^{-\be W_N}\Big]
\le \E_{\mu^{\otimes N}}\Big[e^{-2\be W_N+4\be R_N}\Big]^{1/2}
\E_{\mu^{\otimes N}}\Big[e^{-4\be R_N}\Big]^{1/2}
\le e^{-2\be I_\mu}\,C_{\mathrm{ub}}(4\be,\mu)^{1/2}\,\E_\mu\brak[\big]{e^{-4\be\psi_\mu}}^{1/2}.
\end{align}
Therefore, by $|e^z-1|\le|z|\pa{1+e^{z}}$ with $z=\be R_N$, followed by Cauchy--Schwarz in each term,
\begin{align}
&\Big|\K_{N,\be}(\mu)
-e^{\be I_\mu}\E_{\mu^{\otimes N}}\Big[e^{-\frac{\be}{2}W_N}\Big]\Big|\nn\\
&\le e^{\be I_\mu}\,\be\,
\pa[\Big]{\E_{\mu^{\otimes N}}\Big[|R_N|e^{-\frac{\be}{2}W_N}\Big]
+\E_{\mu^{\otimes N}}\Big[|R_N|e^{-\frac{\be}{2}W_N+\be R_N}\Big]}\nn\\
&\le e^{\be I_\mu}\,\be\,\E_{\mu^{\otimes N}}[R_N^2]^{1/2}
\pa[\Big]{\E_{\mu^{\otimes N}}\Big[e^{-\be W_N}\Big]^{1/2}
+\E_{\mu^{\otimes N}}\Big[e^{-\be W_N+2\be R_N}\Big]^{1/2}}.
\end{align}
By \eqref{eq:KExactOneBodyReduction} at inverse temperature $2\be$ and \eqref{eq:CubDef}, $\E_{\mu^{\otimes N}}[e^{-\be W_N+2\be R_N}]\le e^{-2\be I_\mu}C_{\mathrm{ub}}(2\be,\mu)$.  Combining this with \eqref{eq:RNvarianceRate} and \eqref{eq:expMinusBetaWBound}, the prefactor $e^{\be I_\mu}$ cancels against the factors $e^{-\be I_\mu}$ from the two square roots, and \eqref{eq:oneBodyReductionRate} follows.
\end{proof}

\begin{proof}[Proof of \cref{thm:partitionFunctionBound}]
The uniform bound \eqref{eq:partitionUniformBoundStatement} is exactly \eqref{eq:CubDef} in \cref{prop:uniformBoundRepulsive}.  For the convergence rate, combining \cref{lem:oneBodyReduction}, the determinant estimate \eqref{eq:rieszLaplaceLogFallback}, and the definition \eqref{eq:KinftyDef} of $\K_{\infty,\be}(\mu)$,
\begin{align}
\Big|\K_{N,\be}(\mu)-\K_{\infty,\be}(\mu)\Big|
&\le C_{\be,\mu}N^{-1/2}
+C_{\be,\mu}
\begin{cases}
(\log(e+N))^{-\frac{\d-2\s}{2\s}}, & 0<\s<\d/2,\\[0.4em]
(e+N)^{-c_{\be,\mu}}, & \s=0
\end{cases}\nn\\
&\le C_{\mathrm{rate}}(\be,\mu)
\begin{cases}
(\log(e+N))^{-\frac{\d-2\s}{2\s}}, & 0<\s<\d/2,\\[0.4em]
(e+N)^{-c_{\log}(\be,\mu)}, & \s=0,
\end{cases}
\end{align}
which proves \eqref{eq:partitionLogFallbackStatement}.
\end{proof}

\subsection{Temperature interpolation corollary}
\label{subsec:partition-function-temperature-interpolation}

\begingroup
The following corollary records how the fixed-temperature estimate interpolates when the inverse temperature varies with $N$.  It uses the pointwise lower bound for modulated energies from \cite{HCRS2025}.

\begin{cor}[Temperature interpolation]\label{cor:partitionFunctionTemperatureInterpolation}
Assume $0\le\s<\d/2$ and $\mu\in L^\infty\cap\P(\R^\d)$; if $\s=0$, assume \eqref{eq:logEnergyAssumption}.  Set
\begin{align}\label{eq:temperatureInterpolationScale}
a_{N,\s}\coloneqq
\begin{cases}
\log(e+N), & \s=0,\\
N^{\s/\d}, & 0<\s<\d.
\end{cases}
\end{align}
For every $B>0$, there is a constant $C_B(\mu)<\infty$ such that
\begin{align}\label{eq:fixedTemperatureWindowBound}
\sup_{N\ge1}\ \sup_{0\le\be\le B}\ \K_{N,\be}(\mu)\le C_B(\mu)
\end{align}
and, for all $\be\ge0$ and $N\ge1$,
\begin{align}\label{eq:temperatureInterpolationBound}
\K_{N,\be}(\mu)
\le C_B(\mu)\exp\pa[\Big]{C_B(\mu)\,(\be-B)_+\,a_{N,\s}}.
\end{align}
In particular, for every $B_\ast>0$ and $\gamma\in\R$, there is a constant $C_{B_\ast,\mu}<\infty$ such that, along the temperature scaling $\be_N=\be_\ast N^\gamma$,
\begin{align}\label{eq:growingTemperatureSpecialization}
\sup_{0\le\be_\ast\le B_\ast}\ \log \K_{N,\be_\ast N^\gamma}(\mu)
\le
\begin{cases}
C_{B_\ast,\mu}, & \gamma\le0,\\
C_{B_\ast,\mu}\,N^\gamma\log(e+N), & \gamma>0,\ \s=0,\\
C_{B_\ast,\mu}\,N^{\gamma+\s/\d}, & \gamma>0,\ 0<\s<\d/2.
\end{cases}
\end{align}
\end{cor}

\begin{remark}[Microscopic endpoint]\label{rem:LebleSerfatyComparison}
At the microscopic temperature scale, namely $\be_N\sim N$ in the logarithmic case and $\be_N\sim N^{1-\s/\d}$ in the positive Riesz case, \eqref{eq:growingTemperatureSpecialization} gives $\log\K_{N,\be_N}=O(N\log(e+N))$ for $\s=0$ and $O(N)$ for $0<\s<\d/2$.  These are the expected endpoint orders in the standard Coulomb/Riesz-gas normalization; see \cite[Chapters~8--13]{SerfatyLN}.  This comparison is only a consistency check: \cref{cor:partitionFunctionTemperatureInterpolation} is an upper bound for arbitrary bounded backgrounds and neither identifies the next-order constants nor establishes sharpness at intermediate temperature scales.
\end{remark}
\endgroup

\begin{proof}[Proof of \cref{cor:partitionFunctionTemperatureInterpolation}]
We combine the uniform bound \eqref{eq:CubDef} with the pointwise lower bound for modulated energies from \cite{HCRS2025}, which, applied under the assumption $\mu\in L^\infty$, gives a constant $C_{\mathrm{pt}}(\mu)<\infty$ such that, for every $N\ge1$ and every configuration $\XN$,
\begin{align}\label{eq:pointwiseLowerBoundForInterpolation}
\Fr_N(\XN,\mu)\ge -C_{\mathrm{pt}}(\mu)\frac{a_{N,\s}}{N}.
\end{align}
Set $G_N(\XN)\coloneqq -N\Fr_N(\XN,\mu)$.  By \eqref{eq:pointwiseLowerBoundForInterpolation},
\begin{align}\label{eq:GNPointwiseUpper}
G_N(\XN)\le C_{\mathrm{pt}}(\mu)\,a_{N,\s}.
\end{align}
If $0\le\be\le B$, then H\"older's inequality and \eqref{eq:CubDef} give
\begin{align}
\K_{N,\be}(\mu)
=\E_{\mu^{\otimes N}}\Big[e^{\be G_N}\Big]
\le \E_{\mu^{\otimes N}}\Big[e^{B G_N}\Big]^{\be/B}
=\K_{N,B}(\mu)^{\be/B}
\le C_B(\mu),
\end{align}
after increasing $C_B(\mu)$ if necessary.  This proves \eqref{eq:fixedTemperatureWindowBound}.  If $\be>B$, then \eqref{eq:GNPointwiseUpper} gives
\begin{align}
\K_{N,\be}(\mu)
&=\E_{\mu^{\otimes N}}\Big[e^{B G_N}e^{(\be-B)G_N}\Big]\nn\\
&\le \exp\Big(C_{\mathrm{pt}}(\mu)(\be-B)a_{N,\s}\Big)\K_{N,B}(\mu)
\le C_B(\mu)\exp\Big(C_{\mathrm{pt}}(\mu)(\be-B)a_{N,\s}\Big),
\end{align}
which is \eqref{eq:temperatureInterpolationBound} after increasing $C_B(\mu)$ once more to dominate $C_{\mathrm{pt}}(\mu)$.  Taking $\be=\be_\ast N^\gamma$ gives \eqref{eq:growingTemperatureSpecialization}: for $\gamma\le0$, use \eqref{eq:fixedTemperatureWindowBound} with $B=B_\ast$, while for $\gamma>0$, use \eqref{eq:temperatureInterpolationBound} with any fixed positive $B$ and absorb $B_\ast$ and the fixed-temperature constant into $C_{B_\ast,\mu}$.
\end{proof}

\begingroup

\subsection{Sharpness of the threshold}
\label{subsec:partition-function-sharpness}

This final subsection gives an explicit lower rate, though expected to be nonoptimal, for the divergence beyond the Hilbert--Schmidt threshold.  The proof applies conditional Jensen on a dyadic spatial partition whose resolution is coupled to $N$, and then uses the bounded-feature Laplace estimate from \cref{prop:boundedFeatureQuadLaplace}.

\begin{lemma}[Quantitative dyadic coarse-graining]\label{lem:quantitativeDyadicCompression}
Let $\d/2\le\s<\d$ and $\mu\in L^\infty\cap\P(\R^\d)$.  There exist constants $c_0,C_0,L_0>0$, a dyadic cube $Q_0$, and, for every sufficiently small scale $\delta$ in the dyadic subdivision of $Q_0$, a finite sub-$\sigma$-algebra $\mathcal G_\delta$ of the Borel $\sigma$-algebra such that, with
\begin{align}\label{eq:quantCompressionKernel}
h_\delta
\coloneqq
\E_{\mu^{\otimes2}}\brak*{\tl\g\mid\mathcal G_\delta\otimes\mathcal G_\delta},
\end{align}
the following properties hold:
\begin{enumerate}[(i)]
\item $h_\delta$ is bounded, symmetric, canonical, positive definite, and finite rank.
\item If $T_\delta$ denotes the integral operator with kernel $h_\delta$, then
\begin{align}\label{eq:quantCompressionOperatorBound}
0\le T_\delta\le L_0\,\mathrm{Id}
\end{align}
in the quadratic-form sense.
\item The canonical feature map of $h_\delta$ satisfies
\begin{align}\label{eq:quantCompressionDiagonalBound}
\sup_x h_\delta(x,x)\le C_0\delta^{-\s}.
\end{align}
\item Its Hilbert--Schmidt norm satisfies
\begin{align}\label{eq:quantCompressionHSLower}
\|h_\delta\|_{L^2(\mu^{\otimes2})}^2
\ge c_0
\begin{cases}
\log(1/\delta),&2\s=\d,\\[0.3em]
\delta^{\d-2\s},&2\s>\d.
\end{cases}
\end{align}
\end{enumerate}
\end{lemma}

\begin{proof}
Write $f=\dd\mu/\dd x$.  Choose $a>0$ such that
\begin{align}
E\coloneqq\{x\in\R^\d:f(x)\ge a\}
\end{align}
has positive Lebesgue measure, and let $x_0$ be a Lebesgue density point of $E$.  Since $\s<\d$ and $\mu\in L^1\cap L^\infty$, the potential $h_\mu=\g\ast\mu$ is bounded.  Hence, after choosing a sufficiently small dyadic cube $Q_0$ containing $x_0$, we may arrange that
\begin{align}\label{eq:quantCompressionQ0}
|E\cap Q_0|\ge(1-\varepsilon)|Q_0|,
\qquad
\mu(Q_0)\le\frac12,
\end{align}
and, for $\mu^{\otimes2}$-a.e. $x,y\in Q_0$,
\begin{align}\label{eq:quantCompressionLocalPositive}
\tl\g(x,y)
=\g(x-y)-h_\mu(x)-h_\mu(y)+2I_\mu
\ge\frac12\g(x-y).
\end{align}
Here $\varepsilon>0$ is a sufficiently small dimensional constant.  Indeed, $I_\mu\ge0$, while $\g(z)\to+\infty$ as $z\to0$.

Partition $Q_0$ into dyadic cubes $Q$ of side length $\delta$.  Call $Q$ good if
\begin{align}
|E\cap Q|\ge\frac12|Q|,
\end{align}
and, for each good cube, set $A_Q\coloneqq E\cap Q$.  Since every bad cube contains at least half of its volume in $Q_0\setminus E$, \eqref{eq:quantCompressionQ0} gives
\begin{align}\label{eq:quantCompressionBadCells}
\#\{Q:Q\text{ is bad}\}\,\delta^\d
\le2\varepsilon|Q_0|.
\end{align}
Let $R$ be the complement in $\R^\d$ of the union of the good partition sets $A_Q$, and let $\mathcal G_\delta$ be generated by the finitely many sets $A_Q$ and $R$.  Since $\mu(Q_0)\le1/2$, one has $\mu(R)\ge1/2$.  Moreover, for every good cube,
\begin{align}\label{eq:quantCompressionAtomMass}
\frac a2\delta^\d
\le\mu(A_Q)
\le\|\mu\|_{L^\infty}\delta^\d.
\end{align}

Let $P_\delta$ denote conditional expectation from $L^2(\mu)$ onto the $\mathcal G_\delta$-measurable functions, and let $T_\mu$ be the integral operator with kernel $\tl\g$.  Then $h_\delta$ is the kernel of $P_\delta T_\mu P_\delta$.  It is therefore symmetric, canonical, positive definite, and finite rank.  Also,
\begin{align}
|\tl\g(x,y)|
\le\g(x-y)+h_\mu(x)+h_\mu(y)+2I_\mu,
\end{align}
so
\begin{align}
\sup_x\int_{\R^\d}|\tl\g(x,y)|\,\dd\mu(y)
\le2\|h_\mu\|_{L^\infty}+4I_\mu<\infty.
\end{align}
Schur's test (see, e.g., \cite[Theorem~5.2]{HalmosSunder1978}) yields $\|T_\mu\|_{\mathrm{op}}\le L_0$, and, since $P_\delta$ is an orthogonal projection, \eqref{eq:quantCompressionOperatorBound} follows.

For a good partition set $A_Q$, the definition of conditional expectation, \eqref{eq:quantCompressionAtomMass}, and $\s<\d$ give
\begin{align}
\forall x\in A_Q,\qquad
0\le h_\delta(x,x)
&=\frac{1}{\mu(A_Q)^2}
\iint_{A_Q\times A_Q}\tl\g(u,v)\,\dd\mu(u)\dd\mu(v)\notag\\
&\le C\delta^{-2\d}
\iint_{Q\times Q}|u-v|^{-\s}\,\dd u\,\dd v+C
\le C\delta^{-\s}.
\end{align}
On the remainder set,
\begin{align}
0\le h_\delta(x,x)
\le\mu(R)^{-2}\|\tl\g\|_{L^1(\mu^{\otimes2})}
\le4\|\tl\g\|_{L^1(\mu^{\otimes2})}.
\end{align}
This proves \eqref{eq:quantCompressionDiagonalBound}.

It remains to prove \eqref{eq:quantCompressionHSLower}.  Write $\ell_0$ for the side length of $Q_0$ and $M=\ell_0/\delta$.  Index its dyadic subcubes by the lattice box $\{0,\ldots,M-1\}^\d$.  For every displacement $k\in\Z^\d$ with
\begin{align}
2\le|k|\le cM,
\end{align}
the number of ordered pairs $(Q,Q+k)$ lying in $Q_0$ is at least $cM^\d$.  For fixed $k$, at most twice the number of bad cubes can occur as an endpoint of such a pair.  Thus, by \eqref{eq:quantCompressionBadCells}, after choosing $\varepsilon$ sufficiently small, the number of good--good pairs with displacement $k$ is at least $cM^\d=c\delta^{-\d}$.

If $Q$ and $Q+k$ are good, then \eqref{eq:quantCompressionLocalPositive} and the estimate $|u-v|\le C\delta(1+|k|)$ give
\begin{align}\label{eq:quantCompressionCellLower}
\frac{1}{\mu(A_Q)\mu(A_{Q+k})}
\iint_{A_Q\times A_{Q+k}}
\tl\g(u,v)\,\dd\mu(u)\dd\mu(v)
\ge c\delta^{-\s}(1+|k|)^{-\s}.
\end{align}
Consequently,
\begin{align}
\|h_\delta\|_{L^2(\mu^{\otimes2})}^2
\ge c\delta^{\d-2\s}\sum_{2\le|k|\le c/\delta}(1+|k|)^{-2\s}.
\end{align}
The lattice sum is bounded below by $c\log(1/\delta)$ when $2\s=\d$, and by a positive constant when $2\s>\d$.  This proves \eqref{eq:quantCompressionHSLower}.
\end{proof}

\begin{prop}[Quantitative sharpness of the partition-function threshold]
\label{prop:partition-function-sharpness}
Let $\be>0$, let $\d/2\le\s<\d$, and assume $\mu\in L^\infty\cap\P(\R^\d)$.  There exist constants $c,C>0$ and $N_0\ge3$, depending on $(\be,\mu,\d,\s)$, such that, for every $N\ge N_0$,
\begin{align}\label{eq:quantSharpnessBounds}
\K_{N,\be}(\mu)
\ge C^{-1}
\begin{cases}
(\log N)^c, & \s=\d/2,\\[0.3em]
\exp\left\{c(\log N)^{2-\d/\s}\right\}, & \d/2<\s<\d.
\end{cases}
\end{align}
In particular, $\K_{N,\be}(\mu)\to+\infty$ as $N\to\infty$ throughout $\d/2\le\s<\d$.
\end{prop}

\begin{proof}
Let $\xi_1,\ldots,\xi_N$ be iid with law $\mu$, and set
\begin{align}
W_N
\coloneqq\frac1N\sum_{1\le i\ne j\le N}\tl\g(\xi_i,\xi_j),
\qquad
Z_N\coloneqq\E\brak*{e^{-\frac\be2W_N}}.
\end{align}
Since $h_\mu$ is bounded, the exact one-body decomposition \eqref{eq:roadmapIdentity} gives
\begin{align}\label{eq:quantSharpnessCenteredReduction}
\K_{N,\be}(\mu)
&=e^{\be I_\mu}
\E\left[e^{-\frac\be2W_N+\frac\be N\sum_{i=1}^N(h_\mu(\xi_i)-2I_\mu)}\right]\notag\\
&\ge c_{\be,\mu}Z_N,
\qquad
c_{\be,\mu}
\coloneqq
\exp\left\{\be I_\mu-\be\|h_\mu-2I_\mu\|_{L^\infty}\right\}>0.
\end{align}

Fix a sufficiently small dyadic scale $\delta$ from \cref{lem:quantitativeDyadicCompression}, and set
\begin{align}
W_{N,\delta}
\coloneqq\frac1N\sum_{1\le i\ne j\le N}h_\delta(\xi_i,\xi_j),
\qquad
Z_{N,\delta}
\coloneqq\E\brak*{e^{-\frac\be2W_{N,\delta}}}.
\end{align}
Let $\mathcal F_{N,\delta}$ be the $\sigma$-algebra generated by the events $\{\xi_i\in A_Q\}$ and $\{\xi_i\in R\}$, for $1\le i\le N$ and all good cubes $Q$; equivalently, it records which set in the finite partition $\{A_Q:Q\text{ is good}\}\cup\{R\}$ contains each $\xi_i$.  By \eqref{eq:quantCompressionKernel},
\begin{align}
\E\brak*{W_N\mid\mathcal F_{N,\delta}}
=W_{N,\delta}.
\end{align}
Conditional Jensen therefore gives
\begin{align}\label{eq:quantSharpnessConditionalJensen}
Z_N\ge Z_{N,\delta}.
\end{align}

Let $T_\delta$ be the operator with kernel $h_\delta$.  Since $h_\delta$ is positive definite, the Moore--Aronszajn construction gives a canonical feature map $k_\delta$ (see, e.g., \cite[Theorem~4.21]{SteinwartChristmann2008}); by \eqref{eq:quantCompressionDiagonalBound}, it is bounded and satisfies
\begin{align}
\|k_\delta\|_{L^\infty(\mu;\mathcal H_\delta)}^2
=\sup_x h_\delta(x,x)
\le C\delta^{-\s}.
\end{align}
The bounded specialization \eqref{eq:boundedFeatureSpecialization} of \cref{prop:boundedFeatureQuadLaplace} yields
\begin{align}\label{eq:quantSharpnessFiniteNDeterminant}
\left|Z_{N,\delta}
-\det{}_2(\mathrm{Id}+\be T_\delta)^{-1/2}\right|
\le\frac{C}{N}\exp\left\{C\delta^{-\s}\right\}.
\end{align}
Let $(\lambda_{\delta,j})_j$ denote the nonzero eigenvalues of $T_\delta$.  By \eqref{eq:quantCompressionOperatorBound}, they belong to $[0,L_0]$.  Since
\begin{align}
F_\be(t)\coloneqq\be t-\log(1+\be t)
\end{align}
satisfies
\begin{align}
\inf_{0\le t\le L_0}\frac{F_\be(t)}{t^2}>0,
\end{align}
with the quotient interpreted continuously at $t=0$, one has
\begin{align}\label{eq:quantSharpnessDeterminantLower}
\log\det{}_2(\mathrm{Id}+\be T_\delta)^{-1/2}
=\frac12\sum_jF_\be(\lambda_{\delta,j})
\ge c\sum_j\lambda_{\delta,j}^2
=c\|h_\delta\|_{L^2(\mu^{\otimes2})}^2.
\end{align}

Choose $a>0$ so small that $C2^\s a<1/2$, where $C$ is the constant in the exponential on the right-hand side of \eqref{eq:quantSharpnessFiniteNDeterminant}.  For every sufficiently large $N$, choose a scale $\delta_N$ in the dyadic subdivision of $Q_0$ such that
\begin{align}\label{eq:quantSharpnessScaleChoice}
a\log N
\le\delta_N^{-\s}
\le2^\s a\log N.
\end{align}
After increasing $N_0$ if necessary, \eqref{eq:quantSharpnessFiniteNDeterminant} and \eqref{eq:quantSharpnessScaleChoice} give
\begin{align}
\frac{C}{N}\exp\left\{C\delta_N^{-\s}\right\}
\le\frac12.
\end{align}
The determinant in \eqref{eq:quantSharpnessFiniteNDeterminant} is at least one, so
\begin{align}
Z_{N,\delta_N}
\ge\frac12\det{}_2(\mathrm{Id}+\be T_{\delta_N})^{-1/2}.
\end{align}
Combining this estimate with \eqref{eq:quantSharpnessCenteredReduction}, \eqref{eq:quantSharpnessConditionalJensen}, \eqref{eq:quantSharpnessDeterminantLower}, and \eqref{eq:quantCompressionHSLower} gives
\begin{align}
\log\K_{N,\be}(\mu)
&\ge c\log(1/\delta_N)-C
\ge c\log\log N-C,
&&2\s=\d,\\
\log\K_{N,\be}(\mu)
&\ge c\delta_N^{\d-2\s}-C
\ge c(\log N)^{2-\d/\s}-C,
&&2\s>\d.
\end{align}
Exponentiating proves \eqref{eq:quantSharpnessBounds}.
\end{proof}

\begin{remark}[Expected sharp ultraviolet growth]\label{rem:expectedSharpUltravioletGrowth}
The rates in \cref{prop:partition-function-sharpness} reflect the exponential dependence on the feature diagonal in \eqref{eq:quantSharpnessFiniteNDeterminant} and are not expected to be optimal.  If the density of $\mu$ is bounded below by a positive constant on some ball, the natural prediction is
\begin{align}\label{eq:expectedSharpUltravioletGrowth}
\log\K_{N,\be}(\mu)
\asymp
\begin{cases}
\log N,&\s=\d/2,\\[0.3em]
N^{2-\d/\s},&\d/2<\s<\d.
\end{cases}
\end{align}
The heuristic comes from the singular part of the second Mayer correction,
\begin{align}\label{eq:secondMayerCorrectionHeuristic}
N^2\iint_{(\R^\d)^2}
\left(
 e^{-\be\g(x-y)/N}-1+\frac\be N\g(x-y)
\right)
\,\dd\mu(x)\dd\mu(y).
\end{align}
At fixed inverse temperature, the quadratic Mayer approximation should therefore be cut off at the scale $r\asymp N^{-1/\s}$ at which the single-pair contribution $\be\g(r)/N$ to the Gibbs exponent becomes of order one and ceases to be perturbative.  At $\s=\d/2$, this produces $\int_{N^{-1/\s}}^1r^{-1}\,\dd r\asymp\log N$.  If $\s>\d/2$, the change of variables $x-y=N^{-1/\s}z$ gives the scale $N^{2-\d/\s}$.  The bounded terms introduced by Hoeffding centering do not affect these ultraviolet powers.  The problem of proving matching bounds, identifying the leading constants, and determining whether a renormalized limiting object exists is beyond the scope of the present paper.
\end{remark}
\endgroup

\section{The attractive logarithmic case}
\label{sec:attractive-log-small-beta}

This section proves both assertions of \cref{thm:attractive-log-small-beta-partition}.  The uniform-bound argument is carried out for a general interaction kernel $W$ with the right growth: it uses only the integrability $W\in L^1(\mu^{\otimes2})$ of the kernel, the exponential integrability $e^{2\be h_\mu}\in L^1(\mu)$ of its potential $h_\mu\coloneqq\int_{\R^\d} W(\cdot,y)\,\dd\mu(y)$, and the Gaussian moment growth \eqref{eq:attMomentAssumption} of the centered kernel, and therefore applies verbatim to every interaction satisfying these hypotheses.  The attractive logarithmic kernel $W(x,y)=\log|x-y|$ is the instance appearing in the theorem statement.  In the uniform-bound subsection, we write $h_\mu$, $I_\mu$, $\Fr_N$, $\K_{N,\be}$, and $\tl W$ without the superscript $\mathrm{att}$.  In the quantitative subsection, where the attractive target quantities and the repulsive logarithmic truncations occur simultaneously, we restore the superscript only on the genuinely attractive quantities.  The quantitative determinant argument in \cref{subsec:attractive-quantitative-determinant} uses the additional logarithmic truncation and bounded-density hypotheses.  The section concludes with the periodic uniform-background phase diagram, whose purpose is to illustrate, already in the simplest homogeneous setting, why an $N$-uniform partition-function bound cannot be expected for every inverse temperature.  Under the corresponding small-$\be$ and integrability hypotheses, the commutator, mean-field, and fluctuation consequences follow by the same Donsker--Varadhan, Gronwall, and Gaussian-tilting arguments as in the repulsive case; the commutator input may alternatively be obtained from the argument of Jabin and Wang \cite{JabinWang2018}.

\subsection{Uniform bound at small inverse temperature}
\label{subsec:attractive-uniform-bound}

The uniform-bound proof has two ingredients.  A one-body reduction bounds $\K_{N,\be}(\mu)$ by the centered partition function
\begin{align}\label{eq:attZdef}
Z_{N,\be}(\mu)\coloneqq\E_{\mu^{\otimes N}}\brak[\Big]{e^{-\frac{\be}{2N}S_N}},
\qquad
S_N\coloneqq \sum_{1\le i\ne j\le N}\tl W(\xi_i,\xi_j),
\qquad
\xi_1,\ldots,\xi_N\stackrel{\mathrm{iid}}{\sim}\mu.
\end{align}
A Taylor expansion then reduces the uniform bound on $Z_{N,\be}$ to Gaussian-factorial growth of the even moments of $S_N$, which is proved by decoupling and counting: after the decoupled expectation is expanded, only index patterns in which every active index is repeated survive the centering.

\begin{proof}[Proof of the uniform-bound assertion in \cref{thm:attractive-log-small-beta-partition}]

\step{Step 1: One-body reduction.}  For every $\be\ge0$ with $e^{2\be h_\mu}\in L^1(\mu)$,
\begin{align}\label{eq:KattToZatt}
\K_{N,\be}(\mu)
\le e^{\be|I_\mu|}\,
\E_\mu\brak[\big]{e^{2\be h_\mu}}^{1/2}\,
Z_{N,2\be}(\mu)^{1/2}.
\end{align}

\par\smallskip\noindent  From the definition of $\Fr_N$, with $|I_\mu|\le\frac12\|\g_{\mathrm{att}}\|_{L^1(\mu^{\otimes2})}<\infty$,
\begin{align}
-\be N\Fr_N(\XN,\mu)
=-\frac{\be}{2N}\sum_{i\ne j}\tl W(x_i,x_j)
+\frac{\be}{N}\sum_{i=1}^Nh_\mu(x_i)
-\be I_\mu.
\end{align}
Since the one-body term is an empirical average of iid samples, Jensen's inequality gives\linebreak[4]
$\E_{\mu^{\otimes N}}[e^{\frac{2\be}N\sum_ih_\mu(\xi_i)}]=\E_\mu[e^{\frac{2\be}Nh_\mu}]^N\le\E_\mu[e^{2\be h_\mu}]$.  Therefore, by Cauchy--Schwarz,
\begin{align}
\K_{N,\be}(\mu)
\le e^{\be|I_\mu|}\,
\E_{\mu^{\otimes N}}\brak[\Big]{e^{-\frac{\be}{N}S_N}}^{1/2}
\E_{\mu^{\otimes N}}\brak[\Big]{e^{\frac{2\be}{N}\sum_{i=1}^Nh_\mu(\xi_i)}}^{1/2}
\le e^{\be|I_\mu|}\,Z_{N,2\be}(\mu)^{1/2}\,\E_\mu\brak[\big]{e^{2\be h_\mu}}^{1/2}.
\end{align}

\step{Step 2: Gaussian-factorial moment bound.}  There is a universal constant $C_\ast<\infty$ such that, for all $k\ge1$ and $N\ge1$,
\begin{align}\label{eq:attMomentBound}
\E_{\mu^{\otimes N}}\big[|S_N|^{2k}\big]
\le (C_\ast C_{\mathrm{mom}})^{2k}N^{2k}(2k)!.
\end{align}

\par\smallskip\noindent  The case $N=1$ is trivial, since $S_1=0$; assume $N\ge2$.  The kernel $\tl W$ is canonical by \eqref{eq:tildedef} and Fubini, and integrable by the case $p=1$ of \eqref{eq:attMomentAssumption}, so the moment decoupling of \cref{rem:decoupMoment}, applied with $\Phi(x)=x^{2k}$, and the expansion of the $2k$-th power give
\begin{align}
\E_{\mu^{\otimes N}}\big[|S_N|^{2k}\big]
\le 4^{2k}\,
\E\Bigg[\Big|\sum_{i\ne j}\tl W(\xi_i,\xi'_j)\Big|^{2k}\Bigg]
\le 4^{2k}
\sum_{\substack{1\le i_\ell,j_\ell\le N\\ i_\ell\ne j_\ell,\ 1\le\ell\le2k}}
\Big|\E\brak*{\prod_{\ell=1}^{2k}\tl W(\xi_{i_\ell},\xi'_{j_\ell})}\Big|,
\end{align}
where $\xi'_1,\ldots,\xi'_N$ is an independent copy of $\xi_1,\ldots,\xi_N$.  Fix a multi-index and set $a_m\coloneqq\#\{\ell:i_\ell=m\}$ and $b_m\coloneqq\#\{\ell:j_\ell=m\}$.  By the same centering, the expectation vanishes unless every active $a_m$ and $b_m$ is at least $2$; call such multi-indices \emph{surviving}.  For a surviving multi-index, condition on $\xi'_1,\ldots,\xi'_N$: by the independence of the unprimed variables, the conditional expectation factorizes over the active indices $m$ into factors $\E_{\xi_m}[\prod_{\ell:i_\ell=m}\tl W(\xi_m,\xi'_{j_\ell})]$ with $a_m$ terms each.  H\"older's inequality in $\xi_m$ with the $a_m$ equal exponents $(a_m,\ldots,a_m)$, followed by the moment hypothesis \eqref{eq:attMomentAssumption} with $p=a_m$, applied uniformly in the frozen variables $\xi'_{j_\ell}$, gives
\begin{align}
\Big|\E_{\xi_m}\brak[\Big]{\prod_{\ell:i_\ell=m}\tl W(\xi_m,\xi'_{j_\ell})}\Big|
\le\prod_{\ell:i_\ell=m}
\pa[\Big]{\int_{\R^\d}|\tl W(x,\xi'_{j_\ell})|^{a_m}\,\dd\mu(x)}^{1/a_m}
\le C_{\mathrm{mom}}^{a_m}\,a_m!.
\end{align}
Multiplying over the active $m$ and using $\sum_ma_m=2k$,
\begin{align}
\Big|\E\brak*{\prod_{\ell=1}^{2k}\tl W(\xi_{i_\ell},\xi'_{j_\ell})}\Big|
\le C_{\mathrm{mom}}^{2k}\prod_{m=1}^N a_m!.
\end{align}
Since the $i$-words with multiplicity vector $\alpha\in\N_0^N$, $\sum_m\alpha_m=2k$, number $(2k)!/\prod_m\alpha_m!$ and each contributes $\prod_m\alpha_m!$, dropping the coupling $i_\ell\ne j_\ell$ factorizes the count:
\begin{align}\label{eq:attCountingFactorized}
\sum_{\text{surviving }(i,j)}\ \prod_{m=1}^N a_m!
\le (2k)!\,\cdot\,\#\{\text{surviving }\alpha\}\,\cdot\,\#\{\text{surviving }j\text{-words}\}.
\end{align}
If $N<4k$, we discard the survival constraints: there are at most $N^{2k}$ words and $\binom{2k+N-1}{N-1}\le2^{6k}$ multiplicity vectors, so \eqref{eq:attCountingFactorized} is at most $2^{6k}N^{2k}(2k)!$.  If $N\ge4k$, a surviving word or multiplicity vector has at most $k$ active letters, all repeated; choosing the $r\le k$ active letters and assigning positions crudely,
\begin{align}\label{eq:attCombinatorialBounds}
\#\{\text{surviving }j\text{-words}\}
&\le\sum_{r=1}^{k}\binom Nr\, r^{2k}
\le C^kN^kk^k,\nn\\
\#\{\text{surviving }\alpha\}
&\le\sum_{r=1}^{k}\binom Nr\binom{2k-r-1}{r-1}
\le C^k\frac{N^k}{k^k}.
\end{align}
Indeed, writing $r=\theta k$ and $N=xk$ with $0<\theta\le1$ and $x\ge4$, and using $\binom Nr\le(eN/r)^r$, $y^\theta\le y$ for $y\ge1$, and $\sup_{(0,1]}\theta^{-2\theta}<\infty$, the summands are bounded by $(ex/\theta)^{\theta k}k^{2k}\le C^kN^kk^k$ and by $(eN/r)^r(2ek/r)^r=(2e^2Nk/r^2)^r\le(Cx)^k=C^kN^k/k^k$, respectively, and the factors $k$ from the sums are absorbed into $C^k$.  The factors $k^{\pm k}$ cancel in \eqref{eq:attCountingFactorized}, which is again at most $C^kN^{2k}(2k)!$.  Taking $C_\ast\coloneqq4\max\brac[\big]{2^{3},\sqrt C}$ proves \eqref{eq:attMomentBound}.

\step{Step 3: Summation.}  Since $e^x\le e^x+e^{-x}=2\sum_{k\ge0}x^{2k}/(2k)!$, the moment bound \eqref{eq:attMomentBound} gives, for every $0\le\be<\be_0$, with $\be_0=1/(C_\ast C_{\mathrm{mom}})$ as in \eqref{eq:attBeta0Def} and $C_\ast$ the constant of Step 2, and uniformly in $N$,
\begin{align}\label{eq:ZattMomentExpansion}
Z_{N,2\be}(\mu)
\le
2\sum_{k=0}^{\infty}
\frac{\be^{2k}}{N^{2k}(2k)!}
\E_{\mu^{\otimes N}}\big[|S_N|^{2k}\big]
\le 2\sum_{k=0}^{\infty}\pa[\big]{\be C_\ast C_{\mathrm{mom}}}^{2k}
=\frac{2}{1-\pa[\big]{\be C_\ast C_{\mathrm{mom}}}^2},
\end{align}
which, combined with \eqref{eq:KattToZatt}, proves the theorem with
\begin{align}\label{eq:CubAttExplicit}
C_{\mathrm{ub}}^{\mathrm{att}}(\be,\mu)
=e^{\be|I_\mu|}\,
\E_\mu\brak[\big]{e^{2\be h_\mu}}^{1/2}
\pa*{\frac{2}{1-\pa[\big]{\be C_\ast C_{\mathrm{mom}}}^2}}^{1/2}.
\end{align}
\end{proof}

\begingroup
\subsection{Quantitative convergence to the determinant}
\label{subsec:attractive-quantitative-determinant}

To prove the quantitative attractive result, we need a fixed-kernel estimate for the positive quadratic exponential $\E\brak*{\exp\{\be\mathsf W_N(q)/2\}}$.  This is the counterpart of the bounded-feature estimate in \cref{prop:boundedFeatureQuadLaplace}, with $\mathrm{Id}+\be T_q$ replaced by $\mathrm{Id}-\be T_q$.  The proof requires the stronger subcritical operator bound $2\be\|T_q\|_{\mathrm{op}}<1$, which provides the second exponential moment used below.

Let $q\in L^2(\mu^{\otimes2})$ be symmetric, canonical, and nonnegative definite, and let
\begin{align}
q(x,y)&=\langle \mathsf k(x),\mathsf k(y)\rangle_{\mathcal H},
\qquad
\int_{\R^\d} \mathsf k(x)\,\dd\mu(x)=0,
\label{eq:positiveSignFeatureRepresentation}\\
D(x)&\coloneqq\|\mathsf k(x)\|_{\mathcal H}^2,
\qquad
\mathsf W_N(q)\coloneqq\frac1N\sum_{1\le i\ne j\le N}q(\xi_i,\xi_j),
\label{eq:positiveSignFeatureDefinitions}
\end{align}
be a centered feature representation on a separable real Hilbert space $\mathcal H$, where $\xi_1,\ldots,\xi_N$ are independent with law $\mu$.  Let $T_q$ denote the Hilbert--Schmidt integral operator with kernel $q$.

\begin{prop}[Feature-space Laplace estimate for a positive quadratic exponential]\label{prop:positiveSignFeatureLaplace}
Fix $\be>0$ and assume that
\begin{align}\label{eq:positiveSignFeatureAssumptions}
2\be\|T_q\|_{\mathrm{op}}<1,
\qquad
\sup_{M\ge1}\E\brak*{e^{\be\mathsf W_M(q)}}<\infty,
\qquad
\int_{\R^\d} e^{aD(x)}\,\dd\mu(x)<\infty
\end{align}
for some $a>0$.  Then there is a finite constant $C$ such that, for every $N\ge1$,
\begin{align}\label{eq:positiveSignFeatureRate}
\left|
\E\brak*{e^{\frac\be2\mathsf W_N(q)}}
-\det_2\pa*{\mathrm{Id}-\be T_q}^{-1/2}
\right|
\le C N^{-1/2}.
\end{align}
For fixed $\be$ and $a$, the constant may be chosen to depend polynomially on $1+\int_{\R^\d} e^{aD(x)}\,\dd\mu(x)$, with coefficients depending only on the subcritical margin $1-2\be\|T_q\|_{\mathrm{op}}$, $\|T_q\|_{\mathrm{HS}}$, and the uniform exponential moment in \eqref{eq:positiveSignFeatureAssumptions}.
\end{prop}

\begin{proof}
Let $\mathsf Z$ be an isonormal Gaussian process over $\mathcal H$, independent of $\xi_1,\ldots,\xi_N$, and write $\mathsf Z_x\coloneqq\mathsf Z(\mathsf k(x))$.  Set
\begin{align}\label{eq:positiveSignANdef}
A_N(\mathsf Z)
\coloneqq
\int_{\R^\d}
\exp\left\{
\sqrt{\frac\be N}\,\mathsf Z_x
-\frac\be{2N}D(x)
\right\}
\,\dd\mu(x).
\end{align}
Since
\begin{align}\label{eq:positiveSignQuadraticIdentity}
\mathsf W_N(q)
=
\left\|\frac1{\sqrt N}\sum_{i=1}^N\mathsf k(\xi_i)\right\|_{\mathcal H}^2
-\frac1N\sum_{i=1}^ND(\xi_i),
\end{align}
the Hubbard--Stratonovich transformation and Fubini's theorem give
\begin{align}\label{eq:positiveSignFeatureLinearization}
\E\brak*{e^{\frac\be2\mathsf W_N(q)}}
=\E_{\mathsf Z}\brak*{A_N(\mathsf Z)^N}.
\end{align}

For fixed $\mathsf Z$, define
\begin{align}\label{eq:positiveSignfDef}
f(\tau)
\coloneqq
\int_{\R^\d}
\exp\left\{
\tau\sqrt\be\,\mathsf Z_x
-\frac{\tau^2\be}{2}D(x)
\right\}
\,\dd\mu(x).
\end{align}
Then $A_N(\mathsf Z)=f(N^{-1/2})$, $f(0)=1$, and $f'(0)=0$.  Introduce
\begin{align}\label{eq:positiveSignBRdef}
B&\coloneqq\frac\be2\int_{\R^\d}\big(\mathsf Z_x^2-D(x)\big)\,\dd\mu(x),\nn\\
R&\coloneqq\frac{\be^{3/2}}6\int_{\R^\d}
\big(\mathsf Z_x^3-3D(x)\mathsf Z_x\big)\,\dd\mu(x).
\end{align}
Taylor's formula with integral remainder yields
\begin{align}\label{eq:positiveSignFeatureTaylor}
A_N(\mathsf Z)=1+\frac BN+\frac R{N^{3/2}}+\frac{E_N}{N^2},
\end{align}
where
\begin{align}\label{eq:positiveSignENdef}
E_N\coloneqq\frac16\int_0^1(1-u)^3 f^{(4)}(uN^{-1/2})\,\dd u.
\end{align}
The exponential moment of $D$ implies, after discarding finitely many $N$ and enlarging the final constant, that
\begin{align}\label{eq:positiveSignFeatureMoments}
\sup_N\E_{\mathsf Z}\brak*{|B|^4+|R|^4+|E_N|^4}<\infty.
\end{align}
Indeed, the bounds for $B$ and $R$ follow from Gaussian moment identities and Minkowski's inequality.  For the remainder, write
\begin{align}\label{eq:positiveSigngxDef}
g_x(\tau)\coloneqq\tau\sqrt\be\,\mathsf Z_x-\frac{\tau^2\be}{2}D(x).
\end{align}
Then
\begin{align}\label{eq:positiveSignFourthDerivative}
(e^{g_x})^{(4)}
=e^{g_x}\big((g_x')^4+6(g_x')^2g_x''+3(g_x'')^2\big),
\qquad
g_x'=\sqrt\be\,\mathsf Z_x-\tau\be D(x),
\qquad
g_x''=-\be D(x).
\end{align}
For all sufficiently large $N$, Gaussian moment estimates give
\begin{align}\label{eq:positiveSignFourthDerivativeBound}
\sup_{0\le\tau\le N^{-1/2}}
\E_{\mathsf Z}\brak*{\left|(e^{g_x})^{(4)}(\tau)\right|^4}
\le C_{\be,a}\,e^{aD(x)/2}\big(1+D(x)^{16}\big).
\end{align}
The right-hand side is integrable after decreasing $a$ if necessary and absorbing the polynomial factor into the exponential.  Jensen's inequality in $x$ then gives the asserted uniform fourth-moment bound for $E_N$.

Write $A_N=A_N(\mathsf Z)$ and set $G_N\coloneqq\{|A_N-1|\le1/2\}$.  Using $|\log(1+y)-y|\le2y^2$ for $|y|\le1/2$, together with \eqref{eq:positiveSignFeatureTaylor}, gives, on $G_N$,
\begin{align}\label{eq:positiveSignDeltaPointwise}
|N\log A_N-B|
\le\frac{|R|}{\sqrt N}+\frac{|E_N|}{N}+2N|A_N-1|^2.
\end{align}
Consequently, \eqref{eq:positiveSignFeatureMoments} implies
\begin{align}\label{eq:positiveSignFeatureDeltaBound}
\left\|(N\log A_N-B)\mathbf1_{G_N}\right\|_{L^2(\mathsf Z)}
\le C N^{-1/2}.
\end{align}
Moreover, Markov's inequality and \eqref{eq:positiveSignFeatureTaylor}--\eqref{eq:positiveSignFeatureMoments} yield
\begin{align}\label{eq:positiveSignFeatureGoodComplement}
\mathbb P_{\mathsf Z}(G_N^c)
\le16\E_{\mathsf Z}\brak*{|A_N-1|^4}
\le C N^{-4}.
\end{align}

We next prove the estimate
\begin{align}\label{eq:positiveSignFeatureApproximation}
\left|\E_{\mathsf Z}\brak*{A_N^N}-\E_{\mathsf Z}\brak*{e^B}\right|
\le C N^{-1/2}.
\end{align}
The two second moments used below are uniformly bounded:
\begin{align}
\E_{\mathsf Z}\brak*{A_N^{2N}}
&=\E\brak*{e^{\be\mathsf W_{2N}(q)}}\le C,
\label{eq:positiveSignFeatureL2Identity}\\
\E_{\mathsf Z}\brak*{e^{2B}}
&=\det_2\pa*{\mathrm{Id}-2\be T_q}^{-1/2}<\infty.
\label{eq:positiveSignBSecondMoment}
\end{align}
The first identity follows by introducing $2N$ independent samples in $A_N^{2N}$ and applying the Hubbard--Stratonovich transformation; the bound is \eqref{eq:positiveSignFeatureAssumptions}.  The second follows by diagonalizing the covariance operator of $\mathsf k(\xi)$.

On $G_N$, one has $|A_N^N-e^B|\le|N\log A_N-B|(A_N^N+e^B)$.  Hence, by Cauchy--Schwarz and the preceding estimates,
\begin{align}\label{eq:positiveSignDifferenceGoodSet}
\E_{\mathsf Z}\brak*{|A_N^N-e^B|\mathbf1_{G_N}}
&\le
\left\|(N\log A_N-B)\mathbf1_{G_N}\right\|_{L^2(\mathsf Z)}
\pa*{\E_{\mathsf Z}\brak*{A_N^{2N}}^{1/2}+\E_{\mathsf Z}\brak*{e^{2B}}^{1/2}}
\nn\\
&\le C N^{-1/2}.
\end{align}
On $G_N^c$, Cauchy--Schwarz and \eqref{eq:positiveSignFeatureGoodComplement} give
\begin{align}\label{eq:positiveSignDifferenceBadSet}
\E_{\mathsf Z}\brak*{(A_N^N+e^B)\mathbf1_{G_N^c}}
&\le
\pa*{\E_{\mathsf Z}\brak*{A_N^{2N}}^{1/2}+\E_{\mathsf Z}\brak*{e^{2B}}^{1/2}}
\mathbb P_{\mathsf Z}(G_N^c)^{1/2}
\nn\\
&\le C N^{-2}.
\end{align}
This proves \eqref{eq:positiveSignFeatureApproximation}.

Finally, if $(\lambda_j)_{j\ge1}$ are the nonzero eigenvalues of $T_q$, then
\begin{align}\label{eq:positiveSignBChaos}
B=\frac\be2\sum_{j\ge1}\lambda_j(Z_j^2-1)
\end{align}
in $L^2$, and $\be\|T_q\|_{\mathrm{op}}<1$ gives
\begin{align}\label{eq:positiveSignBDeterminant}
\E_{\mathsf Z}\brak*{e^B}
=\prod_{j\ge1}e^{-\be\lambda_j/2}(1-\be\lambda_j)^{-1/2}
=\det_2\pa*{\mathrm{Id}-\be T_q}^{-1/2}.
\end{align}
Combining \eqref{eq:positiveSignFeatureLinearization}, \eqref{eq:positiveSignFeatureApproximation}, and \eqref{eq:positiveSignBDeterminant} completes the proof.
\end{proof}

\begin{proof}[Proof of the quantitative-convergence assertion in \cref{thm:attractive-log-small-beta-partition}]
Write $\g=-\g_{\mathrm{att}}=-\log|\cdot|$ for the repulsive logarithmic kernel, so that
\begin{align}\label{eq:attractivePositiveKernelDef}
\tl\g=-\tl\g_{\mathrm{att}}
\end{align}
is canonical and nonnegative definite.  In this proof, the superscript $\mathrm{att}$ is restored on the target attractive quantities, while $\g$, $\g_\eta$, $\f_\eta$, $\g_{\eta,t}$, $W_N$, $W_{N,\eta}$, and $W_{N,\eta,t}$ retain the repulsive logarithmic notation established in \cref{sec:truncation-estimates,subsec:uniform-bound-log-riesz}.  Let $T_\mu$, $T_{\mu,\eta}$, and $T_{\mu,\eta,t}$ denote the operators with kernels $\tl\g$, $\tl\g_\eta$, and $\tl\g_{\eta,t}$, respectively.  By linearity of Hoeffding centering,
\begin{align}\label{eq:attractiveInterpolatedKernels}
\tl\g_{\eta,t}=\tl\g_\eta+t\tl\f_\eta
=\tl\g-(1-t)\tl\f_\eta,
\qquad 0\le t\le1.
\end{align}

\textbf{Uniform exponential moments along the interpolation.}
The logarithmic truncation satisfies
\begin{align}\label{eq:attractiveRemainderFactorial}
\sup_{0<\eta\le\eta_0}\sup_{y\in\R^\d}
\int_{\R^\d}|\tl\f_\eta(x,y)|^p\,\dd\mu(x)
\le C_0^p p!
\qquad(p\ge1).
\end{align}
Indeed, by the centering identity,
\begin{align}\label{eq:attractiveCenteredRemainderExpansion}
\tl\f_\eta(x,y)
=\f_\eta(x-y)
-(\f_\eta*\mu)(x)
-(\f_\eta*\mu)(y)
+\iint_{(\R^\d)^2}\f_\eta(z-w)\,\dd\mu(z)\dd\mu(w).
\end{align}
The bounded-density assumption and the logarithmic truncation estimate give
\begin{align}\label{eq:attractiveUncenteredRemainderMoment}
\sup_{y\in\R^\d}\int_{\R^\d}\f_\eta(x-y)^p\,\dd\mu(x)
\le\|\mu\|_{L^\infty}\int_{\R^\d}\f_\eta(z)^p\,\dd z
\le C^p p!.
\end{align}
For $|z|\le\eta$, the last estimate follows from
\begin{align}\label{eq:logFactorialIntegral}
\int_0^1(1+|\log r|)^p r^{\d-1}\,\dd r\le C_\d^p p!,
\end{align}
while the polynomial decay in \eqref{eq:logfetaDecay} gives the bound on $|z|>\eta$.  The convolution terms are uniformly bounded by $C\eta^\d$ by \eqref{eq:logfetaConv}, proving \eqref{eq:attractiveRemainderFactorial}.  Since \eqref{eq:attractiveInterpolatedKernels} gives $\tl\g_{\eta,t}=\tl\g-(1-t)\tl\f_\eta$, combining \eqref{eq:attractiveRemainderFactorial} with \eqref{eq:attMomentAssumption} yields
\begin{align}\label{eq:attractiveInterpolatedFactorial}
\sup_{0<\eta\le\eta_0}\sup_{0\le t\le1}\sup_{y\in\R^\d}
\int_{\R^\d}|\tl\g_{\eta,t}(x,y)|^p\,\dd\mu(x)
\le C_{\mathrm{tr}}^p p!
\qquad(p\ge1)
\end{align}
for a finite $C_{\mathrm{tr}}$ depending only on the displayed data.

The factorial-moment argument used for the uniform-bound assertion applies uniformly to the kernels $-\g_{\eta,t}$.  Decreasing the threshold $\be_{\mathrm q}$ in \eqref{eq:attQuantitativeThreshold}, if necessary, we may therefore arrange that, for every $0<\be<\be_{\mathrm q}$,
\begin{align}\label{eq:attractiveInterpolatedExponentialMoment}
\sup_{N\ge1}\sup_{0<\eta\le\eta_0}\sup_{0\le t\le1}
\E_{\mu^{\otimes N}}\brak*{e^{\be W_{N,\eta,t}}}<\infty.
\end{align}
The case $p=2$ of \eqref{eq:attractiveInterpolatedFactorial} also bounds the Hilbert--Schmidt, hence operator, norm of $T_{\mu,\eta,t}$ uniformly in $\eta$ and $t$.  Decreasing $\be_{\mathrm q}$ once more if needed gives the common subcritical operator-norm bound
\begin{align}\label{eq:attractiveUniformSpectralGap}
\sup_{0<\eta\le\eta_0}\sup_{0\le t\le1}
2\be\|T_{\mu,\eta,t}\|_{\mathrm{op}}<1.
\end{align}
These choices may be made with $\be_{\mathrm q}\le\be_0/2$, as required in \eqref{eq:attQuantitativeThreshold}.

\textbf{Finite-particle truncation error.}
Using the uniform exponential-moment bound \eqref{eq:attractiveInterpolatedExponentialMoment}, differentiation in $t$ and Cauchy--Schwarz give
\begin{align}\label{eq:attractiveFiniteParticleInterpolation}
\left|\frac{\dd}{\dd t}\E_{\mu^{\otimes N}}\brak*{e^{\frac\be2 W_{N,\eta,t}}}\right|
\le\frac\be2\E_{\mu^{\otimes N}}\brak*{|W_N-W_{N,\eta}|e^{\frac\be2 W_{N,\eta,t}}}
\le C\E_{\mu^{\otimes N}}\brak*{(W_N-W_{N,\eta})^2}^{1/2}.
\end{align}
Canonicality gives the exact identity
\begin{align}\label{eq:attractiveCanonicalVariance}
\E_{\mu^{\otimes N}}\brak*{(W_N-W_{N,\eta})^2}
=\frac{2(N-1)}N\|\tl\f_\eta\|_{L^2(\mu^{\otimes2})}^2,
\end{align}
and the logarithmic truncation estimate gives
\begin{align}\label{eq:attractiveL2Truncation}
\|\tl\f_\eta\|_{L^2(\mu^{\otimes2})}\le C\eta^{\d/2}.
\end{align}
Integrating \eqref{eq:attractiveFiniteParticleInterpolation} over $t\in[0,1]$ therefore yields
\begin{align}\label{eq:attractiveFiniteParticleTruncation}
\left|
\E_{\mu^{\otimes N}}\brak*{e^{\frac\be2 W_N}}
-\E_{\mu^{\otimes N}}\brak*{e^{\frac\be2 W_{N,\eta}}}
\right|
\le C\eta^{\d/2}.
\end{align}

\textbf{Determinant truncation error.}
The operators $T_\mu$ and $T_{\mu,\eta}$ are nonnegative and satisfy the common operator-norm bound \eqref{eq:attractiveUniformSpectralGap}.  For a nonnegative self-adjoint Hilbert--Schmidt operator $T$ satisfying this bound,
\begin{align}\label{eq:attractiveDeterminantExpansion}
-\frac12\log\det_2\pa*{\mathrm{Id}-\be T}
=\frac12\sum_{m=2}^\infty\frac{\be^m}{m}\operatorname{Tr}(T^m).
\end{align}
The telescoping identity for $T^m-S^m$, together with the uniform Hilbert--Schmidt bounds, shows that the left-hand side is locally Lipschitz in Hilbert--Schmidt norm on the set of nonnegative self-adjoint Hilbert--Schmidt operators satisfying the common operator-norm bound.  Exponentiating on the resulting bounded range gives
\begin{align}\label{eq:attractiveDeterminantTruncation}
\left|
\det_2\pa*{\mathrm{Id}-\be T_\mu}^{-1/2}
-\det_2\pa*{\mathrm{Id}-\be T_{\mu,\eta}}^{-1/2}
\right|
\le C\|T_\mu-T_{\mu,\eta}\|_{\mathrm{HS}}
\le C\eta^{\d/2}.
\end{align}

\textbf{Quantitative fixed-truncation limit.}
Let $\mathsf k_\eta$ be the canonical feature map of $\tl\g_\eta$ and set $D_\eta(x)\coloneqq\|\mathsf k_\eta(x)\|^2$.  The diagonal identity is
\begin{align}\label{eq:attractiveFeatureDiagonalIdentity}
D_\eta(x)
=\g_\eta(0)-2(\g_\eta*\mu)(x)+2I_\mu^{(\eta)},
\end{align}
where $I_\mu^{(\eta)}\coloneqq\frac12\iint_{(\R^\d)^2}\g_\eta(z-w)\,\dd\mu(z)\dd\mu(w)$.  Since $\g_\eta=\g-\f_\eta$ and $\g=-\g_{\mathrm{att}}$, the truncation bounds imply
\begin{align}\label{eq:attractiveFeatureDiagonalBound}
0\le D_\eta(x)
\le C\big(1+\log(1/\eta)\big)+2h_\mu^{\mathrm{att}}(x).
\end{align}
The assumption $e^{4\be h_\mu^{\mathrm{att}}}\in L^1(\mu)$ therefore supplies the feature-diagonal exponential moment required in \cref{prop:positiveSignFeatureLaplace}, with a bound $C\eta^{-A}$ for some finite $A=A(\be,\mu)$.  Applying \cref{prop:positiveSignFeatureLaplace} and using \eqref{eq:attractiveInterpolatedExponentialMoment} and \eqref{eq:attractiveUniformSpectralGap} gives
\begin{align}\label{eq:attractiveFixedTruncationRate}
\left|
\E_{\mu^{\otimes N}}\brak*{e^{\frac\be2 W_{N,\eta}}}
-\det_2\pa*{\mathrm{Id}-\be T_{\mu,\eta}}^{-1/2}
\right|
\le C\eta^{-A}N^{-1/2}.
\end{align}
Combining \eqref{eq:attractiveFiniteParticleTruncation}, \eqref{eq:attractiveDeterminantTruncation}, and \eqref{eq:attractiveFixedTruncationRate} gives
\begin{align}\label{eq:attractiveCenteredRate}
\left|
\E_{\mu^{\otimes N}}\brak*{e^{\frac\be2 W_N}}
-\det_2\pa*{\mathrm{Id}-\be T_\mu}^{-1/2}
\right|
\le C\left(\eta^{\d/2}+\eta^{-A}N^{-1/2}\right).
\end{align}
Choose $\eta=N^{-1/(2A+\d)}$ when this is at most $\eta_0$, and enlarge the constant to absorb the remaining finite set of $N$.  Then the right-hand side of \eqref{eq:attractiveCenteredRate} is bounded by $C(e+N)^{-c}$, where
\begin{align}\label{eq:attractivePowerExponent}
c\coloneqq\frac{\d}{2(2A+\d)}>0.
\end{align}

\textbf{Restoring the one-body term.}
Set
\begin{align}\label{eq:attractiveOneBodyFluctuation}
\psi_\mu^{\mathrm{att}}\coloneqq h_\mu^{\mathrm{att}}-2I_\mu^{\mathrm{att}},
\qquad
R_N\coloneqq\frac1N\sum_{i=1}^N\psi_\mu^{\mathrm{att}}(\xi_i).
\end{align}
The exact splitting identity is
\begin{align}\label{eq:attractiveExactCenteredSplitting}
\K_{N,\be}^{\mathrm{att}}(\mu)
=e^{\be I_\mu^{\mathrm{att}}}
\E_{\mu^{\otimes N}}\brak*{e^{\frac\be2 W_N+\be R_N}}.
\end{align}
The bounded-density assumption gives the uniform lower bound
\begin{align}\label{eq:attractivePotentialLowerBound}
h_\mu^{\mathrm{att}}(x)
\ge\|\mu\|_{L^\infty}\int_{|z|<1}\log|z|\,\dd z> -\infty,
\end{align}
because the contribution from $|x-y|\ge1$ is nonnegative.  Together with $e^{4\be h_\mu^{\mathrm{att}}}\in L^1(\mu)$, this implies $\psi_\mu^{\mathrm{att}}\in L^2(\mu)$ and
\begin{align}\label{eq:attractiveOneBodyVariance}
\E_{\mu^{\otimes N}}\brak*{R_N^2}
=\frac{\|\psi_\mu^{\mathrm{att}}\|_{L^2(\mu)}^2}{N}.
\end{align}
Using $|e^z-1|\le|z|(1+e^z)$ and Cauchy--Schwarz, we obtain
\begin{align}\label{eq:attractiveOneBodyError}
&\left|
\K_{N,\be}^{\mathrm{att}}(\mu)
-e^{\be I_\mu^{\mathrm{att}}}
\E_{\mu^{\otimes N}}\brak*{e^{\frac\be2 W_N}}
\right|\nn\\
&\qquad\le C\E_{\mu^{\otimes N}}\brak*{R_N^2}^{1/2}
\pa*{
\E_{\mu^{\otimes N}}\brak*{e^{\be W_N}}^{1/2}
+\E_{\mu^{\otimes N}}\brak*{e^{\be W_N+2\be R_N}}^{1/2}}
\le C N^{-1/2}.
\end{align}
The first exponential moment is controlled by \eqref{eq:attractiveInterpolatedExponentialMoment}.  Up to the deterministic factor $e^{-2\be I_\mu^{\mathrm{att}}}$, the second is the attractive modulated partition function at inverse temperature $2\be$, which is uniformly bounded by the first assertion of \cref{thm:attractive-log-small-beta-partition}, since $\be_{\mathrm q}\le\be_0/2$ and $e^{4\be h_\mu^{\mathrm{att}}}\in L^1(\mu)$ is exactly the required one-body integrability at that temperature.

Finally, $T_\mu=-T_\mu^{\mathrm{att}}$, and therefore
\begin{align}\label{eq:attractiveDeterminantSign}
\det_2\pa*{\mathrm{Id}-\be T_\mu}
=\det_2\pa*{\mathrm{Id}+\be T_\mu^{\mathrm{att}}}.
\end{align}
Combining \eqref{eq:attractiveCenteredRate}, \eqref{eq:attractiveOneBodyError}, and \eqref{eq:attractiveDeterminantSign} proves \eqref{eq:attQuantitativeDeterminantRate}.
\end{proof}

\begin{remark}\label{rem:attractiveQuantitativeRate}
The rate in \eqref{eq:attQuantitativeDeterminantRate} is deliberately stated as an unspecified power.  The proof tracks an explicit exponent through the exponential moment of the truncated feature diagonal, but that exponent is not expected to be sharp.  For a periodic uniform background, the one-body term vanishes and the hypotheses involving $h_\mu^{\mathrm{att}}$ are automatic.
\end{remark}
\endgroup

\begingroup
\subsection{Periodic uniform backgrounds: phase diagram and lower bounds}
\label{subsec:periodic-attractive-phase-diagram}

We now specialize to the periodic attractive logarithmic interaction with uniform background.  Let $\mu_{\mathrm{unif}}$ denote normalized Lebesgue measure on $\T^\d$, and let $\g_\d$ be the zero-average repulsive periodic logarithmic kernel normalized by
\begin{align}\label{eq:periodicLogNormalization}
(-\Delta)^{\d/2}\g_\d
=\mathsf c_\d(\delta_0-1),
\qquad
\mathsf c_\d\coloneqq
\frac{\Gamma(\d/2)(4\pi)^{\d/2}}{2},
\end{align}
so that
\begin{align}\label{eq:periodicLogFourier}
\widehat{\g_\d}(k)
=\frac{\mathsf c_\d}{(2\pi|k|)^\d},
\qquad k\in\Z^\d\setminus\{0\}.
\end{align}
The attractive interaction is $-\g_\d$, and its modulated partition function around $\mu_{\mathrm{unif}}$ is
\begin{align}\label{eq:periodicAttractivePartition}
\K_{N,\be}(\mu_{\mathrm{unif}})
\coloneqq
\int_{(\T^\d)^N}
\exp\Bigg\{\frac{\be}{2N}
\sum_{1\le i\ne j\le N}\g_\d(x_i-x_j)\Bigg\}
\,\dd\mu_{\mathrm{unif}}^{\otimes N}(\XN).
\end{align}
Let $T_\d$ be convolution by $\g_\d$ on the mean-zero subspace of $L^2(\mu_{\mathrm{unif}})$.  If a probability measure $\mu$ is absolutely continuous with respect to $\mu_{\mathrm{unif}}$, we continue the standing abuse of notation and also write $\mu$ for its density.  Set
\begin{align}\label{eq:periodicAttractiveFreeEnergy}
\mathcal J_{\be,\d}(\mu)
\coloneqq
\Hr(\mu\vert\mu_{\mathrm{unif}})
-\frac{\be}{2}
\left\langle \mu-1,T_\d(\mu-1)\right\rangle_{L^2(\mu_{\mathrm{unif}})}.
\end{align}
We denote by
\begin{align}\label{eq:betaGmDef}
\be_{\mathrm{gm}}(\d)
\coloneqq
\sup\Big\{\be\in[0,2\d):
\mathcal J_{\be,\d}(\mu)\ge0
\text{ for every probability measure }\mu\ll\mu_{\mathrm{unif}}\Big\}
\end{align}
the global-minimization threshold of the uniform phase.  The largest eigenvalue of $T_\d$ equals $(\be_{\mathrm{s}}(\d))^{-1}$, where
\begin{align}\label{eq:betaSDef}
\be_{\mathrm{s}}(\d)
\coloneqq
\frac{(2\pi)^\d}{\mathsf c_\d},
\end{align}
and its eigenspace is
\begin{align}\label{eq:lowestFourierEigenspace}
E_1
\coloneqq
\operatorname{span}\left\{\cos(2\pi x_j),\sin(2\pi x_j):1\le j\le\d\right\}.
\end{align}
Thus, $E_1$ is the $2\d$-dimensional lowest nonzero Fourier eigenspace, corresponding to the frequencies $k=\pm e_j$.  The value $\be_{\mathrm{s}}(\d)$ is the nonlinear stability threshold identified in \cite{chodrondecourcelAttractiveLogGas2025}; moreover, $\be_{\mathrm{s}}(\d)<2\d$ exactly in integer dimensions $\d\ge11$.

For the formulation of part (iii) below, decompose a density near the uniform state as $\mu-1=u+w$, with $u\in E_1$ and $w\in E_1^\perp$.  Solving the $E_1^\perp$ component of the Euler--Lagrange equation for $w=w(u)$ near the origin gives the finite-dimensional reduced free energy
\begin{align}\label{eq:periodicReducedFreeEnergy}
\mathcal J_{\be,\d}^{\mathrm{red}}(u)
\coloneqq
\mathcal J_{\be,\d}(1+u+w(u));
\end{align}
at $\be=\be_{\mathrm{s}}(\d)$, its first nonconstant term is quartic.

In dimension one, the circular Selberg--Dyson integral \cite{Dyson1962III} gives the exact formula
\begin{align}\label{eq:DysonExactAttractive}
\K_{N,\be}(\mu_{\mathrm{unif}})
=\frac{\Gamma(1-\be/2)}
{\Gamma(1-\be/(2N))^N},
\qquad 0\le\be<2.
\end{align}
Since $\widehat{\g_1}(k)=(2|k|)^{-1}$ for $k\ne0$, it follows that
\begin{align}\label{eq:DysonDeterminantLimit}
\K_{N,\be}(\mu_{\mathrm{unif}})
\longrightarrow
 e^{-\gamma_{\mathrm E}\be/2}\Gamma(1-\be/2)
=\det_2(\mathrm{Id}-\be T_1)^{-1/2}.
\end{align}
Here $\gamma_{\mathrm E}$ denotes Euler's constant.  At the opposite end of the presently understood phase diagram, Lei and the third author \cite{LeiRosenzweigSharpLogHLS2026} prove, by means of an explicit wrapped Cauchy competitor, the strict dimension-eleven bound
\begin{align}\label{eq:d11BetaGmBound}
\be_{\mathrm{gm}}(11)
<\frac{8653}{8673}\,\be_{\mathrm{s}}(11)
=\frac{8653}{8673}\frac{64\pi^5}{945}
<\be_{\mathrm{s}}(11).
\end{align}
Thus, in dimension eleven, the uniform phase loses global minimality strictly before it loses local stability.

\begin{conj}[Periodic attractive phase diagram]
\label{conj:periodicAttractivePhaseDiagram}
\begin{enumerate}[(i)]
\item If $1\le\d\le10$, then $\be_{\mathrm{gm}}(\d)=2\d$ and
\begin{align}\label{eq:conjLowDimDeterminant}
\K_{N,\be}(\mu_{\mathrm{unif}})
\longrightarrow
\det_2(\mathrm{Id}-\be T_\d)^{-1/2},
\qquad 0\le\be<2\d.
\end{align}
\item In dimension $\d=11$, the determinant limit holds for
$0\le\be<\be_{\mathrm{gm}}(11)$, while
\begin{align}\label{eq:conjD11Pressure}
\frac1N\log\K_{N,\be}(\mu_{\mathrm{unif}})
\longrightarrow
-\inf_{\mu\ll\mu_{\mathrm{unif}}}\mathcal J_{\be,11}(\mu)>0,
\qquad
\be_{\mathrm{gm}}(11)<\be<22.
\end{align}
No assertion is made here at the coexistence point
$\be=\be_{\mathrm{gm}}(11)$.
\item If $\d\ge12$, then
$\be_{\mathrm{gm}}(\d)=\be_{\mathrm{s}}(\d)$ and the transition at
$\be_{\mathrm{s}}(\d)$ is continuous.  The leading quartic term of $\mathcal J_{\be_{\mathrm{s}}(\d),\d}^{\mathrm{red}}$ is strictly positive away from the origin in $E_1$.  Moreover,
\begin{align}\label{eq:conjHighDimDeterminant}
\K_{N,\be}(\mu_{\mathrm{unif}})
&\longrightarrow
\det_2(\mathrm{Id}-\be T_\d)^{-1/2},
&&0\le\be<\be_{\mathrm{s}}(\d),\\
\K_{N,\be_{\mathrm{s}}(\d)}(\mu_{\mathrm{unif}})
&\sim C_\d N^{\d/2},
&&C_\d\in(0,\infty),\label{eq:conjHighDimCritical}\\
\frac1N\log\K_{N,\be}(\mu_{\mathrm{unif}})
&\longrightarrow
-\inf_{\mu\ll\mu_{\mathrm{unif}}}\mathcal J_{\be,\d}(\mu)>0,
&&\be_{\mathrm{s}}(\d)<\be<2\d.\label{eq:conjHighDimPressure}
\end{align}
\end{enumerate}
\end{conj}

The exact formula \eqref{eq:DysonExactAttractive} proves part (i) of the conjecture in dimension one.  The forthcoming transfer theorem \cite{DGRPhaseTransitions2026}, combined with the sharp zero-defect log-HLS/Beckner--Onofri inequality of Lei and the third author \cite{LeiRosenzweigSharpLogHLS2026}, proves \eqref{eq:conjLowDimDeterminant} in dimensions $2\le\d\le6$.  We next record elementary lower bounds supporting the conjectured critical and supercritical behavior in dimensions $\d\ge11$.

\begin{prop}[Critical and supercritical lower bounds]
\label{prop:periodicAttractiveLowerBounds}
Assume $\d\ge11$. Then:
\begin{enumerate}[(i)]
\item At the stability threshold,
\begin{align}\label{eq:periodicCriticalLowerBound}
\liminf_{N\to\infty}
N^{-\d/2}\K_{N,\be_{\mathrm{s}}(\d)}(\mu_{\mathrm{unif}})
\ge
\left(\frac{\sqrt\pi}{e}\right)^\d.
\end{align}
\item For every $\be_{\mathrm{s}}(\d)<\be<2\d$,
\begin{align}\label{eq:periodicExponentialLowerBound}
\liminf_{N\to\infty}
\frac1N\log\K_{N,\be}(\mu_{\mathrm{unif}})>0.
\end{align}
\item For every $\be\ge2\d$ and every $N\ge2$,
\begin{align}\label{eq:periodicInfinitePartition}
\K_{N,\be}(\mu_{\mathrm{unif}})=+\infty.
\end{align}
\end{enumerate}
\end{prop}

\begin{proof}[Proof of \cref{prop:periodicAttractiveLowerBounds}]
\step{Part (i): Critical lower bound at the stability threshold.}
For $t>0$, let $\g_{\d,t}=e^{t\Delta}\g_\d$ be the heat regularization and let $\K_{N,\be}^{(t)}(\mu_{\mathrm{unif}})$ denote \eqref{eq:periodicAttractivePartition} with $\g_\d$ replaced by $\g_{\d,t}$.  The local comparison
\begin{align}\label{eq:heatLogComparison}
\g_{\d,t}(x)\le \g_\d(x)+C_\d,
\qquad 0<t\le1,
\end{align}
follows from $\g_\d(x)=-\log|x|+O(1)$ near the origin by splitting into $|x|\le2\sqrt t$ and its complement.  If a block of $k$ particles collapses on scale $R$, then the relative-coordinate volume element contributes $R^{\d(k-1)-1}\,\dd R$, while the attractive pair factor contributes at worst $R^{-\be k(k-1)/(2N)}$.  The corresponding radial integral is therefore bounded by
\begin{align}\label{eq:clusterRadialIntegrability}
\int_0^\delta
R^{(k-1)(\d-\be k/(2N))-1}\,\dd R,
\end{align}
which is finite for every $2\le k\le N$ when $\be<2\d$.  Decomposing a neighborhood of the collision set according to the finitely many set partitions of $\{1,\ldots,N\}$, and applying \eqref{eq:clusterRadialIntegrability} successively to each collapsing block, proves local integrability near every collision stratum.  Hence, since $\be_{\mathrm{s}}(\d)<2\d$, dominated convergence gives
\begin{align}\label{eq:regularizedPartitionConvergence}
\K_{N,\be_{\mathrm{s}}(\d)}^{(t)}(\mu_{\mathrm{unif}})
\longrightarrow
\K_{N,\be_{\mathrm{s}}(\d)}(\mu_{\mathrm{unif}})
\qquad(t\downarrow0)
\end{align}
for each fixed $N$.

Let $X_t$ be the centered Gaussian field with covariance $\g_{\d,t}$.  The Hubbard--Stratonovich transformation gives
\begin{align}\label{eq:periodicGaussianRepresentation}
\K_{N,\be}^{(t)}(\mu_{\mathrm{unif}})
=
\E\brak*{
\pa*{\int_{\T^\d}
\exp\Big\{\sqrt{\frac{\be}{N}}X_t(x)
-\frac{\be}{2N}\g_{\d,t}(0)\Big\}
\,\dd\mu_{\mathrm{unif}}(x)}^N}.
\end{align}
Decompose $X_t=X_{t,1}+X_{t,\perp}$ according to $E_1\oplus E_1^\perp$.  Conditional Jensen in $X_{t,\perp}$ removes the latter from \eqref{eq:periodicGaussianRepresentation}.  The eigenvalue on $E_1$ is
\begin{align}
\lambda_t
=\frac{e^{-4\pi^2t}}{\be_{\mathrm{s}}(\d)},
\end{align}
and one may write
\begin{align}
X_{t,1}(x)
=\sqrt{2\lambda_t}
\sum_{j=1}^{\d}
\big(A_j\cos(2\pi x_j)+B_j\sin(2\pi x_j)\big),
\end{align}
where the $A_j,B_j$ are independent standard Gaussians.  Passing to polar coordinates in each pair $(A_j,B_j)$ therefore yields
\begin{align}\label{eq:firstShellBesselRegularized}
\K_{N,\be}^{(t)}(\mu_{\mathrm{unif}})
\ge
 e^{-\be\d\lambda_t}
\left[
\int_0^\infty
2r e^{-r^2}
I_0\left(2\sqrt{\frac{\be\lambda_t}{N}}\,r\right)^N
\,\dd r
\right]^\d,
\end{align}
where
\begin{align}
I_0(z)
\coloneqq
\int_0^1e^{z\cos(2\pi x)}\,\dd x
\end{align}
is the modified Bessel function.  Taking $\be=\be_{\mathrm{s}}(\d)$ and then $t\downarrow0$ in \eqref{eq:firstShellBesselRegularized}, using \eqref{eq:regularizedPartitionConvergence}, gives
\begin{align}\label{eq:firstShellBessel}
\K_{N,\be_{\mathrm{s}}(\d)}(\mu_{\mathrm{unif}})
\ge e^{-\d}J_N^\d,
\qquad
J_N\coloneqq
\int_0^\infty
2r e^{-r^2}
I_0\left(\frac{2r}{\sqrt N}\right)^N\,\dd r.
\end{align}
The power series of $I_0$ gives
\begin{align}\label{eq:BesselLogExpansion}
\log I_0(z)
=\frac{z^2}{4}-\frac{z^4}{64}+O(z^6)
\qquad(z\to0).
\end{align}
After the change of variables $r=N^{1/4}u$,
\begin{align}
N^{-1/2}J_N
=
\int_0^\infty
2u\exp\Bigg\{-N^{1/2}u^2
+N\log I_0\left(\frac{2u}{N^{1/4}}\right)\Bigg\}\,\dd u.
\end{align}
The integrand converges pointwise to $2u e^{-u^4/4}$.  Fatou's lemma and \eqref{eq:firstShellBessel} imply
\begin{align}
\liminf_{N\to\infty}N^{-1/2}J_N
\ge
\int_0^\infty2u e^{-u^4/4}\,\dd u
=\sqrt\pi,
\end{align}
which proves \eqref{eq:periodicCriticalLowerBound}.

\step{Part (ii): Exponential lower bound above the stability threshold.}
Let $\mu\ll\mu_{\mathrm{unif}}$ be a probability measure.  The Donsker--Varadhan lemma, applied with the product trial $\mu^{\otimes N}$, gives
\begin{align}\label{eq:periodicProductTrial}
\log\K_{N,\be}(\mu_{\mathrm{unif}})
\ge
\frac{\be(N-1)}2
\left\langle\mu-1,T_\d(\mu-1)\right\rangle_{L^2(\mu_{\mathrm{unif}})}
-N\Hr(\mu\vert\mu_{\mathrm{unif}}).
\end{align}
For $|\vep|<1$, take
$\mu_\vep(x)=1+\vep\cos(2\pi x_1)$.  By \eqref{eq:periodicLogFourier} and \eqref{eq:betaSDef},
\begin{align}\label{eq:firstShellTrialEnergyEntropy}
\left\langle\mu_\vep-1,T_\d(\mu_\vep-1)\right\rangle_{L^2(\mu_{\mathrm{unif}})}
=\frac{\vep^2}{2\be_{\mathrm{s}}(\d)},
\qquad
\Hr(\mu_\vep\vert\mu_{\mathrm{unif}})
=\frac{\vep^2}{4}+O(\vep^4).
\end{align}
Consequently, for every $\be>\be_{\mathrm{s}}(\d)$ and every sufficiently small nonzero $\vep$,
\begin{align}
\frac\be2
\left\langle\mu_\vep-1,T_\d(\mu_\vep-1)\right\rangle_{L^2(\mu_{\mathrm{unif}})}
-\Hr(\mu_\vep\vert\mu_{\mathrm{unif}})>0.
\end{align}
Divide \eqref{eq:periodicProductTrial} by $N$ and let $N\to\infty$ to obtain \eqref{eq:periodicExponentialLowerBound}.

\step{Part (iii): Divergence at and above the collision threshold.}
Fix a small coordinate cube $Q\subset\T^\d$ and restrict the integral in \eqref{eq:periodicAttractivePartition} to configurations with $x_1\in Q$ and $x_i=x_1+y_i$, $2\le i\le N$, where the relative vector $y=(y_2,\ldots,y_N)\in\R^{\d(N-1)}$ has sufficiently small norm $R=|y|$.  Since $\g_\d(z)=-\log|z|+O(1)$ near the origin and every pair distance is at most $2R$ on this set, the integrand is bounded from below by $C_{N,\be,\d}R^{-\be(N-1)/2}$.  The corresponding radial integral is bounded from below by
\begin{align}\label{eq:totalCollisionRadialIntegral}
C_{N,\be,\d}
\int_0^\delta
R^{(N-1)(\d-\be/2)-1}\,\dd R.
\end{align}
The integral diverges logarithmically when $\be=2\d$ and by a power when $\be>2\d$, proving \eqref{eq:periodicInfinitePartition}.
\end{proof}

\begin{remark}[The dimension-eleven critical point]
\label{rem:d11CriticalExponential}
The lower bound \eqref{eq:periodicCriticalLowerBound} holds in dimension eleven but is not rate-sharp there.  Indeed, the strict inequality \eqref{eq:d11BetaGmBound} and the product trial \eqref{eq:periodicProductTrial} imply
\begin{align}
\liminf_{N\to\infty}
\frac1N\log\K_{N,\be_{\mathrm{s}}(11)}(\mu_{\mathrm{unif}})>0.
\end{align}
Thus the conjectured sharp role of the polynomial factor $N^{\d/2}$ begins in dimensions $\d\ge12$.
\end{remark}
\endgroup

\section{Entropic commutators and diffusive mean-field estimates}
\label{sec:entropic-commutator-estimate-proof}

We now turn to the proof of \cref{thm:entropicCommutator}, which has three steps.  The Donsker--Varadhan reduction in \cref{subsec:donsker-varadhan-reduction} reduces the desired entropic inequality to an exponential moment under $\Q_{N,\be}(\mu)$, the local numerator estimate in \cref{subsec:local-numerator-estimate} proves this moment for small inverse temperature, and the H\"older bootstrap in \cref{subsec:removal-smallness-assumption} removes the smallness restriction.  After proving the theorem, \cref{subsec:commutator-temperature-interpolation} proves \cref{cor:entropicCommutatorTemperatureInterpolation}, which combines the fixed-temperature entropic estimate with the deterministic pointwise commutator bound and is adapted to inverse temperatures growing with $N$.  The section concludes with the proof of the diffusive mean-field application, \cref{thm:diffusiveMeanFieldConvergence}, in \cref{sec:diffusive-mean-field-convergence}.

\subsection{Donsker--Varadhan reduction}
\label{subsec:donsker-varadhan-reduction}

The case $\|\nab v\|_{L^\infty}=0$ has already been disposed of in the statement, so we assume $\|\nab v\|_{L^\infty}>0$.  By the Donsker--Varadhan lemma, for any $\vep>0$ and $\sigma\in\{\pm 1\}$,
\begin{align}\label{eq:expDV}
\sigma\E_{\XN\sim f_N}\Big[N\mathsf{A}_n(v,\XN,\mu)\Big] \le \frac{1}{\vep}\Bigg(\Hr(f_N\vert \Q_{N,\be}(\mu)) + \log \E_{\XN\sim \Q_{N,\be}(\mu)}\Big[e^{\sigma\vep N \mathsf{A}_n(v,\XN,\mu)}\Big] \Bigg).
\end{align}
Using the entropy rewriting \eqref{eq:intro-entropy-rewrite}, together with \eqref{eq:CpfDef}, it remains to prove \eqref{eq:QNexpMomentAn} for the choice $\vep=\vep_{\be,n}$ in \eqref{eq:epsnStatement}.

Fix $\sigma\in\{\pm 1\}$.  Unpacking the definition of $\Q_{N,\be}(\mu)$, we have
\begin{align}\label{eq:EQNbeAn0}
\E_{\XN\sim \Q_{N,\be}(\mu)}\Big[e^{\sigma\vep_{\be,n} N \mathsf{A}_n(v,\XN,\mu)}\Big]
= \frac{1}{\K_{N,\be}(\mu)}\E_{\XN\sim \mu^{\otimes N}}\Big[ e^{\sigma\vep_{\be,n} N  \mathsf{A}_n(v,\XN,\mu) - \be N\Fr_N(\XN,\mu)}\Big].
\end{align}
In view of \eqref{eq:KlowerJensen}, it suffices to prove the product-space numerator estimate
\begin{align}\label{eq:numeratorGoal}
\sup_{N\ge1}
\E_{\XN\sim \mu^{\otimes N}}\Big[ e^{\sigma\vep_{\be,n} N  \mathsf{A}_n(v,\XN,\mu) - \be N\Fr_N(\XN,\mu)}\Big]<\infty
\end{align}
uniformly in $N$ and in $\sigma\in\{\pm 1\}$.

\subsection{Commutator truncation estimates}
\label{subsec:commutator-truncation-estimates}

The next lemma converts the potential truncations of \cref{subsec:positive-definite-truncations-riesz-kernel} into the derivative and short-range estimates needed for the commutator argument.

\begin{lemma}[Commutator truncation estimates]\label{lem:commutatorTruncationEstimates}
Let $0<\s<\d$, let $n\ge1$, and let $v:\R^\d\to\R^\d$ be globally Lipschitz.  For a differentiable kernel $\phi$, let
\begin{align}\label{eq:knetadef}
\k_{n,\phi}(x,y) \coloneqq \nab^{\otimes n}\phi(x-y):(v(x)-v(y))^{\otimes n},
\end{align}
which is symmetric whenever $\phi$ is even.  There are constants $C_n$ and $c_n\ge2$, depending only on $\d,\s,n$, such that the truncations in \eqref{eq:HCRStruncDef} satisfy
\begin{align}\label{eq:knPointwise}
|\k_{n,\g_\eta}(x,y)|&\le C_n\|\nab v\|_{L^\infty}^n\g_\eta(x-y),\nn\\
|\k_{n,\f_\eta}(x,y)|&\le C_n\|\nab v\|_{L^\infty}^n\f_{c_n\eta}(x-y).
\end{align}
If $\mu\in L^\infty$ and $\sigma\in\{\pm1\}$, then the short-range deterministic estimate holds:
\begin{multline}\label{eq:detCommShortRange}
\frac{\sigma}{N^2}\sum_{1\le i\ne j\le N}\tl{\k}_{n,\f_\eta}(x_i,x_j) \le \frac{C_{\mathrm{sr},n}\|\nab v\|_{L^\infty}^n}{N^2}\sum_{1\le i\ne j\le N}\f_{c_n\eta}(x_i-x_j)\\
{}+ C_n\|\nab v\|_{L^\infty}^n\|\mu\|_{L^\infty}\eta^{\d-\s},
\end{multline}
where $C_{\mathrm{sr},n}$ depends only on $\d,\s,n$.
\end{lemma}

\subsection{Local numerator estimate}
\label{subsec:local-numerator-estimate}

This lemma is the main analytic estimate in the commutator proof.  It repeats the partition-function layer-cake mechanism with the commutator observable inserted, choosing the coefficient so that the short-range commutator contribution is absorbed by the modulated energy.

\begin{lemma}[Local numerator estimate]\label{lem:localNumeratorComm}
There exists $\be_{\mathrm{loc}}=\be_{\mathrm{loc}}(\mu,v,n)>0$ such that for every $0<\gamma\le\be_{\mathrm{loc}}$ and every $\sigma\in\{\pm1\}$,
\begin{align}\label{eq:localNumeratorComm}
\sup_{N\ge1}
\E_{\mu^{\otimes N}}
\Bigg[\exp\pa*{
\sigma\frac{\gamma}{4C_{\mathrm{sr},n}\|\nab v\|_{L^\infty}^n}
N\mathsf A_n(v,\XN,\mu)
-\gamma N\Fr_N(\XN,\mu)}\Bigg]
\le e+L.
\end{align}
\end{lemma}

\begin{proof}
We use the truncation and commutator-kernel notation of \cref{lem:commutatorTruncationEstimates}, and set
\begin{align}\label{eq:epsGammaLocal}
\vep_{\gamma,n}\coloneqq \frac{\gamma}{4C_{\mathrm{sr},n}\|\nab v\|_{L^\infty}^n}.
\end{align}
The one-body terms in the expansion of $\mathsf{A}_n(v,\XN,\mu)$ are bounded by a constant depending on $\mu,v,n$.  More precisely,
\begin{align}\label{eq:AnCenteredExpansion}
\mathsf{A}_n(v,\XN,\mu)=\frac{1}{N^2}\sum_{1\le i\ne j\le N}\tl\k_{n,\g}(x_i,x_j)+R_{n,N}(\XN),\qquad |NR_{n,N}(\XN)|\le C_{\mu,v,n}.
\end{align}
We treat first the positive Riesz case $0<\s<\d/2$; the logarithmic case is Step 4.

\step{Step 1: Reduction to the centered statistics.}  For every $\sigma\in\{\pm1\}$ and $0<\gamma\le\be_{\mathrm{loc}}$, the bound \eqref{eq:localNumeratorComm} follows from
\begin{align}\label{eq:commReducedGoal}
\sup_{N\ge1}\ \E_{\mu^{\otimes N}}\brak[\big]{e^{\Xi_N^{\sigma}}}\le e+L,
\qquad
\Xi_N^{\sigma}\coloneqq \vep_{\gamma,n}C_{\mu,v,n}+\frac{\sigma\vep_{\gamma,n}}{N}\sum_{1\le i\ne j\le N}\tl\k_{n,\g}(x_i,x_j)-\gamma N\Fr_N(\XN,\mu),
\end{align}
and, by \eqref{eq:FNdef}, $\Xi_N^{\sigma}$ is of the master form \eqref{eq:masterXiDef}, with $g=\gamma\g-2\sigma\vep_{\gamma,n}\k_{n,\g}$ and $w=\gamma(h_\mu-I_\mu)+\vep_{\gamma,n}C_{\mu,v,n}$.

\par\smallskip\noindent  By \eqref{eq:AnCenteredExpansion}, the exponent in \eqref{eq:localNumeratorComm} is $\sigma\vep_{\gamma,n}N\mathsf A_n-\gamma N\Fr_N\le\Xi_N^{\sigma}$ pointwise, so \eqref{eq:commReducedGoal} implies \eqref{eq:localNumeratorComm}.  The second assertion is the splitting identity \eqref{eq:FNdef}, the constant $\vep_{\gamma,n}C_{\mu,v,n}$ being part of the one-body function $w$.

\step{Step 2: Choice of the data.}  With $c_n$ and $C_n$ the constants in \eqref{eq:knPointwise} and \eqref{eq:detCommShortRange}, and $C_0$ the near-field constant in \eqref{eq:HCRSconvBound}, set
\begin{align}\label{eq:CabsDef}
C_{\mathrm{abs},n}\coloneqq \frac{C_n\|\mu\|_{L^\infty}}{4C_{\mathrm{sr},n}}+c_n^{\d-\s}\,C_0.
\end{align}
Define $U_0\coloneqq1$ and, for $N\ge1$, $u\ge1$, and $0<\eta\le\eta_0$, the kernels, adaptive scales, and caps
\begin{align}
\tl h_{n,\eta}^{\sigma}&\coloneqq \sigma\vep_{\gamma,n}\tl\k_{n,\g_\eta}-\frac{\gamma}{2}\tl\g_{c_n\eta},
\qquad
\eta(u,N)\coloneqq \min\Bigg\{\eta_0,\Big(\frac{u}{4C_{\mathrm{abs},n}\gamma N}\Big)^{\frac{1}{\d-\s}}\Bigg\},\label{eq:hetandef}\\
q_{N,u}&\coloneqq -4\,\tl h^{\sigma}_{n,\eta(u,N)},
\qquad
\mathsf m_{N,u}\coloneqq \frac N2\,\|\tl q_{N,u}\|_{L^\infty}.\label{eq:adaptiveEtaComm}
\end{align}
\begingroup

Assuming, as we may, $C_{\mathrm{sr},n}\ge C_n$, define the uncentered kernel
\begin{align}
 r_{N,u}^{\sigma}
 \coloneqq
 -4\sigma\vep_{\gamma,n}\k_{n,\g_{\eta(u,N)}}
 +2\gamma\,\g_{c_n\eta(u,N)}.
\end{align}
Then $q_{N,u}$ is the $\mu$-centering of $r_{N,u}^{\sigma}$ and, by \eqref{eq:knPointwise}, \eqref{eq:epsGammaLocal}, and $0\le\g_{c_n\eta}\le\g_\eta$,
\begin{align}\label{eq:commMajorant}
 |r_{N,u}^{\sigma}|\le 3\gamma\,\g_{\eta(u,N)}.
\end{align}
\endgroup
Since $\k_{n,\g_\eta}$ is not Fourier sign-definite, \cref{lem:PDdeterministicDiagonalBound} does not apply, and the full $L^\infty$ norm has to be used in the cap of \eqref{eq:adaptiveEtaComm}.

Then there are constants $M_{\mu,v,n}<\infty$ and $\be_{\mathrm{loc}}=\be_{\mathrm{loc}}(\mu,v,n)>0$ such that, for every $0<\gamma\le\be_{\mathrm{loc}}$ and every $\sigma\in\{\pm1\}$, these data satisfy the hypotheses (M1)--(M3) of \cref{prop:masterLayerCake} for $\Xi_N^{\sigma}$, together with
\begin{align}\label{eq:commSlopeBound}
\sup_{u\ge1}\ \min\brac[\big]{u,\ \mathsf m_{N,u}}\pa[\big]{1+\|\tl q_{N,u}\|_{L^\infty}}\le M_{\mu,v,n}\,\gamma\,N
\qquad\text{and}\qquad
2L^2M_{\mu,v,n}\gamma\le1.
\end{align}
In particular, $2L^2\,\mathsf S_N\le 2L^2M_{\mu,v,n}\gamma N\le N$ for every $N\ge1$, so the growth condition \eqref{eq:masterLayerCondition} holds with $N_0=1$.

\begingroup
\par\smallskip\noindent\textbf{(M2).}  Hypothesis (M2) holds by the definition of $\mathsf m_{N,u}$.  Indeed, for every configuration,
\begin{align}
-\frac{1}{2N}\sum_{1\le i\ne j\le N}\tl q_{N,u}(x_i,x_j)
\le \frac N2\|\tl q_{N,u}\|_{L^\infty}.
\end{align}
\endgroup

\par\smallskip\noindent\textbf{(M1).}  Using the decomposition \eqref{eq:FNdef} with $\g=\g_{c_n\eta}+\f_{c_n\eta}$, and bounding the one-body term by $\gamma\|h_\mu\|_{L^\infty}$ and $-\gamma I_\mu\le0$,
\begin{multline}\label{eq:ZBeforeShortAbsorb}
\Xi_{N}^{\sigma}
\le C_{\mu,v,n}'\gamma
+\frac{1}{N}\sum_{1\le i\ne j\le N}\tl h_{n,\eta}^{\sigma}(x_i,x_j)\\
+\frac{\sigma\vep_{\gamma,n}}{N}\sum_{1\le i\ne j\le N}\tl\k_{n,\f_\eta}(x_i,x_j)-\frac{\gamma}{2N}\sum_{1\le i\ne j\le N}\tl\f_{c_n\eta}(x_i,x_j),
\end{multline}
where $C_{\mu,v,n}'\coloneqq C_{\mu,v,n}/(4C_{\mathrm{sr},n}\|\nab v\|_{L^\infty}^n)+\|h_\mu\|_{L^\infty}$.  By \eqref{eq:detCommShortRange}, \eqref{eq:HCRSfetaDecay}, \eqref{eq:HCRSconvBound}, and the choice \eqref{eq:epsGammaLocal},
\begin{align}\label{eq:shortAbsorb}
\frac{\sigma\vep_{\gamma,n}}{N}\sum_{1\le i\ne j\le N}\tl\k_{n,\f_\eta}(x_i,x_j)-\frac{\gamma}{2N}\sum_{1\le i\ne j\le N}\tl\f_{c_n\eta}(x_i,x_j)\le C_{\mathrm{abs},n}\gamma N\eta^{\d-\s};
\end{align}
indeed, multiplying \eqref{eq:detCommShortRange} by $\vep_{\gamma,n}N$ contributes $\frac{\gamma}{4N}\sum\f_{c_n\eta}+\frac{C_n\|\mu\|_{L^\infty}}{4C_{\mathrm{sr},n}}\gamma N\eta^{\d-\s}$, the total coefficient of $\sum\f_{c_n\eta}$ is then nonpositive, and the centered one-body terms of $-\frac{\gamma}{2N}\sum\tl\f_{c_n\eta}$ are bounded by $\gamma N\|\f_{c_n\eta}\ast\mu\|_{L^\infty}\le c_n^{\d-\s}C_0\,\gamma N\eta^{\d-\s}$, by \eqref{eq:HCRSconvBound}, matching \eqref{eq:CabsDef}.  Combining \eqref{eq:ZBeforeShortAbsorb} and \eqref{eq:shortAbsorb} gives, for every $0<\eta\le\eta_0$,
\begin{align}\label{eq:ZNetabound}
\Xi_{N}^{\sigma}\le C_{\mu,v,n}'\gamma+\frac{1}{N}\sum_{1\le i\ne j\le N}\tl h_{n,\eta}^{\sigma}(x_i,x_j)+C_{\mathrm{abs},n}\gamma N\eta^{\d-\s}.
\end{align}
Assume $4C_{\mu,v,n}'\gamma\le1$, and let $u\ge1$.  Taking $\eta=\eta(u,N)$ in \eqref{eq:ZNetabound}, the last term is at most $u/4$ by \eqref{eq:hetandef}, while $C_{\mu,v,n}'\gamma\le1/4\le u/4$; hence
\begin{align}\label{eq:eventReductionComm}
\Xi_{N}^{\sigma}\ge u
\quad\Longrightarrow\quad
\frac{1}{N}\sum_{1\le i\ne j\le N}\tl h_{n,\eta(u,N)}^{\sigma}(x_i,x_j)\ge \frac{u}{2}
\quad\Longleftrightarrow\quad
-\frac{1}{2N}\sum_{1\le i\ne j\le N}\tl q_{N,u}(x_i,x_j)\ge u,
\end{align}
which is (M1).

\par\smallskip\noindent\textbf{(M3) and the $L^\infty$ norm.}  Since the centering costs at most a factor $4$ in each of the norms of \cref{rem:Ustattail}, the majorant \eqref{eq:commMajorant} bounds every norm of $\tl q_{N,u}$ by $12\gamma$ times the corresponding norm of $\g_{\eta(u,N)}$.  These are bounded uniformly in the scale: $0\le\g_\eta\le\g$, the $L^2(\mu^{\otimes2})$ and $L^\infty L^2(\mu)$ norms of $\g$ against $\mu$ are finite exactly because $2\s<\d$, as in the partition-function case, and the operator norm follows from the Hilbert--Schmidt bound; only the $L^\infty$ norm sees the scale, through $\|\g_{\eta}\|_{L^\infty}\le C_2\eta^{-\s}$ from \eqref{eq:HCRSdiagPD}.  Altogether, for $0<\gamma\le1$, there is $C_{\mathsf K,\mu,v,n}<\infty$, independent of $u$, $N$, $\gamma$, and $\sigma$, with
\begin{align}\label{eq:commKernelNorms}
\max\brac[\big]{\|\tl q_{N,u}\|_{L^2(\mu^{\otimes2})}^2,\ \|\tl q_{N,u}\|_{L^\infty L^2(\mu)}^2,\ \|\Tc_{\tl q_{N,u}}\|_{\mathrm{op}}}&\le \gamma\,C_{\mathsf K,\mu,v,n},\nn\\
\|\tl q_{N,u}\|_{L^\infty}&\le \gamma\,C_{\mathsf K,\mu,v,n}\,\eta(u,N)^{-\s},
\end{align}
in the norms of \cref{rem:Ustattail}; hence (M3) holds as soon as $\gamma C_{\mathsf K,\mu,v,n}\le1/L$.

\par\smallskip\noindent\textbf{Growth bound \eqref{eq:commSlopeBound}.}  Assume $0<\gamma\le1$ and let $u\ge1$.  If $\eta(u,N)=\eta_0$, then, by \eqref{eq:commKernelNorms} and the definition of $\mathsf m_{N,u}$,
\begin{align}
\min\brac[\big]{u,\mathsf m_{N,u}}\pa[\big]{1+\|\tl q_{N,u}\|_{L^\infty}}
\le \frac{\gamma}{2}\,C_{\mathsf K,\mu,v,n}\eta_0^{-\s}\pa[\big]{1+C_{\mathsf K,\mu,v,n}\eta_0^{-\s}}N.
\end{align}
If instead the second entry of the minimum in \eqref{eq:hetandef} is selected, then necessarily $u\le 4C_{\mathrm{abs},n}\eta_0^{\d-\s}\,\gamma N$, and
\begin{align}\label{eq:commGrowthAdaptive}
\min\brac[\big]{u,\mathsf m_{N,u}}\pa[\big]{1+\|\tl q_{N,u}\|_{L^\infty}}
&\le u+\gamma C_{\mathsf K,\mu,v,n}\pa*{\frac{4C_{\mathrm{abs},n}\gamma N}{u}}^{\frac{\s}{\d-\s}}u\nn\\
&= u+C_{\mathsf K,\mu,v,n}\pa[\big]{4C_{\mathrm{abs},n}}^{\frac{\s}{\d-\s}}\,\gamma^{\frac{\d}{\d-\s}}\,u^{\frac{\d-2\s}{\d-\s}}\,N^{\frac{\s}{\d-\s}},
\end{align}
which is at most $M_{\mu,v,n}\gamma N$ by $u\le C\gamma N$, $\frac{\d-2\s}{\d-\s}+\frac{\s}{\d-\s}=1$, $\frac{\d}{\d-\s}+\frac{\d-2\s}{\d-\s}=2$, and $\gamma\le1$.  Enlarging $M_{\mu,v,n}$ to cover the case $\eta(u,N)=\eta_0$, this proves the first bound in \eqref{eq:commSlopeBound}, and it remains to set
\begin{align}\label{eq:admissibleBetaLocPositive}
\be_{\mathrm{loc}}\coloneqq\min\brac[\Big]{1,\ \frac{1}{4C_{\mu,v,n}'},\ \frac{1}{L\,C_{\mathsf K,\mu,v,n}},\ \frac{1}{2L^2M_{\mu,v,n}}},
\end{align}
so that, for $0<\gamma\le\be_{\mathrm{loc}}$, the smallness assumptions used in (M1) and (M3) hold and $2L^2M_{\mu,v,n}\gamma\le1$.  Step 2 is proved.

\step{Step 3: Conclusion, positive Riesz case.}
By Steps 1 and 2, \cref{prop:masterLayerCake} applies to $\Xi_N^{\sigma}$ with $U_0=1$ and $N_0=1$, so \eqref{eq:masterConclusion} gives
\begin{align}
\sup_{N\ge1}\ \E_{\mu^{\otimes N}}\brak[\big]{e^{\Xi_N^{\sigma}}}\le e^{U_0}+L=e+L,
\end{align}
which is \eqref{eq:commReducedGoal}.  This proves \eqref{eq:localNumeratorComm} in the positive Riesz case.

\step{Step 4: The logarithmic case $\s=0$.}  After further decreasing $\be_{\mathrm{loc}}$, the bound \eqref{eq:localNumeratorComm} holds in the logarithmic case as well.

\par\smallskip\noindent  We shall use the following elementary consequence of the logarithmic truncation.  For each fixed $n\ge1$,
\begin{align}\label{eq:logknBounds}
|\k_{n,\g_\eta}(x,y)|\le C_n\|\nab v\|_{L^\infty}^n,
\qquad
|\k_{n,\f_\eta}(x,y)|\le C_n\|\nab v\|_{L^\infty}^n\f_{c_n\eta}(x-y);
\end{align}
the first inequality follows from the homogeneity of $\nab^{\otimes n}\log|x|$ together with the small-scale truncation, and the second is \eqref{eq:logfetaDerivative}.

By \eqref{eq:logEnergyAssumption}, $\Fr_{N,T}\to\Fr_N$ $\mu^{\otimes N}$-a.e.\ as $T\to\infty$, so, by Step 1 and Fatou's lemma, it suffices to prove the bound \eqref{eq:commReducedGoal}, uniformly in $T\ge2$, for
\begin{align}\label{eq:logCommStart}
\Xi_{N,T}^{\sigma}\coloneqq \vep_{\gamma,n}C_{\mu,v,n}+\frac{\sigma\vep_{\gamma,n}}{N}\sum_{1\le i\ne j\le N}\tl\k_{n,\g}(x_i,x_j)-\gamma N\Fr_{N,T}(\XN,\mu),
\end{align}
which, by \eqref{eq:logFiniteTExponent}, is of the master form \eqref{eq:masterXiDef} with $g=\gamma\g_T-2\sigma\vep_{\gamma,n}\k_{n,\g}$ and $w=\gamma(h_{\mu,T}-I_{\mu,T})+\vep_{\gamma,n}C_{\mu,v,n}$.

Fix $K=1$, set $C_I(\mu)\coloneqq\sup_{T\ge2}|I_{\mu,T}|<\infty$, and enlarge the logarithmic constant $C_{\mu,v,n}'$ to include $C_I(\mu)$.  Set $C_{\mathrm{abs},n}'\coloneqq \frac{C_n\|\mu\|_{L^\infty}}{4C_{\mathrm{sr},n}}+c_n^{\d}\,C_0'$ with $C_0'$ the near-field constant in \eqref{eq:logfetaConv}, and repeat Steps 2 and 3 at each fixed $T\ge2$ with the data
\begin{align}\label{eq:logCommKernelDef}
\tl h_{n,\eta,T}^{\sigma}\coloneqq \sigma\vep_{\gamma,n}\tl\k_{n,\g_\eta}-\frac\gamma2\tl\g_{c_n\eta,T,>K},
\qquad
\eta(u,N)\coloneqq \min\Bigg\{\eta_0,\pa*{\frac{u}{4C_{\mathrm{abs},n}'\gamma N}}^{1/\d}\Bigg\},
\end{align}
$q_{N,u}\coloneqq-4\,\tl h^{\sigma}_{n,\eta(u,N),T}$, and $\mathsf m_{N,u}\coloneqq\frac N2\|\tl q_{N,u}\|_{L^\infty}$, so that (M2) again holds by definition.  The remaining verifications change as follows.

\par\smallskip\noindent\textbf{Verification of (M1).}  Decompose $\g=\g_\eta+\f_\eta$ in the commutator term and $\g_T=\g_{c_n\eta,T,\le K}+\g_{c_n\eta,T,>K}+\f_{c_n\eta}$ in the energy term; combining the one-body term in $-\gamma N\Fr_{N,T}$ with the low-frequency energy by \cref{lem:logLargeScaleCancellation}(ii), and using \eqref{eq:logknBounds}, \eqref{eq:logfetaConv}, and the choice \eqref{eq:epsGammaLocal}, gives the analogue of \eqref{eq:ZNetabound},
\begin{align}\label{eq:logZNetabound}
\Xi_{N,T}^{\sigma}
\le C_{\mu,v,n}'\gamma+\frac1N\sum_{1\le i\ne j\le N}\tl h_{n,\eta,T}^{\sigma}(x_i,x_j)
+C_{\mathrm{abs},n}'\gamma N\eta^\d,
\end{align}
with $\eta^\d$ in place of $\eta^{\d-\s}$, uniformly in $T\ge2$; (M1) then follows exactly as in \eqref{eq:eventReductionComm}.

\par\smallskip\noindent\textbf{Verification of (M3) and the $L^\infty$ norm.}   \eqref{eq:logknBounds} and \eqref{eq:logHighFreqNorms} give, uniformly in $T\ge2$, the analogue of \eqref{eq:commKernelNorms} with $\eta(u,N)^{-\s}$ replaced by $1+\log(1/\eta(u,N))$, for a constant $C_{\mathsf K,\mu,v,n}'$; hence (M3) holds after further decreasing $\be_{\mathrm{loc}}$.

\par\smallskip\noindent\textbf{Verification of the growth bound \eqref{eq:commSlopeBound}.}  The case $\eta(u,N)=\eta_0$ is as before.  If the second entry of the minimum in \eqref{eq:logCommKernelDef} is selected, then $r\coloneqq u/(4C_{\mathrm{abs},n}'\gamma N)\le\eta_0^{\d}\le1$ and $1+\log(1/\eta(u,N))=1+\frac1\d\log(1/r)$, so
\begin{align}\label{eq:logCommGrowth}
\min\brac[\big]{u,\mathsf m_{N,u}}\pa[\big]{1+\|\tl q_{N,u}\|_{L^\infty}}
\le u+\gamma C_{\mathsf K,\mu,v,n}'\Big(1+\tfrac1\d\log(1/r)\Big)\,u
\le M_{\mu,v,n}\gamma N,
\end{align}
after enlarging $M_{\mu,v,n}$, because $u=4C_{\mathrm{abs},n}'\gamma N\,r$ and $r\mapsto r\pa[\big]{1+\log(1/r)}$ is bounded on $(0,1]$.

Hence, for $0<\gamma\le\be_{\mathrm{loc}}$, \cref{prop:masterLayerCake}, applied with $U_0=1$ and $N_0=1$, gives, uniformly in $T\ge2$, $\sup_{N\ge1}\E_{\mu^{\otimes N}}[e^{\Xi_{N,T}^{\sigma}}]\le e+L$, and Fatou's lemma concludes.
\end{proof}

\subsection{Removal of the smallness assumption}
\label{subsec:removal-smallness-assumption}

The preceding lemma proves the required numerator estimate only for inverse temperatures at most $\be_{\mathrm{loc}}$.  We now complete the proof of \cref{thm:entropicCommutator} by retaining one factor at this local inverse temperature and controlling the remaining Gibbs factor by H\"older's inequality and the partition-function theorem.

\begin{proof}[Proof of \cref{thm:entropicCommutator}]
In the local range $0<\be\le\be_{\mathrm{loc}}$, applying \cref{lem:localNumeratorComm} with $\gamma=\be$ gives \eqref{eq:numeratorGoal}.  Dividing by $\K_{N,\be}(\mu)$ in \eqref{eq:EQNbeAn0} and using \eqref{eq:KlowerJensen} gives \eqref{eq:QNexpMomentAn}; in the positive Riesz case, one may take $C_{\mathrm{em}}(\be,\mu,v,n)=e+L$, while in the logarithmic case, one may take $C_{\mathrm{em}}(\be,\mu,v,n)=e^{-\be I_\mu}(e+L)$.

Assume next that $\be>\be_{\mathrm{loc}}$.  Then
\begin{align}\label{eq:epsLargeBetaProof}
\vep_{\be,n}
=\frac{\be_{\mathrm{loc}}}{8C_{\mathrm{sr},n}\|\nab v\|_{L^\infty}^n}.
\end{align}
For the positive Riesz energy, and for the finite-$T$ logarithmic energy before passage to the limit, we have the exact decomposition
\begin{align}\label{eq:holderDecomposition}
&\sigma\vep_{\be,n}N\mathsf A_n(v,\XN,\mu)-\be N\Fr_N(\XN,\mu)\nn\\
&\qquad=
\frac12\Bigg[
\sigma\frac{\be_{\mathrm{loc}}}{4C_{\mathrm{sr},n}\|\nab v\|_{L^\infty}^n}
N\mathsf A_n(v,\XN,\mu)
-\be_{\mathrm{loc}}N\Fr_N(\XN,\mu)
\Bigg]
-\pa*{\be-\frac{\be_{\mathrm{loc}}}{2}}N\Fr_N(\XN,\mu).
\end{align}
H\"older's inequality with exponents $2,2$ gives
\begin{align}\label{eq:holderBootstrapComm}
&\E_{\mu^{\otimes N}}
\Big[e^{\sigma\vep_{\be,n}N\mathsf A_n(v,\XN,\mu)
-\be N\Fr_N(\XN,\mu)}\Big]\nn\\
&\qquad\le
\Bigg(
\E_{\mu^{\otimes N}}
\Big[e^{
\sigma\frac{\be_{\mathrm{loc}}}
{4C_{\mathrm{sr},n}\|\nab v\|_{L^\infty}^n}
N\mathsf A_n(v,\XN,\mu)-\be_{\mathrm{loc}}N\Fr_N(\XN,\mu)}\Big]
\Bigg)^{1/2}
\K_{N,2\be-\be_{\mathrm{loc}}}(\mu)^{1/2}\nn\\
&\qquad\le
\pa*{e+L}^{1/2}
C_{\mathrm{ub}}(2\be-\be_{\mathrm{loc}},\mu)^{1/2}.
\end{align}
In the logarithmic case, the same estimate is first applied with $\Fr_{N,T}$ in place of $\Fr_N$, using the finite-$T$ local numerator estimate and the uniform-in-$T$ partition-function bound \eqref{eq:logFiniteTPartitionGoal} from \cref{subsec:uniform-bound-logarithmic}.  Fatou's lemma then gives \eqref{eq:holderBootstrapComm} for $\Fr_N$.

Dividing \eqref{eq:holderBootstrapComm} by $\K_{N,\be}(\mu)$ and using the Jensen lower bound \eqref{eq:KlowerJensen} proves \eqref{eq:QNexpMomentAn} for $\be>\be_{\mathrm{loc}}$ with the constant displayed in \eqref{eq:CemLargeBetaDef}.  This proves the exponential-moment assertion for all $\be>0$.

Finally, inserting \eqref{eq:QNexpMomentAn} and \eqref{eq:CpfDef} into \eqref{eq:expDV} and using the entropy rewriting \eqref{eq:intro-entropy-rewrite} with $\vep=\vep_{\be,n}$, gives, for every $\sigma\in\{\pm 1\}$,
\begin{align}
\sigma\E_{\XN\sim f_N}\Big[N\mathsf{A}_n(v,\XN,\mu)\Big]
&\le \frac{\be}{\vep_{\be,n}}N\Er_N(f_N,\mu)+\frac{1}{\vep_{\be,n}}\pa*{\log\K_{N,\be}(\mu)+\log C_{\mathrm{em}}(\be,\mu,v,n)}\nn\\
&\le \frac{\be}{\vep_{\be,n}}N\Er_N(f_N,\mu)+\frac{1}{\vep_{\be,n}}\pa*{\log C_{\mathrm{pf}}(\be,\mu)+\log C_{\mathrm{em}}(\be,\mu,v,n)}.
\end{align}
Taking the maximum over $\sigma\in\{\pm 1\}$ yields \eqref{eq:entropicCommutatorEstimate}.  The constants are exactly those displayed in \eqref{eq:CnvExplicit}, which completes the proof.
\end{proof}

\subsection{Thermal interpolation corollary}\label{subsec:commutator-temperature-interpolation}

The following corollary records the better of the fixed-temperature entropic estimate and the deterministic pointwise commutator estimate.  The latter becomes useful when the inverse temperature grows with $N$.

\begingroup
\begin{cor}\label{cor:entropicCommutatorTemperatureInterpolation}
Assume the hypotheses of \cref{thm:entropicCommutator} and moreover that the sharp pointwise commutator bound holds with the scale $a_{N,\s}$ from \eqref{eq:temperatureInterpolationScale}: there are constants $C_{\mathrm{pw}}(\mu,v,n),C_{\mathrm{pt}}(\mu)<\infty$ such that, for every $N\ge1$ and every configuration $\XN$,
\begin{align}\label{eq:pointwiseCommutatorBound}
\big|\mathsf A_n(v,\XN,\mu)\big|
\le C_{\mathrm{pw}}(\mu,v,n)
\pa*{\Fr_N(\XN,\mu)+C_{\mathrm{pt}}(\mu)\frac{a_{N,\s}}{N}}.
\end{align}
Then, for every $f_N\in\P((\R^\d)^N)$,
\begin{align}\label{eq:entropicCommutatorTemperatureMin}
\Big|\E_{\XN\sim f_N}\Big[N\mathsf A_n(v,\XN,\mu)\Big]\Big|
\le
\min\Big\{
&C_{\be,n,v}N\Er_N(f_N,\mu)+C_{\be,\mu,v,n},\nn\\
&C_{\mathrm{pw}}(\mu,v,n)N\Er_N(f_N,\mu)
+C_{\mathrm{pw}}'(\mu,v,n)a_{N,\s}
\Big\},
\end{align}
where one may take $C_{\mathrm{pw}}'(\mu,v,n)=C_{\mathrm{pw}}(\mu,v,n)C_{\mathrm{pt}}(\mu)$.
\end{cor}
\endgroup

\begin{proof}[Proof of \cref{cor:entropicCommutatorTemperatureInterpolation}]
The entropic estimate is precisely \cref{thm:entropicCommutator}.  For the pointwise estimate, \eqref{eq:pointwiseCommutatorBound} gives
\begin{align}
\Big|\E_{\XN\sim f_N}\Big[N\mathsf A_n(v,\XN,\mu)\Big]\Big|
&\le \E_{\XN\sim f_N}\Big[N\big|\mathsf A_n(v,\XN,\mu)\big|\Big]\nn\\
&\le C_{\mathrm{pw}}(\mu,v,n)N\E_{\XN\sim f_N}\Big[\Fr_N(\XN,\mu)\Big]
+C_{\mathrm{pw}}'(\mu,v,n)a_{N,\s}\nn\\
&\le C_{\mathrm{pw}}(\mu,v,n)N\Er_N(f_N,\mu)
+C_{\mathrm{pw}}'(\mu,v,n)a_{N,\s},
\end{align}
using the nonnegativity of the relative entropy in the definition \eqref{eq:ENdef} of $\Er_N$.  Combining the two estimates proves \eqref{eq:entropicCommutatorTemperatureMin}.
\end{proof}

\subsection{Diffusive mean-field convergence}
\label{sec:diffusive-mean-field-convergence}

We have all the ingredients to prove \cref{thm:diffusiveMeanFieldConvergence}.  The dissipation identity for the modulated free energy is standard in the modulated-free-energy method; in the notation used here, it is the dynamical counterpart of the entropy rewriting \eqref{eq:intro-entropy-rewrite} and may be read, for instance, from \cite{RS2023lsi}.  After discarding the nonpositive relative Fisher-information term, this identity gives the abstract differential inequality \eqref{eq:abstractMFEdiss}.

Applying \cref{thm:entropicCommutator} with $n=1$, $\mu=\mu^t$, and $v=u^t$ gives, for a.e. $t$,
\begin{align}\label{eq:dynamicalCommutatorBound}
\Big|\E_{\XN\sim f_N^t}\Big[N\mathsf{A}_1(u^t,\XN,\mu^t)\Big]\Big|
\le C_{\be,1,u^t}N\Er_N^t+C_{\be,\mu^t,u^t,1},
\end{align}
where the constants are those in \cref{thm:entropicCommutator}.  Combining \eqref{eq:abstractMFEdiss} and \eqref{eq:dynamicalCommutatorBound}, we obtain
\begin{align}\label{eq:dynamicMFEGronwallInput}
\frac{\dd}{\dd t}\Er_N^t
\le C_{\be,1,u^t}\Er_N^t+\frac{C_{\be,\mu^t,u^t,1}}{N}
\end{align}
for a.e. $t$.  Multiplying by the integrating factor
$\exp(-\int_0^t C_{\be,1,u^\tau}(\mu^\tau)\,\dd\tau)$ and integrating from $0$ to $t$ gives
\begin{align}\label{eq:dynamicMFEGronwallIntegrated}
\Er_N^t
\le \exp\pa*{\int_0^t C_{\be,1,u^\tau}(\mu^\tau)\,\dd\tau}\Er_N^0
+\frac1N\int_0^t \exp\pa*{\int_\tau^t C_{\be,1,u^r}(\mu^r)\,\dd r}
C_{\be,\mu^\tau,u^\tau,1}\,\dd\tau,
\end{align}
which is exactly \eqref{eq:MFEdiffusiveGronwall} in view of \eqref{eq:MFEGronwallFactors}.  The uniform-in-time claim in \cref{rem:uniform-generation-chaos} follows from the same formula when $\sup_{t\ge0}\mathcal A^t<\infty$ and $\sup_{t\ge0}\mathcal B^t<\infty$.

\begingroup

\section{Static and dynamical fluctuations}
\label{sec:static-dynamic-clt-proofs}

This section proves \cref{prop:linearEmpiricalFluctuationCLT,thm:dynamicLinearStatisticsCLT}.  \Cref{subsec:static-linear-statistics-proof} proves the bounded-test moment-generating-function formula in \cref{prop:linearEmpiricalFluctuationCLT} by inserting a linear tilt into the modulated partition function and combining the joint linear/quadratic fluctuation limit with the determinant normalization; a uniform $L^2(\mu^{\otimes N})$ density bound and truncation then extend the weak central limit theorem to arbitrary $L^2(\mu)$ tests.  \Cref{subsec:dynamic-entropy-transfers} proves \cref{lem:dynamicEntropyTransfers}, which supplies the uniform entropy, commutator, and bounded-linear-statistic estimates needed for the dynamical argument.  Finally, \cref{subsec:dynamic-characteristic-functions} first proves the finite-$N$ adjoint-duality identity \cref{prop:finiteNAdjointDuality} and then uses it to prove the dynamical central limit theorem \cref{thm:dynamicLinearStatisticsCLT}.

\subsection{Static linear statistics}
\label{subsec:static-linear-statistics-proof}

\begin{proof}[Proof of \cref{prop:linearEmpiricalFluctuationCLT}]
\par\smallskip\noindent\textbf{Bounded tests and moment generating functions.}
Suppose first that $\phi_1,\ldots,\phi_k\in L^\infty(\mu)$.  Fix $\theta=(\theta_1,\ldots,\theta_k)\in\R^k$ and set
\begin{align}
\phi_\theta\coloneqq\sum_{i=1}^k\theta_i\bar\phi_i.
\end{align}
For $\XN=(x_1,\ldots,x_N)\sim\mu^{\otimes N}$, set
\begin{align}
U_N\coloneqq\frac1N\sum_{1\le i\ne j\le N}\tl\g(x_i,x_j).
\end{align}
Let $\mathsf Z$ be the centered isonormal Gaussian process over $L_0^2(\mu)$.  If $(\lambda_\ell,e_\ell)_{\ell\ge1}$ are the nonzero spectral data of $T_\mu$, set $G_\ell\coloneqq\mathsf Z(e_\ell)$ and
\begin{align}
U_\infty\coloneqq\sum_{\ell\ge1}\lambda_\ell(G_\ell^2-1),
\end{align}
where the series converges in $L^2$.  The general degenerate $U$-statistic limit theorem \cite[Section~5.5.2]{Serfling1980}, applied jointly with the linear statistic, gives\footnote{The cited result is stated for the quadratic degenerate $U$-statistic.  The joint formulation follows from the same orthogonal-expansion argument after adjoining the linear statistic.}
\begin{align}\label{eq:jointLinearQuadraticLimit}
(S_N(\phi_\theta),U_N)\Longrightarrow(\mathsf Z(\phi_\theta),U_\infty).
\end{align}

From \eqref{eq:FNdef},
\begin{align}
-\be N\Fr_N(\XN,\mu)
=-\frac\be2U_N+\be\int_{\R^\d} h_\mu\,\dd\mu_N-\be I_\mu.
\end{align}
In the logarithmic case, \eqref{eq:logEnergyAssumption} implies $h_\mu\in L^1(\mu)$; in the positive Riesz case, $h_\mu\in L^\infty$.  Hence the weak law of large numbers gives $\int_{\R^\d} h_\mu\,\dd\mu_N\to2I_\mu$ in probability.  Combining this with \eqref{eq:jointLinearQuadraticLimit},
\begin{align}
S_N(\phi_\theta)-\be N\Fr_N(\XN,\mu)
\Longrightarrow
\mathsf Z(\phi_\theta)-\frac\be2U_\infty+\be I_\mu.
\end{align}

To pass to exponential moments, choose $r>1$ and conjugate exponents $p,q>1$.  H\"older's inequality gives
\begin{align}
\E_{\mu^{\otimes N}}
\brak*{e^{rS_N(\phi_\theta)-r\be N\Fr_N(\XN,\mu)}}
\le
\E_{\mu^{\otimes N}}\brak*{e^{prS_N(\phi_\theta)}}^{1/p}
\K_{N,qr\be}(\mu)^{1/q}.
\end{align}
The first factor is bounded uniformly in $N$ because $\phi_\theta$ is bounded and centered, and the second is bounded by \cref{thm:partitionFunctionBound}(i).  Thus, the exponentials are uniformly integrable, and
\begin{align}
\E_{\mu^{\otimes N}}
\brak*{e^{S_N(\phi_\theta)-\be N\Fr_N(\XN,\mu)}}
\longrightarrow
e^{\be I_\mu}\E\brak*{e^{\mathsf Z(\phi_\theta)-\frac\be2U_\infty}}.
\end{align}

For a centered $f\in L^2(\mu)$, let $f_0$ be its orthogonal projection onto $\ker T_\mu$ and set $f_\ell\coloneqq\langle f,e_\ell\rangle_{L^2(\mu)}$.  Writing $\mathsf Z(f)$ and $U_\infty$ in their orthogonal-series representations and using the independence of the Gaussian coordinates,
\begin{align}\label{eq:staticGaussianTiltIdentity}
&\E\brak*{\exp\pa*{\mathsf Z(f)-\frac\be2U_\infty}}\notag\\
&\quad=
\exp\pa*{\frac12\|f_0\|_{L^2(\mu)}^2}
\prod_{\ell\ge1}e^{\be\lambda_\ell/2}
\E\brak*{\exp\pa*{f_\ell G_\ell-\frac{\be\lambda_\ell}{2}G_\ell^2}}\notag\\
&\quad=
\exp\pa*{\frac12\|f_0\|_{L^2(\mu)}^2}
\prod_{\ell\ge1}
e^{\be\lambda_\ell/2}(1+\be\lambda_\ell)^{-1/2}
\exp\pa*{\frac{f_\ell^2}{2(1+\be\lambda_\ell)}}\notag\\
&\quad=
\det_2(\mathrm{Id}+\be T_\mu)^{-1/2}
\exp\pa*{\frac12\langle f,(\mathrm{Id}+\be T_\mu)^{-1}f\rangle_{L^2(\mu)}}.
\end{align}
The product converges because $(\lambda_\ell)_\ell\in\ell^2$, $(f_\ell)_\ell\in\ell^2$, and
\begin{align}
\frac{\be\lambda}{2}-\frac12\log(1+\be\lambda)=O(\lambda^2)
\end{align}
as $\lambda\to0$.  Dividing by the partition-function limit in \cref{thm:partitionFunctionBound}(ii) and taking $f=\phi_\theta$ proves \eqref{eq:linearEmpiricalFluctuationMGF}.  Since the limiting moment generating function is finite in a neighborhood of the origin, the continuity theorem for moment generating functions gives the asserted Gaussian convergence for bounded tests (see, e.g., \cite[Section~30]{Billingsley1995}).

\begingroup

\par\smallskip\noindent\textbf{Extension to $L^2(\mu)$.}
Define the density
\begin{align}
L_N(\XN)
\coloneqq
\frac{\dd\Q_{N,\be}(\mu)}{\dd\mu^{\otimes N}}(\XN)
=
\frac{\exp\{-\be N\Fr_N(\XN,\mu)\}}{\K_{N,\be}(\mu)}.
\end{align}
By the definition of the partition function,
\begin{align}
\E_{\mu^{\otimes N}}\brak*{L_N^2}
=
\frac{\K_{N,2\be}(\mu)}{\K_{N,\be}(\mu)^2}.
\end{align}
At $\be=0$, one has $L_N\equiv1$.  For $\be>0$, \cref{thm:partitionFunctionBound}(i), applied at $2\be$, and the Jensen lower bound \eqref{eq:KlowerJensen} give the required bound.  Thus, in either case, there is a constant $C_{\be,\mu}<\infty$ such that
\begin{align}
\sup_{N\ge1}\|L_N\|_{L^2(\mu^{\otimes N})}
\le C_{\be,\mu}.
\end{align}
In the logarithmic case, only the uniform-bound part of \cref{thm:partitionFunctionBound} is used at $2\be$, so no doubled version of \eqref{eq:logRateExtraAssumptions} is required.

Fix $\theta=(\theta_1,\ldots,\theta_k)\in\R^k$ and set
\begin{align}
f\coloneqq\sum_{i=1}^k\theta_i\phi_i,
\qquad
f^{(M)}\coloneqq(-M)\vee(f\wedge M).
\end{align}
Since $S_N$ centers its argument,
\begin{align}
S_N(f)-S_N(f^{(M)})=S_N(f-f^{(M)}).
\end{align}
Cauchy--Schwarz, the preceding density bound, and the exact product-law variance identity give
\begin{align}
&\sup_{N\ge1}
\E_{\XN\sim\Q_{N,\be}(\mu)}
\brak*{|S_N(f)-S_N(f^{(M)})|}\notag\\
&\quad\le
C_{\be,\mu}
\sup_{N\ge1}
\E_{\mu^{\otimes N}}
\brak*{|S_N(f-f^{(M)})|^2}^{1/2}\notag\\
&\quad=
C_{\be,\mu}
\operatorname{Var}_\mu(f-f^{(M)})^{1/2}
\le
C_{\be,\mu}\|f-f^{(M)}\|_{L^2(\mu)}
\longrightarrow0
\end{align}
as $M\to\infty$.

For each fixed $M$, the bounded-test Gaussian convergence just proved gives
\begin{align}
\E_{\XN\sim\Q_{N,\be}(\mu)}
\brak*{e^{\iu S_N(f^{(M)})}}
\longrightarrow
\exp\pa*{-\frac12\Sigma_{\be,\mu}(f^{(M)},f^{(M)})}.
\end{align}
Moreover, using $|e^{\iu u}-e^{\iu v}|\le|u-v|$ for $u,v\in\R$,
\begin{align}
&\left|
\E_{\XN\sim\Q_{N,\be}(\mu)}\brak*{e^{\iu S_N(f)}}
-
\E_{\XN\sim\Q_{N,\be}(\mu)}\brak*{e^{\iu S_N(f^{(M)})}}
\right|\le
\E_{\XN\sim\Q_{N,\be}(\mu)}
\brak*{|S_N(f)-S_N(f^{(M)})|}.
\end{align}
Since $f^{(M)}\to f$ in $L^2(\mu)$, the continuity of $\Sigma_{\be,\mu}$ yields
\begin{align}
\Sigma_{\be,\mu}(f^{(M)},f^{(M)})
\longrightarrow
\Sigma_{\be,\mu}(f,f).
\end{align}
First letting $N\to\infty$ and then $M\to\infty$ proves \eqref{eq:linearEmpiricalFluctuationCF} for the chosen $\theta$.  Since $\theta\in\R^k$ was arbitrary, L\'evy's continuity theorem gives the joint Gaussian convergence (see, e.g., \cite[Theorem~3.3.17]{Durrett2019}).
\endgroup
\end{proof}

\subsection{Entropy transfers}
\label{subsec:dynamic-entropy-transfers}

Before deriving the finite-$N$ adjoint-duality identity, we collect three uniform estimates used to control the dynamical error terms: a relative-entropy bound, an $L^1$ estimate for the transport commutator, and an $L^1$ estimate for bounded linear statistics of the empirical measure.

\begin{lemma}[Finite-time entropy transfers]\label{lem:dynamicEntropyTransfers}
Assume the hypotheses of \cref{thm:dynamicLinearStatisticsCLT}.  Then:
\begin{enumerate}[(i)]
\item
\begin{align}\label{eq:dynamicUniformEntropy}
\sup_{N\ge1}\sup_{0\le t\le T}
\Hr\pa*{f_N^t\vert\Q_{N,\be}(\mu^t)}<\infty.
\end{align}
\item For each terminal pair $(t,\phi)$, with $f^r=\Uprop^{r,t}\phi$,
\begin{align}\label{eq:dynamicCommutatorL1}
\sup_{0\le r\le t}
\E_{\XN\sim f_N^r}
\brak*{|\mathsf A_1(\nabla f^r,\XN,\mu^r)|}
\le \frac{C_{t,\phi,T}}{N}.
\end{align}
\item If $(q^r)_{0\le r\le T}$ is a deterministic family of bounded measurable functions with $\sup_r\osc(q^r)<\infty$, then
\begin{align}\label{eq:dynamicLinearL1}
\sup_{0\le r\le T}
\E_{\XN\sim f_N^r}
\brak*{\left|\int_{\R^\d} q^r\,\dd(\mu_N^r-\mu^r)\right|}
\le \frac{C_{q,T}}{\sqrt N}.
\end{align}
\end{enumerate}
\end{lemma}

\begin{proof}
Since $f_N^0=\Q_{N,\be}(\mu^0)$, \eqref{eq:intro-entropy-rewrite} gives
\begin{align}
\Er_N(f_N^0,\mu^0)=-\frac1{\be N}\log\K_{N,\be}(\mu^0).
\end{align}
Combining \eqref{eq:MFEdiffusiveGronwall} with \eqref{eq:intro-entropy-rewrite},
\begin{align}
\Hr\pa*{f_N^t\vert\Q_{N,\be}(\mu^t)}
\le
-\mathcal A^t\log\K_{N,\be}(\mu^0)
+\be\mathcal B^t
+\log\K_{N,\be}(\mu^t),
\end{align}
which is uniformly bounded by hypotheses~(i) and~(ii) of \cref{thm:dynamicLinearStatisticsCLT}.  This proves (i).

For (ii), fix a terminal pair $(t,\phi)$ and set $f^r=\Uprop^{r,t}\phi$.  By hypotheses~(i) and~(iii),
\begin{align}\label{eq:dynamicCommutatorScalarControls}
M_T&\coloneqq\sup_{0\le r\le T}\|\mu^r\|_{L^\infty}<\infty,
&
\Lambda_{t,\phi}&\coloneqq\sup_{0\le r\le t}\|\nabla^{\otimes2}f^r\|_{L^\infty}<\infty,\notag\\
\text{when }\s=0,\qquad
J_T&\coloneqq\max\left\{0,-\inf_{0\le r\le T}I_{\mu^r}\right\}<\infty.
\end{align}
If $\Lambda_{t,\phi}=0$, then $\mathsf A_1(\nabla f^r,\XN,\mu^r)=0$ for every $r$.  Otherwise, the locally bounded dependence recorded in \cref{rem:entropicCommutatorConstants}, together with hypothesis~(ii), shows that \cref{thm:entropicCommutator} supplies constants $\vep_{t,\phi}>0$ and $C_{t,\phi}<\infty$, independent of $r$ and $N$, for both signs in
\begin{align}\label{eq:dynamicUniformCommutatorBounds}
\sup_{0\le r\le t}\sup_{N\ge1}\sup_{\sigma\in\{-1,1\}}
\E_{\XN\sim\Q_{N,\be}(\mu^r)}
\brak*{\exp\pa*{\sigma\vep_{t,\phi}N
\mathsf A_1(\nabla f^r,\XN,\mu^r)}}
\le C_{t,\phi}.
\end{align}
Consequently, $e^{|z|}\le e^z+e^{-z}$ gives a uniform exponential moment of $N|\mathsf A_1(\nabla f^r,\XN,\mu^r)|$ under the time-$r$ modulated Gibbs measure $\Q_{N,\be}(\mu^r)$.  The Donsker--Varadhan lemma, together with \eqref{eq:dynamicUniformEntropy}, now yields \eqref{eq:dynamicCommutatorL1}.

For (iii), put
\begin{align}
S_N^r(q^r)\coloneqq\sqrt N\int_{\R^\d} q^r\,\dd(\mu_N^r-\mu^r).
\end{align}
For every fixed $\lambda>0$, Cauchy--Schwarz and the definition of $\Q_{N,\be}(\mu^r)$ give
\begin{align}
\E_{\Q_{N,\be}(\mu^r)}\brak*{e^{\lambda|S_N^r(q^r)|}}
\le
\K_{N,\be}(\mu^r)^{-1}
\K_{N,2\be}(\mu^r)^{1/2}
\E_{(\mu^r)^{\otimes N}}\brak*{e^{2\lambda|S_N^r(q^r)|}}^{1/2}.
\end{align}
By Hoeffding's lemma (see, e.g., \cite[Lemma~2.2]{BoucheronLugosiMassart2013}) and $e^{|z|}\le e^z+e^{-z}$,
\begin{align}
\E_{(\mu^r)^{\otimes N}}\brak*{e^{2\lambda|S_N^r(q^r)|}}
\le
2\exp\pa*{\frac{\lambda^2}{2}\osc(q^r)^2}.
\end{align}
The uniform oscillation bound and the partition-function bounds therefore make the preceding exponential moment uniform in $N$ and $r$.  A second application of the Donsker--Varadhan lemma gives
\begin{align}
\sup_{0\le r\le T}\E_{\XN\sim f_N^r}\brak*{|S_N^r(q^r)|}\le C_{q,T},
\end{align}
which is \eqref{eq:dynamicLinearL1} after division by $\sqrt N$.
\end{proof}

\subsection{Adjoint duality and characteristic functions}
\label{subsec:dynamic-characteristic-functions}

The next proposition gives an exact finite-$N$ decomposition of a terminal fluctuation observable into its backward-transported initial value, a continuous martingale, and the transport-commutator remainder controlled by \cref{lem:dynamicEntropyTransfers}.

\begin{prop}[Finite-$N$ adjoint-duality identity]\label{prop:finiteNAdjointDuality}
For $t\in[0,T]$ and $\phi\in\mathscr D_T$, let $f^r=\Uprop^{r,t}\phi$ denote the classical solution of the backward adjoint problem \eqref{eq:backwardAdjointEquation}.  Then
\begin{align}\label{eq:finiteNAdjointDuality}
S_N^t(\phi)
=S_N^0(\Uprop^{0,t}\phi)
+\sqrt N\,M_N^t[f]
-\frac{\sqrt N}{2}\int_0^t
\mathsf A_1(\nabla f^r,\XN^r,\mu^r)\,\dd r,
\end{align}
where
\begin{align}\label{eq:dynamicRawMartingale}
M_N^\tau[f]
\coloneqq
\frac1N\sqrt{\frac2\be}\sum_{i=1}^N
\int_0^\tau\nabla f^r(x_i^r)\cdot \dd B_i^r,
\qquad 0\le\tau\le t.
\end{align}
By the boundedness of $\nabla f$, $(M_N^\tau[f])_{0\le\tau\le t}$ is a continuous square-integrable martingale.  The last term in \eqref{eq:finiteNAdjointDuality} converges to zero in $L^1$ at rate $O(N^{-1/2})$.
\end{prop}

\begin{proof}
For a fixed $C_b^2$ test $\varphi$, It\^o's formula in \eqref{eq:diffusiveNParticleSystem} gives
\begin{align}\label{eq:dynamicItoEmpirical}
\dd\!\int_{\R^\d}\varphi\,\dd\mu_N^r
={}&-\frac1{N^2}\sum_{i\ne j}
\nabla\varphi(x_i^r)\cdot\nabla\g(x_i^r-x_j^r)\,\dd r
+\frac1\be\int_{\R^\d}\Delta\varphi\,\dd\mu_N^r\,\dd r\\
&+\sqrt{\frac2\be}\frac1N\sum_{i=1}^N
\nabla\varphi(x_i^r)\cdot \dd B_i^r.
\end{align}
Since $\g$ is even and $\nabla\g$ is odd,
\begin{align}\label{eq:dynamicSymmetrization}
\frac1{N^2}\sum_{i\ne j}
\nabla\varphi(x_i)\cdot\nabla\g(x_i-x_j)
={}&\frac12\int_{(\R^\d)^2\setminus\triangle}
\nabla\g(x-y)\cdot\pa*{\nabla\varphi(x)-\nabla\varphi(y)}
\,\dd\mu_N^{\otimes2}(x,y).
\end{align}
Testing the mean-field equation \eqref{eq:diffusiveMeanFieldEquation} against $\varphi$, integrating by parts, and symmetrizing the interaction term give
\begin{align}\label{eq:dynamicMeanFieldWeakIdentity}
\frac{\dd}{\dd r}\int_{\R^\d}\varphi\,\dd\mu^r
={}&-\frac12\int_{(\R^\d)^2}
\nabla\g(x-y)\cdot\pa*{\nabla\varphi(x)-\nabla\varphi(y)}
\,\dd\mu^r(x)\,\dd\mu^r(y)
+\frac1\be\int_{\R^\d}\Delta\varphi\,\dd\mu^r.
\end{align}
Because $\mu^r$ has a bounded density, it has no atoms, and the difference-of-gradients kernel is locally integrable.  Hence
\begin{align}\label{eq:dynamicQuadraticExpansion}
&\int_{(\R^\d)^2\setminus\triangle}
\nabla\g(x-y)\cdot\pa*{\nabla\varphi(x)-\nabla\varphi(y)}
\,\dd\pa*{(\mu_N^r)^{\otimes2}-(\mu^r)^{\otimes2}}(x,y)\\
&\qquad=
2\int_{(\R^\d)^2}
\nabla\g(x-y)\cdot\pa*{\nabla\varphi(x)-\nabla\varphi(y)}
\,\dd(\mu_N^r-\mu^r)(x)\,\dd\mu^r(y)
+\mathsf A_1(\nabla\varphi,\XN^r,\mu^r).
\end{align}
Subtracting \eqref{eq:dynamicMeanFieldWeakIdentity} from \eqref{eq:dynamicItoEmpirical}, multiplying by $\sqrt N$, and using \eqref{eq:adjointLinearizedOperator}, we obtain
\begin{align}\label{eq:fixedTestFluctuationEquation}
\dd S_N^r(\varphi)
=S_N^r(\mathcal L_{\be,\mu^r}\varphi)\,\dd r
-\frac{\sqrt N}{2}\mathsf A_1(\nabla\varphi,\XN^r,\mu^r)\,\dd r
+\sqrt{\frac2{\be N}}\sum_{i=1}^N\nabla\varphi(x_i^r)\cdot \dd B_i^r.
\end{align}
For the time-dependent test $f^r$, the additional drift $S_N^r(\partial_r f^r)\,\dd r$ cancels the linear term by \eqref{eq:backwardAdjointEquation}.  Integrating this identity with respect to $r$ over $[0,t]$, using $f^t=\phi$ and the backward adjoint equation, proves \eqref{eq:finiteNAdjointDuality}.  Finally, \eqref{eq:dynamicCommutatorL1} and Fubini's theorem give
\begin{align}
\E\brak*{\left|\frac{\sqrt N}{2}\int_0^t
\mathsf A_1(\nabla f^r,\XN^r,\mu^r)\,\dd r\right|}
\le \frac{tC_{t,\phi,T}}{2\sqrt N},
\end{align}
which proves the stated remainder estimate.
\end{proof}

\begin{proof}[Proof of \cref{thm:dynamicLinearStatisticsCLT}]
After relabeling, assume $0\le t_1\le\cdots\le t_k\le T$.  Fix $\theta=(\theta_1,\ldots,\theta_k)\in\R^k$, set $f_a^r=\Uprop^{r,t_a}\phi_a$, and extend $\nabla f_a^r$ by zero for $r>t_a$.  Define
\begin{align}
\phi_\theta&\coloneqq\sum_{a=1}^k\theta_af_a^0,\\
v_\theta^r&\coloneqq\sum_{a:\,r\le t_a}\theta_a\nabla f_a^r,\\
M_{N,\theta}^r&\coloneqq
\sqrt{\frac2{\be N}}\sum_{i=1}^N
\int_0^r v_\theta^u(x_i^u)\cdot \dd B_i^u.
\end{align}
The process $(M_{N,\theta}^r)_{0\le r\le t_k}$ is a continuous square-integrable martingale.  By \cref{prop:finiteNAdjointDuality},
\begin{align}\label{eq:dynamicJointRepresentation}
\sum_{a=1}^k\theta_aS_N^{t_a}(\phi_a)
=S_N^0(\phi_\theta)+M_{N,\theta}^{t_k}+R_{N,\theta},
\qquad
\E\brak*{|R_{N,\theta}|}\longrightarrow0.
\end{align}
Its predictable quadratic variation is
\begin{align}
\dd\langle M_{N,\theta}\rangle_r
=\frac2\be\int_{\R^\d}|v_\theta^r|^2\,\dd\mu_N^r\,\dd r.
\end{align}
By the adjoint regularity, $(|v_\theta^r|^2)_{0\le r\le t_k}$ is a deterministic family satisfying
\begin{align}
\sup_{0\le r\le t_k}\big\||v_\theta^r|^2\big\|_{L^\infty}<\infty.
\end{align}
Applying \eqref{eq:dynamicLinearL1} with $q^r=|v_\theta^r|^2$ gives
\begin{align}\label{eq:dynamicBracketReplacement}
\int_0^{t_k}
\E\brak*{\left|\int_{\R^\d}|v_\theta^r|^2\,\dd(\mu_N^r-\mu^r)\right|}\,\dd r
=O(N^{-1/2}).
\end{align}

Set
\begin{align}
\chi_N(r)
\coloneqq
\E\brak*{\exp\pa*{\iu\pa*{S_N^0(\phi_\theta)+M_{N,\theta}^r}}},
\qquad 0\le r\le t_k.
\end{align}
It\^o's formula gives the exact identity
\begin{align}
\chi_N(r)
=\chi_N(0)
-\frac1\be\int_0^r
\E\brak*{e^{\iu(S_N^0(\phi_\theta)+M_{N,\theta}^u)}
\int_{\R^\d}|v_\theta^u|^2\,\dd\mu_N^u}\,\dd u.
\end{align}
The complex exponential has modulus one, so
\begin{align}\label{eq:dynamicUnitModulusEstimate}
&\left|
\E\brak*{e^{\iu(S_N^0(\phi_\theta)+M_{N,\theta}^r)}
\int_{\R^\d}|v_\theta^r|^2\,\dd(\mu_N^r-\mu^r)}
\right|\le
\E\brak*{\left|\int_{\R^\d}|v_\theta^r|^2\,\dd(\mu_N^r-\mu^r)\right|}.
\end{align}
Using \eqref{eq:dynamicBracketReplacement} and \eqref{eq:dynamicUnitModulusEstimate},
\begin{align}
\sup_{0\le r\le t_k}
\left|\chi_N(r)-\chi_N(0)
+\frac1\be\int_0^r\pa*{\int_{\R^\d}|v_\theta^u|^2\,\dd\mu^u}\chi_N(u)\,\dd u\right|
\longrightarrow0.
\end{align}
Comparison with the corresponding scalar Volterra equation and Gronwall's inequality yield
\begin{align}\label{eq:dynamicCharacteristicFactorization}
\chi_N(t_k)
-\chi_N(0)
\exp\pa*{-\frac1\be\int_0^{t_k}\int_{\R^\d}|v_\theta^r|^2\,\dd\mu^r\,\dd r}
\longrightarrow0.
\end{align}
By \cref{prop:linearEmpiricalFluctuationCLT},
\begin{align}
\chi_N(0)
\longrightarrow
\exp\pa*{-\frac12\Sigma_{\be,\mu^0}(\phi_\theta,\phi_\theta)}.
\end{align}
Moreover,
\begin{align}\label{eq:dynamicTerminalCharacteristicComparison}
&\left|
\E\brak*{\exp\pa*{\iu\sum_{a=1}^k\theta_aS_N^{t_a}(\phi_a)}}
-\chi_N(t_k)
\right|
\le \E\brak*{|R_{N,\theta}|}\longrightarrow0.
\end{align}
Consequently,
\begin{align}\label{eq:dynamicLimitingCharacteristicFunction}
\E\brak*{\exp\pa*{\iu\sum_{a=1}^k\theta_aS_N^{t_a}(\phi_a)}}
\longrightarrow
\exp\pa*{-\frac12\Sigma_{\be,\mu^0}(\phi_\theta,\phi_\theta)
-\frac1\be\int_0^{t_k}\int_{\R^\d}|v_\theta^r|^2\,\dd\mu^r\,\dd r}.
\end{align}
The first quadratic expression in the exponent is
\begin{align}
\Sigma_{\be,\mu^0}(\phi_\theta,\phi_\theta)
=
\sum_{a,b=1}^k\theta_a\theta_b
\Sigma_{\be,\mu^0}\pa*{\Uprop^{0,t_a}\phi_a,\Uprop^{0,t_b}\phi_b},
\end{align}
while the second satisfies
\begin{align}
\frac2\be\int_0^{t_k}\int_{\R^\d}|v_\theta^r|^2\,\dd\mu^r\,\dd r
=
\sum_{a,b=1}^k\theta_a\theta_b\frac2\be
\int_0^{t_a\wedge t_b}\int_{\R^\d}
\nabla\Uprop^{r,t_a}\phi_a\cdot\nabla\Uprop^{r,t_b}\phi_b\,\dd\mu^r\,\dd r.
\end{align}
Their sum is exactly the covariance in \eqref{eq:dynamicFluctuationCovariance}.  The Cram\'er--Wold device and L\'evy's continuity theorem now give the joint convergence (see, e.g., \cite[Theorems~3.10.6 and~3.3.17]{Durrett2019}).
\end{proof}
\endgroup

\section{Consequences and open problems}
\label{sec:further-applications-future-directions}

We close by recording several directions suggested by the preceding estimates.  The first is to go beyond the leading determinant normalization.  In the Hilbert--Schmidt regime, the present work identifies the limiting modulated partition function as the Carleman--Fredholm determinant associated with the centered interaction operator.  It remains natural to ask for sharper asymptotics, including the first correction term and whether rare near-collision configurations, with pair separations below the typical interparticle scale $N^{-1/\d}$, produce additional local contributions not encoded by the limiting Gaussian second chaos and its Carleman--Fredholm determinant.

\begingroup
A second direction is to determine the sharp ultraviolet growth beyond the Hilbert--Schmidt threshold.  \Cref{prop:partition-function-sharpness} gives a polylogarithmic lower bound at $\s=\d/2$ and a stretched-exponential-in-$\log N$ lower bound when $\d/2<\s<\d$.  The larger rates predicted under the local nondegeneracy hypothesis in \cref{rem:expectedSharpUltravioletGrowth} remain open, as do the leading constants and the question whether a renormalized limiting object governs the corresponding modulated Gibbs ensembles.
\endgroup

\begingroup
For the periodic attractive logarithmic interaction with uniform background, the phase diagram formulated in \cref{conj:periodicAttractivePhaseDiagram} leaves four principal problems: proving the sharp zero-defect inequality and determinant limit in dimensions $7\le\d\le10$; determining the exact value of $\be_{\mathrm{gm}}(11)$, the coexisting minimizers, and the partition asymptotics at the first-order transition; proving $\be_{\mathrm{gm}}(\d)=\be_{\mathrm{s}}(\d)$ together with the matching critical upper asymptotic in dimensions $\d\ge12$; and identifying the determinant and fluctuation normalizations around the nonuniform minimizing phases.
\endgroup

Another open problem is to pass from the fixed diffusive regime to growing inverse temperature.  \Cref{cor:partitionFunctionTemperatureInterpolation} gives an upper bound as $\be=\be_N$ increases, up to $\be_N\sim N^{1-\s/\d}$ for $0<\s<\d/2$, where the factor $\be_N/N$ compensates the size $N^{\s/\d}$ of the Riesz kernel at the typical spacing $N^{-1/\d}$; in the logarithmic case, the corresponding scale is $\be_N\sim N$.  The sharpness of this interpolation is not addressed here.  This regime is also connected to forthcoming work by the third author and Serfaty \cite{rosenzweigCommutatorEstimatesSteins}, which proves quantitative Gaussian CLTs for linear statistics of logarithmic and Riesz gases across the range from $\be_N=O(1)$ to the microscopic Riesz scale.

The finite-dimensional fluctuation theorem \cref{thm:dynamicLinearStatisticsCLT} leaves open the corresponding path-space tightness problem.  As explained in \cref{rem:pathSpaceCLT}, tightness in a suitable negative-regularity path space, together with trajectory regularity of the limiting Gaussian field, would upgrade the identification proved here to a path-space central limit theorem.

Finally, as described in the introduction, the estimates proved here serve as inputs for the authors' companion works {\cite{DGRSizeOfChaos2026,DelgadinoGvalaniStein2026}} on sharp size-of-chaos bounds and on quantitative fluctuations of the canonical ensemble.  Together, these applications suggest that the modulated partition-function estimates provide a flexible normalization input for fluctuation, commutator, and quantitative-chaos problems for singular mean-field Gibbs measures.

\begingroup
\appendix
\crefalias{section}{appendix}
\section{Dependence of constants}
\label{subsec:quantitative-bookkeeping}

This appendix records how the constants in the main estimates may be chosen from the quantities appearing in the proofs.  The point is not to optimize universal numerical constants, but to make clear which norms of the background measure and transport field enter the estimates.
Throughout the appendix, universal constants may depend on $\d,\s$ and, in the commutator estimates, on the order $n$, as well as on the fixed choices in the GLZ and truncation estimates.  For the reference measure $\mu$, it is convenient to introduce the notation
\begin{align}\label{eq:measureBookkeepingPackage}
\mathsf M_{\be,\mu}
\coloneqq
1+\|\mu\|_{L^\infty}+|I_\mu|
+\|\tl\g\|_{L^2(\mu^{\otimes2})}
+\|\psi_\mu\|_{L^2(\mu)}
+\E_\mu\brak[\big]{e^{-4\be\psi_\mu}},
\qquad
\psi_\mu\coloneqq h_\mu-2I_\mu.
\end{align}
In the logarithmic case, one also keeps track of
\begin{align}\label{eq:logBookkeepingPackage}
\mathsf L_{\be,\mu}
\coloneqq
1+\int_{(\R^\d)^2}|\log|x-y||\,\dd\mu^{\otimes2}(x,y)
+\int_{\R^\d}e^{-4\be h_\mu}\,\dd\mu
+\|\tl\g\|_{L^2(\mu^{\otimes2})}.
\end{align}
These are exactly the additional logarithmic inputs used for the determinant limit and its consequences.  The commutator theorem itself requires only \eqref{eq:logEnergyAssumption} in the logarithmic case.

For the attractive logarithmic case, fix any admissible constant $C_{\mathrm{mom}}$ in \eqref{eq:attMomentAssumption}.  The threshold $\be_0$ is given by \eqref{eq:attBeta0Def}, with $C_\ast$ the universal constant of the moment bound \eqref{eq:attMomentBound}.  For every $0\le\be<\be_0$ with $e^{2\be h_\mu^{\mathrm{att}}}\in L^1(\mu)$, one admissible attractive uniform-bound constant is
\begin{equation}
C_{\mathrm{ub}}^{\mathrm{att}}(\be,\mu)
=
e^{\be|I_\mu|}\,
\E_\mu\brak[\big]{e^{2\be h_\mu^{\mathrm{att}}}}^{1/2}
\pa*{\frac{2}{1-\pa[\big]{\be C_\ast C_{\mathrm{mom}}}^2}}^{1/2}.
\end{equation}
For the quantitative conclusion in \eqref{eq:attQuantitativeDeterminantRate}, the threshold $\be_{\mathrm q}$ is chosen smaller than $\be_0/2$ so that the factorial-moment estimate and the operator spectral gap hold uniformly along the logarithmic truncation interpolation.  The constants $C_{\mathrm{rate}}^{\mathrm{att}}$ and $c_{\mathrm{att}}$ may depend on $\d$, $\|\mu\|_{L^\infty}$, $C_{\mathrm{mom}}$, $\be$, the logarithmic truncation constants, and $\int_{\R^\d} e^{4\be h_\mu^{\mathrm{att}}}\,\dd\mu$.  No uniformity over a class of backgrounds specified only by $\|\mu\|_{L^\infty}$ is asserted.

For $0<\s<\d/2$ and fixed $\be>0$, an admissible uniform partition-function constant can be obtained as follows.  Choose $K_{\be,\mu}\ge1$ satisfying
\begin{align}\label{eq:bookkeepingKbeta}
\frac{1}{\be^2 C_\mu K_{\be,\mu}^{2\s-\d}}\ge 4L,
\qquad
\frac{1}{\be C_\mu K_{\be,\mu}^{-(\d-\s)}}\ge 4L,
\end{align}
where $C_\mu$ denotes the constants in \eqref{eq:CtailFourier}--\eqref{eq:DtailFourier}.  With the layer-cake threshold
\begin{align}\label{eq:bookkeepingUbeta}
U_{\be,\mu}=4C\be K_{\be,\mu}^{\s}+1,
\end{align}
one may take $C_{\mathrm{ub}}(\be,\mu)$ to dominate the corresponding layer-cake bound
\begin{align}\label{eq:bookkeepingCubRiesz}
e^{\be\|h_\mu\|_{L^\infty}}\pa*{
e^{U_{\be,\mu}}+L\int_{U_{\be,\mu}}^\infty e^{-u}\,\dd u
+\max_{1\le N<N_{\be,\mu}}\K_{N,\be}(\mu)}
\end{align}
after increasing $N_{\be,\mu}$ so that \eqref{eq:allTermsBeatNoSmallBeta} holds.  In the logarithmic case, the same recipe uses the cutoff in \eqref{eq:logChooseK}, the finite-$T$ cancellation estimate \eqref{eq:logLargeScaleCancellation}, and the threshold in \eqref{eq:logLayerThresholdPartition}; the resulting constant is a function of $\be$ and the logarithmic package \eqref{eq:logBookkeepingPackage}.

For the convergence rate, let $C_{\mathrm{det}}(\be,\mu)$ denote any constant for the determinant and truncation error in \eqref{eq:rieszLaplaceLogFallback}.  Then the one-body reduction \eqref{eq:oneBodyReductionRate} shows that $C_{\mathrm{rate}}(\be,\mu)$ may be chosen to dominate
\begin{align}\label{eq:bookkeepingCrate}
C_{\mathrm{det}}(\be,\mu)
+\be\,\|\psi_\mu\|_{L^2(\mu)}
\pa[\Big]{C_{\mathrm{ub}}(2\be,\mu)^{1/2}
+C_{\mathrm{ub}}(4\be,\mu)^{1/4}\,\E_\mu\brak[\big]{e^{-4\be\psi_\mu}}^{1/4}}.
\end{align}
Thus, the only new one-body information needed after the determinant estimate is the $L^2(\mu)$ size and the exponential moment $\E_\mu[e^{-4\be\psi_\mu}]$ of $\psi_\mu$.

For the quantitative divergence estimate in \cref{prop:partition-function-sharpness}, the constants may depend on $\be$, $\mu$, $\d$, and $\s$, as well as on the density-point construction in \cref{lem:quantitativeDyadicCompression}.  More precisely, that construction selects a positive-density level, a density point, a local dyadic cube, and a sufficiently fine scale; the constants $c_0,C_0,L_0$ and the large-$N$ threshold inherit these choices.  The present proof therefore does not assert uniformity over a class of background laws specified only through a common $L^\infty$ bound.

\begin{remark}[Explicit constants]\label{rem:entropicCommutatorConstants}
If $\|\nab v\|_{L^\infty}=0$, then $\mathsf{A}_n(v,\XN,\mu)=0$ and \cref{thm:entropicCommutator} holds trivially with $C_{\mathrm{em}}=1$ and any $\vep_{\be,n}>0$.  If $\|\nab v\|_{L^\infty}>0$, the threshold $\be_{\mathrm{loc}}$ is determined by the constants in the local numerator estimate of \cref{lem:localNumeratorComm}, and the proof yields
\begin{align}\label{eq:epsnStatement}
\vep_{\be,n}=
\begin{cases}
\displaystyle
\frac{\be}{4C_{\mathrm{sr},n}\|\nab v\|_{L^\infty}^n},
&0<\be\le\be_{\mathrm{loc}},\\[1em]
\displaystyle
\frac{\be_{\mathrm{loc}}}{8C_{\mathrm{sr},n}\|\nab v\|_{L^\infty}^n},
&\be>\be_{\mathrm{loc}},
\end{cases}
\end{align}
where $C_{\mathrm{sr},n}$ is the short-range constant in \eqref{eq:detCommShortRange}, depending only on $\d,\s,n$, and
\begin{align}\label{eq:CemLargeBetaDef}
C_{\mathrm{em}}(\be,\mu,v,n)
=\begin{cases}
e+L,
&0<\be\le\be_{\mathrm{loc}},\ 0<\s<\d/2,\\[0.5em]
e^{-\be I_\mu}\pa*{e+L},
&0<\be\le\be_{\mathrm{loc}},\ \s=0,\\[0.5em]
e^{-\be I_\mu}\pa*{e+L}^{1/2}
C_{\mathrm{ub}}(2\be-\be_{\mathrm{loc}},\mu)^{1/2},
&\be>\be_{\mathrm{loc}},
\end{cases}
\end{align}
where $C_{\mathrm{ub}}$ is the uniform partition-function bound from \eqref{eq:CubDef}; see \cref{subsec:quantitative-bookkeeping} for the associated dependency bookkeeping.

The dependence is locally bounded in explicit scalar controls.  More precisely, fix $M,\Lambda<\infty$ and, when $\s=0$, $J<\infty$.  For fixed $\d,\s,n$, the proof of \cref{lem:localNumeratorComm} may be run uniformly over all backgrounds and transports satisfying
\begin{align}\label{eq:commutatorFamilyScalarControls}
\|\mu\|_{L^\infty}\le M,
\qquad
\|\nabla v\|_{L^\infty}\le\Lambda,
\qquad
I_\mu\ge-J\quad\text{if }\s=0.
\end{align}
After the trivial case $\|\nabla v\|_{L^\infty}=0$ is separated, every occurrence of the transport field in the centered commutator kernel and its one-body remainder is controlled by $|v(x)-v(y)|\le\Lambda|x-y|$; the local numerator proof may therefore be run with $\Lambda$ in place of $\|\nabla v\|_{L^\infty}$ in \eqref{eq:epsnStatement}.  In the logarithmic case, the remaining background dependence is controlled by
\begin{align}\label{eq:uniformAbsoluteLogEnergyControl}
\int_{(\R^\d)^2}|\log|x-y||\,\dd\mu^{\otimes2}(x,y)
\le
2M\int_{|z|\le1}-\log|z|\,\dd z+2J,
\end{align}
which follows by decomposing $-\log|x-y|$ into its positive and negative parts.  Consequently, $\be_{\mathrm{loc}}$ has a positive lower bound, and the local numerator constants have finite upper bounds, depending only on $\d,\s,n,M,\Lambda$ and, in the logarithmic case, $J$.

For a family satisfying \eqref{eq:commutatorFamilyScalarControls}, the exponential-moment constants under $\Q_{N,\be}(\mu)$ may therefore be chosen uniformly whenever $\K_{N,\be}(\mu)$ has a common positive lower bound and $\K_{N,2\be}(\mu)$ has a common upper bound.  Indeed, for $0\le\gamma\le\be$, log-convexity of the partition function and $\K_{N,0}(\mu)=1$ give
\begin{align}\label{eq:intermediatePartitionConvexity}
\K_{N,2\be-\gamma}(\mu)
\le
\K_{N,2\be}(\mu)^{1-\gamma/(2\be)},
\end{align}
which uniformly controls the intermediate partition-function factor in the H\"older bootstrap.
\end{remark}

For the commutator estimate, once a local threshold $\be_{\mathrm{loc}}(\mu,v,n)$ has been fixed in \cref{lem:localNumeratorComm}, the constants in \cref{thm:entropicCommutator} are explicitly
\begin{align}\label{eq:bookkeepingCommConstants}
C_{\be,n,v}=\frac{\be}{\vep_{\be,n}},
\qquad
C_{\be,\mu,v,n}
=\frac{1}{\vep_{\be,n}}
\pa*{\log C_{\mathrm{pf}}(\be,\mu)+\log C_{\mathrm{em}}(\be,\mu,v,n)},
\end{align}
with $\vep_{\be,n}$ and $C_{\mathrm{em}}$ given in \eqref{eq:epsnStatement} and \eqref{eq:CemLargeBetaDef}.  The dependence on the transport field enters through $\|\nab v\|_{L^\infty}$ and through the local numerator constants $C_{\mu,v,n}$, $M_{\mu,v,n}$, and $C_{\mathsf K,\mu,v,n}$ appearing in \cref{lem:localNumeratorComm}.  The local-boundedness statement in \cref{rem:entropicCommutatorConstants} shows that common scalar controls yield common choices of these constants and of $\be_{\mathrm{loc}}$.  For time-dependent applications, $\mathcal A^t$ and $\mathcal B^t$ in \eqref{eq:MFEGronwallFactors} are obtained by integrating precisely the resulting $C_{\be,1,u^t}(\mu^t)$ and $C_{\be,\mu^t,u^t,1}$.

\endgroup

\bibliographystyle{amsalpha}
\bibliography{DGR_PFunc_Current_20260902_145019_UTC}
\end{document}